\documentclass{amsart}

\usepackage[T1]{fontenc}
\usepackage[utf8]{inputenc} 
\usepackage[english]{babel}
\usepackage{lmodern}
\usepackage{microtype}
\usepackage[a4paper,margin=1in]{geometry}
\usepackage{amsmath,amssymb,amsfonts,mathtools}
\usepackage{amsthm}
\usepackage{graphicx}
\graphicspath{{figs/}}
\usepackage[percent]{overpic}
\usepackage[section]{placeins}
\usepackage{enumitem}
\usepackage{csquotes}
\usepackage{silence}
\usepackage[
  backend=biber,
  style=alphabetic,
  maxnames=99,
  giveninits=true,
  doi=false,
  url=false,
  eprint=true,
  isbn=false
]{biblatex}
\usepackage[colorlinks=true,allcolors=blue]{hyperref}

\allowdisplaybreaks

\renewbibmacro{in:}{}
\DeclareFieldFormat{pages}{#1}

\numberwithin{equation}{section}
\numberwithin{figure}{section}
\newtheorem{theorem}{Theorem}[section]
\newtheorem{lemma}[theorem]{Lemma}
\newtheorem{proposition}[theorem]{Proposition}
\newtheorem{corollary}[theorem]{Corollary}
\theoremstyle{definition}
\newtheorem{definition}[theorem]{Definition}
\newtheorem{convention}[theorem]{Convention}
\newtheorem{example}[theorem]{Example}
\newtheorem{question}[theorem]{Question}

\theoremstyle{remark}
\newtheorem{remark}[theorem]{Remark}

\DeclareMathOperator{\tb}{tb}
\DeclareMathOperator{\rot}{rot}
\DeclareMathOperator{\lk}{lk}

\DeclareMathOperator{\id}{id}

\newcommand{\R}{\mathbb{R}}
\newcommand{\Z}{\mathbb{Z}}
\let\oldDelta\Delta
\renewcommand{\Delta}{\oldDelta} 

\let\oldSigma\Sigma
\renewcommand{\Sigma}{\oldSigma}
\let\oldtau\tau
\renewcommand{\tau}{\oldtau}
\let\oldOmega\Omega 
\renewcommand{\Omega}{\oldOmega}

\let\oldalpha\alpha 
\renewcommand{\alpha}{\oldalpha}
\newcommand{\alphast}{\oldalpha_{\mathrm{st}}}
\let\olddxi\xi 
\renewcommand{\xi}{\olddxi}
\newcommand{\xist}{\olddxi_{\mathrm{st}}}
\newcommand{\X}{X}

\newcommand{\G}{G}

\let\oldgamma\Gamma 
\renewcommand{\Gamma}{\oldgamma}
\newcommand{\MCG}{\mathrm{MCG}}
\renewcommand{\aa}{\mathfrak{a}}
\newcommand{\bb}{\mathfrak{b}}
\newcommand{\cc}{\mathfrak{c}}
\newcommand{\dd}{\mathfrak{d}}

\title{A contact version of Kirby's theorem}

\author{Marc Kegel}
\address{Universidad de Sevilla, Dpto.\ de Álgebra,
Avda.\ Reina Mercedes s/n,
41012 Sevilla, Spain}
\email{kegelmarc87@gmail.com}

\author{Eric Stenhede}

\author{Vera Vértesi}
\address{University of Vienna,
Faculty of Mathematics,
Oskar-Morgenstern-Platz 1,
1090 Vienna, Austria}
\email{eric.stenhede@univie.ac.at}
\email{vera.vertesi@univie.ac.at}

\date{\today}

\def\subjclassname{\textup{2020} Mathematics Subject Classification}
\expandafter\let\csname subjclassname@1991\endcsname=\subjclassname


\begin{document}

\begin{abstract}
A theorem of Ding and Geiges states that every closed, connected contact $3$-manifold can be obtained from the standard tight contact $3$-sphere by contact $(\pm1)$-surgery along a Legendrian link. The literature also contains some examples of contact Kirby moves, i.e.\ explicit operations on front projections of Legendrian surgery links that change the surgery link but preserve the contactomorphism type of the surgered manifold. Among the most commonly used are cancelling pairs and contact handle slides; however, these moves alone are not sufficient to relate all contact surgery diagrams of contactomorphic contact manifolds.

In this article, we introduce two new families of contact Kirby moves, called \emph{lantern moves} and \emph{chain moves}, and use them to give a complete set of contact Kirby moves. More precisely, we show that two contact surgery diagrams represent contactomorphic contact manifolds if and only if they are related by a sequence of planar isotopies, Legendrian Reidemeister moves, insertions or removals of standard cancelling pairs, the two standard contact handle slides, the standard lantern move, and the standard chain move. All these moves are explicit diagrammatic operations in the front projection.

The proof follows an approach initiated by Avdek through his ribbon-move framework, which is rooted in the Giroux correspondence, and combines it with a presentation by Gervais of the mapping class group. We also discuss several consequences of the main theorem, illustrating the effectiveness of the contact Kirby calculus by recovering the invariance of Gompf's $d_3$-invariant purely diagrammatically and by deriving the topological Kirby theorem from contact-geometric methods.
\end{abstract}

\maketitle

\setcounter{tocdepth}{1}

\vspace{-0.6cm}
\section{Introduction}
A classical theorem due to Lickorish and Wallace~\cite{Lickorish,Wallace} states that any connected, closed, oriented $3$-manifold can be obtained by integral Dehn surgery on a link in the $3$-sphere $S^3$. Such a Dehn surgery presentation of a given \(3\)-manifold is, however, far from unique. Indeed, there are operations on Dehn surgery presentations, the so-called \emph{Kirby moves}, which change the surgery presentation but preserve the diffeomorphism type of the surgered \(3\)-manifold. These operations can be used to produce infinitely many surgery descriptions of the same \(3\)-manifold, along pairwise non-isotopic links in \(S^3\). Kirby's celebrated theorem~\cite{Kirby} gives a complete set of explicit moves relating any two Dehn surgery presentations of diffeomorphic $3$-manifolds: handle slides, blow-ups, and blow-downs.

The Lickorish--Wallace theorem admits a contact-geometric analogue. Ding and Geiges~\cite{Ding_Geiges} proved that every connected, closed, oriented contact $3$-manifold $(M,\xi)$ with positive, coorientable contact structure can be obtained from $(S^3,\xist)$ by contact $(\pm1)$-surgery along a Legendrian link. Since then, contact surgery has become a fundamental tool in $3$-dimensional contact topology. Its applications range from the study of symplectic fillings~\cite{Eliashberg_handle,Weinstein,Gompf,Gompf_Stipsicz,Ozbagci_Stipsicz} and the classification of tight contact structures~\cite{Honda_classification,Lisac_StipsiczI,Lisac_StipsiczII,Lisac_StipsiczIII,Wand} to the study of Legendrian knots~\cite{Geiges_Onaran} and Reeb dynamics~\cite{BEE,Avdek_surgery}. Apart from these structural results, contact surgery also provides a flexible and effective way to describe contact $3$-manifolds.

The literature contains several examples of contact Kirby moves\footnote{Here, a \emph{contact surgery link} $\boldsymbol{L}$ is a Legendrian link in $(\R^3,\xist)\subset (S^3,\xist)$ whose components are decorated with contact surgery coefficients $(\pm1)$. Its front projection is called a \emph{contact surgery diagram}. A \emph{contact Kirby move} is an explicit operation on contact surgery diagrams which replaces one diagram by another without changing the contactomorphism type of the surgered contact manifold. We refer to Section~\ref{sec:moves} for more discussion.}; see, for example,~\cite{Gompf,Ding_Geiges,Ding_Geiges_Stipsicz,Ding_Geiges_slides,Lisca_Stipsicz_lantern,Avdek13,Kegel_thesis,Casals_Etnyre_Kegel,Etnyre_Kegel_Onaran}. Until this work, however, it remained open whether there existed an explicit collection of contact Kirby moves sufficient to relate any two contact surgery links describing the same contact $3$-manifold.
The previously known moves are not sufficient for this purpose; see Section~\ref{sec:relation_moves} for further discussion. A decisive advance was made by Avdek~\cite{Avdek13}, who introduced \emph{ribbon moves} and proved that any two contact surgery links describing contactomorphic contact manifolds are related by a sequence of ribbon moves and Legendrian isotopies; see Section~\ref{sec:ribbon_moves_sec}. Ribbon moves, however, are not explicit operations on front projections (they are not contact Kirby moves), and they form an infinite class which does not seem to be describable by finitely many types of diagrammatic moves.

In this article, we build on Avdek's work to obtain a satisfactory diagrammatic calculus. We introduce two new families of contact Kirby moves, called \emph{chain moves} and \emph{lantern moves}, and use these to present a complete list of contact Kirby moves.

\begin{theorem}[Contact Kirby theorem]\label{thm:main}
Two contact surgery links describe contactomorphic contact $3$-manifolds if and only if their front projections are related by a sequence of 
\begin{itemize}
    \item planar isotopies,
    \item Legendrian Reidemeister moves,
    \item insertions or removals of standard cancelling pairs,
    \item standard contact handle slides,
    \item standard lantern moves, and
    \item standard chain moves.
\end{itemize}  
\end{theorem}

\begin{figure}[htbp]
    \centering
    \begin{overpic}[scale=1]{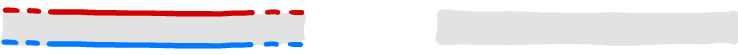}
    \put(47,2){$\leftrightarrow$}
    \put(78,2){$\boldsymbol{\emptyset}$}
    \put(-5,0){$\mp$}
    \put(-5,5){$\pm$}
    \end{overpic}
    \caption{A standard cancelling pair.}\label{fig:cancelling_intro}
\end{figure}

\begin{figure}[htbp]
    \centering
    \begin{overpic}[scale=1]{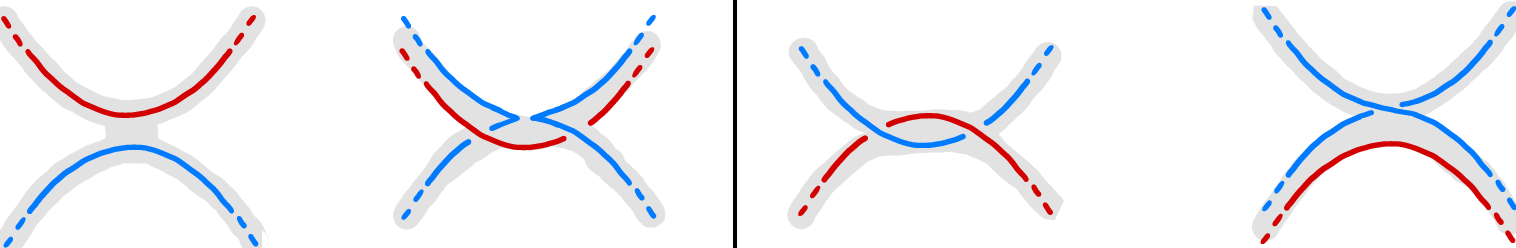}
    \put(20,8){$\leftrightarrow$}
    \put(75,8){$\leftrightarrow$}
    \put(-3,0){$-$}
    \put(-3,15){$-$}
    \put(23,2){$-$}
    \put(23,12){$-$}
    \put(50,2){$-$}
    \put(50,12){$-$}
    \put(80,15){$-$}
    \put(80,0){$-$}
    \end{overpic}
    \caption{The two standard contact handle slides needed in Theorem~\ref{thm:main}.}\label{fig:braid_intro}
\end{figure}

\begin{figure}[htbp]
    \centering
    \begin{overpic}[scale=1]{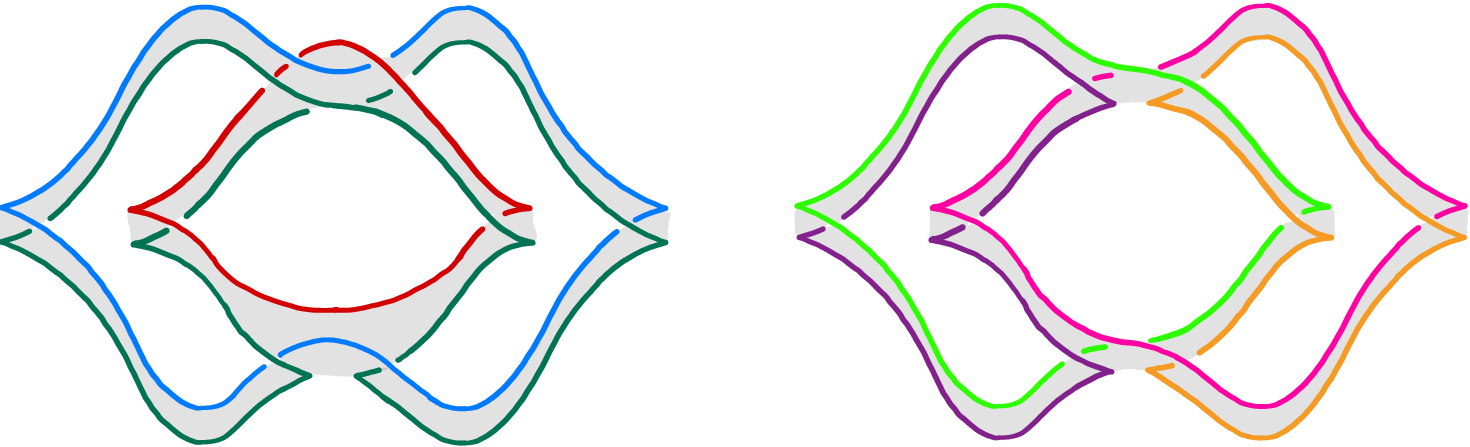}
    \put(-3,15.5){$-$}
    \put(-3,13.5){$-$}
    \put(6,15.5){$-$}
    \put(60,15.5){$-$}
    \put(60,13.5){$-$}
    \put(91,15.5){$-$}
    \put(91,13.5){$-$}
    \put(48,14.5){$\leftrightarrow$}
    \end{overpic}
    \caption{The standard lantern move.}\label{fig:lantern_intro}
\end{figure}

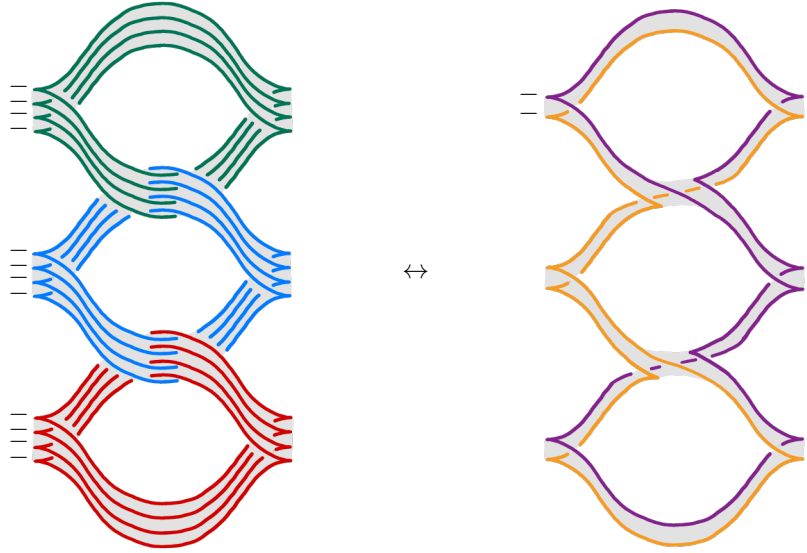
\begin{figure}[htbp]
    \centering
    \begin{overpic}[scale=1]{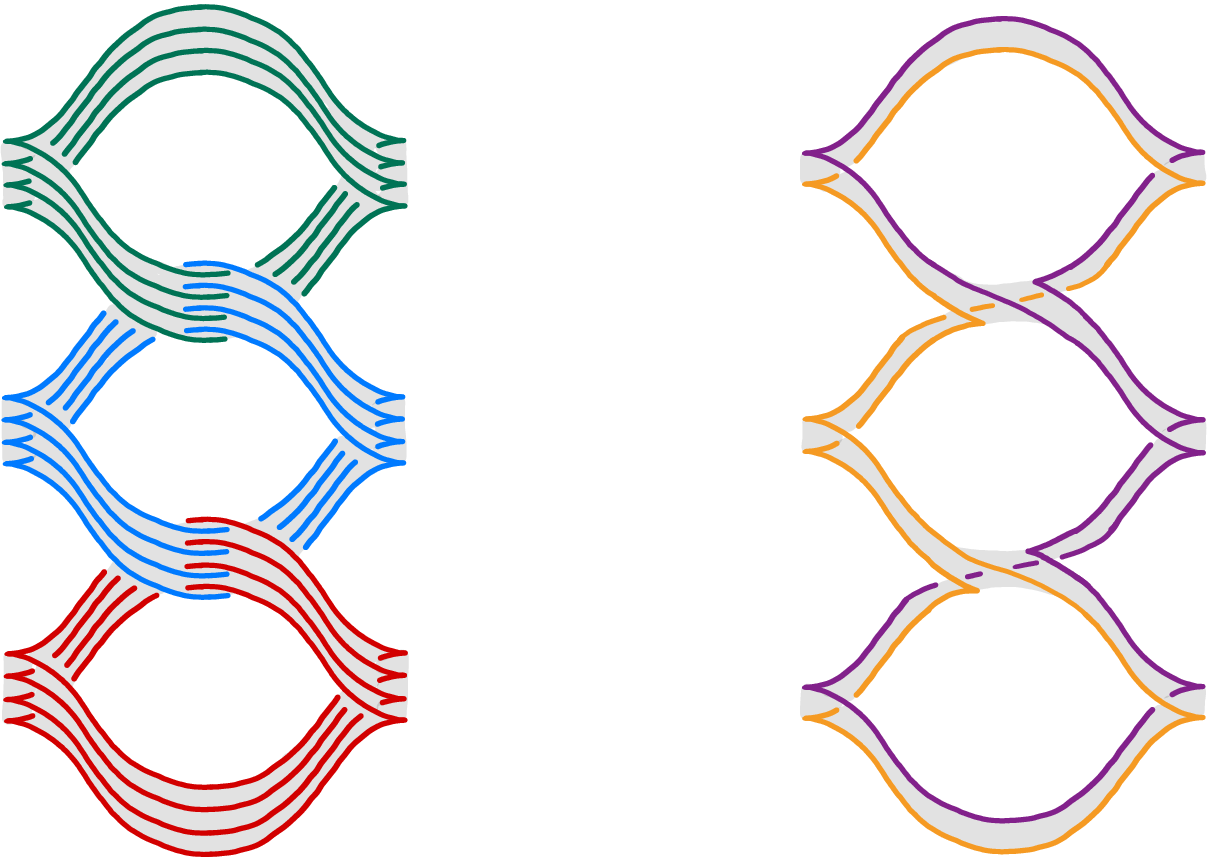}
    \put(-3,59){$-$}
    \put(-3,57.1){$-$}
    \put(-3,55.5){$-$}
    \put(-3,53.5){$-$}
    \put(-3,37.8){$-$}
    \put(-3,35.8){$-$}
    \put(-3,34.2){$-$}
    \put(-3,32.2){$-$}
    \put(-3,16.5){$-$}
    \put(-3,14.6){$-$}
    \put(-3,13){$-$}
    \put(-3,11){$-$}
    \put(48,35){$\leftrightarrow$}
    \put(63,58){$-$}
    \put(63,55.5){$-$}
    \end{overpic}
    \caption{The standard chain move.}\label{fig:chain_intro}
\end{figure}

\subsection{The standard contact Kirby moves}
We now describe the standard contact Kirby moves appearing in Theorem~\ref{thm:main}. In Section~\ref{sec:elementary_R_equivalences} we verify that these moves do not change the contactomorphism type of the surgered contact manifolds. In Figures~\ref{fig:cancelling_intro}--\ref{fig:chain_intro}, the shaded region indicates the front projection of the supporting surface or handlebody on which the move is performed. The corresponding supporting region is assumed to be disjoint from all components of the contact surgery link except those explicitly involved in the move. All components outside the supporting region are left unchanged.
\begin{itemize}
    \item \emph{Insertion or removal of a standard cancelling pair.} A standard cancelling pair consists of the front projection of two Legendrian knots which differ by vertical translation in the $\partial_z$-direction and are decorated with opposite contact surgery coefficients. Equivalently, the two knots cobound the annulus swept out by translating one of them to the other in the $\partial_z$-direction. The move consists of inserting or removing such a pair, provided that this annulus is disjoint from the remaining components of the contact surgery link. It is shown in Figure~\ref{fig:cancelling_intro}; the shaded region is the front projection of the annulus. 
    \item The two \emph{standard contact handle slides} are the two moves shown in Figure~\ref{fig:braid_intro}. Each is supported in a genus two handlebody containing precisely the two Legendrian components involved in the handle slide.
    \item The \emph{standard lantern move} is the move shown in Figure~\ref{fig:lantern_intro}. It is supported in a genus three handlebody containing precisely the Legendrian components involved in the move.
    \item The \emph{standard chain move} is the move shown in Figure~\ref{fig:chain_intro}. It is supported in a genus three handlebody containing precisely the Legendrian components involved in the move.
\end{itemize}

Cancelling pairs appeared first in~\cite{Ding_Geiges_fillable}, cf.~\cite{Ding_Geiges_Stipsicz}, contact handle slides were introduced in~\cite{Ding_Geiges_slides}, cf.~\cite{Avdek_surgery,Casals_Etnyre_Kegel}, and versions of the chain moves and lantern moves can be found in~\cite{Avdek13}. As their names suggest, these moves arise from the chain and lantern relations in the mapping class group.


\subsection{The independence of the standard contact Kirby moves}\label{sec:relation_moves}

Theorem~\ref{thm:main} gives a theoretically satisfying answer to the question of when two contact surgery diagrams describe contactomorphic contact manifolds. In practice, however, finding an explicit sequence of moves between two such diagrams can be difficult if one uses only the standard contact Kirby moves from Theorem~\ref{thm:main}. It is often useful to have more general contact Kirby moves at hand, even when these moves can ultimately be expressed as combinations of the moves from Theorem~\ref{thm:main}. For this reason, in Section~\ref{sec:moves} we collect more general versions of the lantern and chain moves together with several previously known contact Kirby moves. That section is also intended as a ready-to-use toolbox for readers who want to use or consult the moves without reading the entire article.
In the opposite direction, we demonstrate in Section~\ref{sec:proof_rel_kirbymoves} that the standard contact Kirby moves appearing in Theorem~\ref{thm:main} are all independent of each other.

\begin{theorem}[Independence of the standard contact Kirby moves]\label{thm:indep}
None of the following contact Kirby moves can be expressed as a sequence of the others, together with planar isotopies and Legendrian Reidemeister moves:
\begin{itemize}
    \item insertions or removals of standard cancelling pairs,
    \item the two standard contact handle slides,
    \item the standard lantern move, and
    \item  the standard chain move.
\end{itemize}
\end{theorem}

Nevertheless, we may ask the following informal question; we refer to Section~\ref{sec:proof_rel_kirbymoves} for more discussion.

\begin{question}
Does there exist a simpler complete set of contact Kirby moves?\footnote{Here we leave the meaning of ``simpler'' intentionally open and ambiguous. It may refer, for example, to using only finitely many moves, fewer types of contact Kirby moves, or more local moves.}
\end{question}


\subsection{Invariants of contact manifolds from surgery descriptions}\label{sec:invariants}

Theorem~\ref{thm:main} provides a general strategy for defining invariants of contact manifolds from surgery diagrams, or for giving diagrammatic proofs of well-definedness: one gives a diagrammatic formula and then checks invariance under the standard contact Kirby moves. We illustrate this strategy with Gompf's \(d_3\)-invariant. Recall that \(d_3\) is a rational-valued invariant of oriented, tangential \(2\)-plane fields with torsion first Chern class~\cite{Gompf}. There are explicit formulas for computing it from contact surgery diagrams~\cite{Gompf,Ding_Geiges_Stipsicz,Durst_Kegel,Etnyre_Kegel_Onaran}. Here we use these formulas in the opposite direction: we define a quantity associated with a contact surgery diagram and then prove directly that it is invariant under the standard contact Kirby moves from Theorem~\ref{thm:main}. Here and throughout the paper, we use the shifted normalization of the \(d_3\)-invariant for which \(d_3(S^3,\xist)=0\); see the conventions at the end of the introduction.

Let \(\boldsymbol L\) be a contact surgery link with components $L_1,\ldots,L_n$ representing a contact manifold $(M,\xi)$ with torsion first Chern class. We choose orientations on the $L_i$ and denote by $Q_{\boldsymbol L}$ the linking matrix of the underlying smooth surgery diagram, with diagonal entries given by the topological framings \(\tb(L_i)\pm1\), and let $\boldsymbol r_{\boldsymbol L}
=
(\rot(L_1),\ldots,\rot(L_n))^T$
be the vector of the rotation numbers. Let \(q(\boldsymbol L)\) denote the number of components on which contact \((+1)\)-surgery is performed. Since the first Chern class is torsion, there is a rational solution $\boldsymbol{x}$ of $Q_{\boldsymbol L}\boldsymbol x=\boldsymbol r_{\boldsymbol L}$ \cite{Gompf}. 
For any such solution, we set $c_{\boldsymbol L}^2
:=
\boldsymbol x^TQ_{\boldsymbol L}\boldsymbol x$ and define
\[
d_3^{\mathrm{surg}}(\boldsymbol L)
:=
\frac14\left(
c_{\boldsymbol L}^2
-3\sigma(Q_{\boldsymbol L})
-2n
\right)
+
q(\boldsymbol L).
\]

\begin{proposition}[$d_3$-invariant via surgery diagrams]\label{prop:d3-invariance}
Let $\boldsymbol{L}$ be a contact surgery description of a contact manifold $(M,\xi)$ with torsion first Chern class.  
Then the rational number $d_3^{\mathrm{surg}}(\boldsymbol L)$ is unchanged under the standard contact Kirby moves from Theorem~\ref{thm:main} and therefore depends only on the contactomorphism type of $(M,\xi)$. In particular, $d_3^{\mathrm{surg}}(\boldsymbol L)$ is independent of the choice of orientation of $\boldsymbol{L}$ and of the solution \(\boldsymbol x\). 
\end{proposition}

The proof is given in Section~\ref{sec:diagrammatic_d3}. There we also obtain a diagrammatic mod \(8\) congruence for the quantity \(\sigma(Q_{\boldsymbol L})-c_{\boldsymbol L}^2\), which can be viewed as a shadow of the Spin\(^c\) congruence \(c_1^2\equiv\sigma\pmod 8\); we refer to Section~\ref{sec:mod8inv} for details.

In general, even more interesting are invariants of contact structures that do not depend only on the underlying tangential $2$-plane field. Such invariants are usually defined using holomorphic curves. Examples include contact homology~\cite{Contact_homology}, symplectic homology~\cite{symplectic_homology}, embedded contact homology~\cite{ECH}, and the contact class in Heegaard Floer homology~\cite{contact_class}. It would be very interesting to find combinatorial descriptions of such invariants in terms of contact surgery links and then use Theorem~\ref{thm:main} to give a combinatorial proof --- still relying on the Giroux correspondence --- of their well-definedness. Combinatorial formulations of analytic invariants can already be found, for example, in~\cite{Avdek13,BEE,contact_class_algo,Chekanov}.

\subsection{Rational contact surgery}
Although the theorem of Ding and Geiges shows that contact $(\pm1)$-surgery is sufficient to describe all closed contact $3$-manifolds, in concrete situations it is often more convenient to allow more general rational contact surgery coefficients. Recall that if the contact surgery coefficient is of the form $1/n$, with $n\in\mathbb{Z}$, then there is a unique tight contact structure on the glued-in solid torus that extends the contact structure on the knot exterior. 
For a general contact surgery coefficient $r\in\mathbb Q\setminus\{0\}$, there are only finitely many tight contact structures on the glued-in solid torus extending the contact structure on the exterior. Thus, for general contact surgery coefficients, the surgery data must specify which of these contact structures is used. 
The \textit{transformation lemma} and the \textit{replacement lemma} allow one to convert any contact surgery diagram with rational coefficients into a contact $(\pm1)$-surgery diagram. We refer to~\cite{Honda_classification,Ding_Geiges_fillable,Ding_Geiges_Stipsicz,Kegel_thesis,Etnyre_Kegel_Onaran} for further background and details. Thus Theorem~\ref{thm:main} directly implies the following Kirby theorem for rational contact surgery diagrams.

\begin{theorem}[Contact Kirby theorem for rational contact surgery]\label{prop:rational_contact_Kirby}
Two rational contact surgery links describe contactomorphic contact $3$-manifolds if and only if their front projections are related by a sequence of 
\begin{itemize}
    \item planar isotopies,
    \item Legendrian Reidemeister moves,
    \item  insertions or removals of standard cancelling pairs,
    \item standard contact handle slides,
    \item standard lantern moves,
    \item standard chain moves,
    \item replacement moves, and
    \item transformation moves.\qed
\end{itemize}      
\end{theorem}

\subsection{Legendrian knots in surgery diagrams}

Let \(J\) be a Legendrian link in a contact \(3\)-manifold \((M,\xi)\), and let \(\boldsymbol L\) be a contact surgery link for \((M,\xi)\). Then \(J\) can be represented by a Legendrian link in the exterior of $\boldsymbol{L}$; see, for example,~\cite[Section~4.7]{Kegel_thesis}. Thus, one can describe $J$ in \((M,\xi)\) by the front projection of $J\cup \boldsymbol L$, where the components of \(\boldsymbol L\) carry contact surgery coefficients and the components of \(J\) are unframed. We say that two Legendrian links $J_i$ in $(M_i,\xi_i)$, for \(i=1,2\), are \emph{equivalent} if there is a contactomorphism $(M_1,\xi_1)\longrightarrow (M_2,\xi_2)$ sending \(J_1\) to \(J_2\).

Recall that from a three-dimensional perspective a contact handle slide of a Legendrian knot $J$ over a knot $L$ with contact surgery coefficient $(\pm1)$ can be seen as a Legendrian isotopy of $J$ over the newly glued-in meridional disk in the surgered manifold along $L$. In particular, this Legendrian isotopy is also valid if $J$ is unframed. We refer to \cite{Ding_Geiges_slides,Casals_Etnyre_Kegel} and Section~\ref{sec:ctchadleslidegeneral} for more discussion. The proof of the following reult can be found in Section~\ref{sec:proofLeg_links_Kirby_thm}

\begin{theorem}[Contact Kirby theorem for Legendrian links in surgery diagrams]\label{prop:Leg_links_Kirby_thm}
For $i=1,2$, let \(\boldsymbol L_i\) be contact surgery links for $(M_i,\xi_i)$ and let \(J_i\) in $(M_i,\xi_i)$ be Legendrian links presented by links in the exteriors of \(\boldsymbol L_i\). Then \(J_1\) and \(J_2\) are equivalent if and only if the front projections of $J_1\cup\boldsymbol L_1$ and $J_2\cup\boldsymbol L_2$ are related by a sequence of 
\begin{itemize}
    \item planar isotopies of the framed and unframed components,
    \item Legendrian Reidemeister moves of the framed and unframed components,
    \item insertions or removals of standard cancelling pairs,
    \item standard contact handle slides,
    \item standard lantern moves,
    \item standard chain moves, and
    \item standard contact handle slides of unframed components over framed surgery components.
\end{itemize}
\end{theorem}

\subsection{Symplectic fillings}

Contact \((-1)\)-surgery corresponds to the attachment of a symplectic \(2\)-handle in dimension four~\cite{Eliashberg_handle,Weinstein}. Hence, a contact surgery link whose surgery coefficients are all equal to \((-1)\) determines a Stein filling of the surgered contact manifold, up to symplectomorphism and deformation equivalence. In this setting, the chain and lantern moves can be used to produce different Stein fillings of the same contact manifold.

\begin{proposition}[Non-equivalent symplectic fillings via contact Kirby moves]
    Let \((M,\xi)\) be a contact \(3\)-manifold with contact surgery diagram \(\boldsymbol L\) in which all contact surgery coefficients are equal to \((-1)\). Suppose that either a chain move or a lantern move can be applied to \(\boldsymbol L\). Then \((M,\xi)\) admits two Stein fillings that are not homotopy equivalent.
\end{proposition}
    
\begin{proof}
All the chain moves and the lantern moves preserve the property that all contact surgery coefficients are equal to \((-1)\), see Section~\ref{sec:moves}. Thus, the contact surgery links before and after applying the move both determine Stein fillings of \((M,\xi)\).

These two Stein fillings are obtained from \(D^4\) by attaching Weinstein \(2\)-handles, one for each component of the corresponding surgery link. Therefore, their second homology groups are free abelian of ranks equal to the number of surgery components. A standard lantern move changes this number by one, while a standard chain move changes it by ten. Hence the two fillings have second homology groups of different ranks, and in particular, they are not homotopy equivalent.
\end{proof}

\subsection{A contact-geometric proof of Kirby's theorem}

In Appendix~\ref{sec:top_Kirby}, we sketch an argument to deduce the smooth Kirby theorem from Theorem~\ref{thm:main}. This gives a proof of the smooth Kirby theorem based on contact geometry and a presentation of the mapping class group. For other proofs of Kirby's theorem, we refer to~\cite{Kirby,Matveev_Polyak_Kirbytheorem,Lu_Kirbystheorem}.

\subsection{Proof strategy of the contact Kirby theorem}
Kirby's original proof that two framed links describe the same $3$-manifold if and only if they are related by handle slides and blow-ups/blow-downs uses Cerf theory in dimension five, viewing integral surgery as $4$-dimensional $2$-handle attachment to $D^4$~\cite{Kirby}. A similar strategy for a contact version of Kirby's theorem is, at least in principle, conceivable. Indeed, contact $(-1)$-surgery can be realized by attaching a Weinstein $2$-handle along the convex boundary of a symplectic manifold, while contact $(+1)$-surgery is naturally interpreted through the inverse symplectic cobordism; see, for instance,~\cite{Eliashberg_handle,Weinstein,Geiges_book}. This is not the approach taken here.

Instead, our strategy, building on Avdek's ribbon moves framework, is closer in spirit to Lu's proof of Kirby's theorem~\cite{Lu_Kirbystheorem}. Roughly speaking, Lu's argument proceeds by translating framed links describing the same $3$-manifold $M$ into Heegaard splitting data. After suitable modifications of the surgery diagrams by handle slides and blow-ups and blow-downs, the links are placed in a precise way on a common Heegaard surface $\Sigma$ of $S^3$ (more precisely, each component is embedded on a fiber of a tubular neighborhood $\Sigma\times[0,1]$ of $\Sigma$) and hence give rise to two factorizations of the same gluing map for a Heegaard splitting of $M$. One then compares these factorizations using a presentation of the mapping class group and interprets the resulting relations back on the level of framed links.

A crucial feature of Lu's proof is the final step: the moves on framed links arising from mapping class group relations are not left as a new auxiliary family of moves, but are shown to be expressible in terms of handle slides and blow-ups and blow-downs. In this way, Lu recovers the classical Kirby theorem.

Our proof follows the same general philosophy, but replaces Heegaard splittings with open book decompositions. Roughly speaking, Avdek's main result says that, up to known contact Kirby moves, two contact surgery links representing contactomorphic contact $3$-manifolds $(M,\xi)$ can be embedded on pages of a common open book supporting $(S^3,\xist)$ in such a way that the corresponding products of Dehn twists define the same element of the mapping class group of the page. In particular, this yields two abstract open book presentations of $(M,\xi)$. The main ingredients up to this point are an algorithm, by Avdek, that produces an open book supporting any given Legendrian link, and the fact that a positive stabilization of an open book corresponds to the insertion of a cancelling pair plus a standard handle slide.

At this stage, we can apply Gervais' presentation of the mapping class group from~\cite{Gervais} to break the equality of these factorizations into elementary relations. The key point is then to interpret each of these relations geometrically on the level of the surgery diagram. In this way, two contact surgery links defining the same element of the mapping class group as explained above differ by a finite list of diagrammatic operations, which we claim are precisely the desired contact Kirby moves. The main technical point is to show that any two pairs of contact surgery links differing by a given relation in the mapping class group admit the same normal form at the level of front projections, so that each relation can be represented by a single diagrammatic move.

At this point, however, the analogy with Lu's proof breaks down: there is no direct contact analogue of blow-up or blow-down. Indeed, contact $(\pm1)$-surgery on a Legendrian unknot never returns $(S^3,\xist)$~\cite{Kegel_Legendrian_complement}. For this reason, one cannot expect to carry out the final reduction step in Lu's argument in the same way in the contact setting.

\begin{remark}\label{rmk_weak_blowup}
There is also a more conceptual reason why one should not expect a contact analogue of blow-up or blow-down, even in a weaker sense. In the classical Heegaard-splitting picture, blow-ups and blow-downs can be interpreted in terms of boundary mapping classes that extend over a handlebody. In the contact setting, an open book gives rise to a contact Heegaard splitting, so one may ask whether an analogous mechanism could still produce a weak version of blow-up or blow-down. The page-supported mapping classes produced by contact surgery turn out to be too rigid for this: their intersection with the boundary mapping classes extending over the contact handlebody is trivial. We postpone the precise formulation and proof of this statement to Section~\ref{sec:blow_up_discussion}; see in particular Lemma~\ref{lem:trivial_intersection}.
\end{remark}

\subsection*{Conventions}

We assume familiarity with Dehn surgery, mapping class groups, and \(3\)-dimensional contact topology at the level of~\cite{Gompf_Stipsicz,FM12,Geiges_book}. For background on contact surgery and Kirby calculus in contact topology, we refer to~\cite{Gompf,Ozbagci_Stipsicz,Ding_Geiges_fillable,Ding_Geiges_Stipsicz,Ding_Geiges,Ding_Geiges_slides,Kegel_thesis,Kegel_Legendrian_complement,Kegel_Onaran,Etnyre_Kegel_Onaran,Casals_Etnyre_Kegel}.

We work in the smooth category. Unless otherwise stated, all manifolds, maps, and auxiliary objects are smooth and oriented, and all diffeomorphisms are orientation-preserving. All contact structures are assumed to be positive and coorientable, and all contactomorphisms are understood to preserve the coorientation. Since a contact manifold obtained from a contact surgery diagram or from an abstract open book is only determined up to contactomorphism, we usually consider contact manifolds up to contactomorphism (and not isotopy). By slight abuse of notation, Legendrian links in $(\R^3,\xist)$ and $(S^3,\xist)$ are sometimes considered up to isotopy, and Legendrian links in more general manifolds are sometimes considered up to equivalence.

We regard \((\R^3,\xist)\) as the complement of a point in \((S^3,\xist)\). Our convention for the standard contact structure on \(\R^3\) is $\xist=\ker(dz+x\,dy)$. Thus the Reeb vector field is \(\partial_z\), and the front projection is the projection to the \((y,z)\)-plane. Throughout the article, whenever we draw objects in $(\R^3,\xist)$, we use the front projection unless otherwise stated. We always assume that the front projections are generic, which can be achieved by a perturbation. We use the following conventions for such pictures. If Legendrian strands leave the depicted region in parallel, then they are understood to remain vertical translates of each other in the $\partial_z$-direction until they re-enter the picture. Since all Legendrian knots under consideration carry contact surgery coefficients $(\pm1)$, we usually indicate these coefficients simply by the symbols $\pm$. When colours are used, matching colours indicate the corresponding strands outside the depicted region.

Finally, throughout the introduction and until Section~\ref{sec:main_theorem_body}, the term \emph{contact Kirby move} is used in a diagrammatic sense: it means an explicit move on front projections of contact surgery diagrams. Later, in Section~\ref{sec:moves}, we also discuss the broader formulation of contact Kirby moves directly at the level of contact surgery links.

We normalize the $d_3$-invariant so that $d_3(S^3,\xist)=0$. With this convention, it takes integral values on homology spheres and is additive under connected sums, following conventions used, for example, in~\cite{Casals_Etnyre_Kegel,Etnyre_Kegel_Onaran,Kegel_Onaran}. This differs from the original normalization in~\cite{Gompf,Ding_Geiges_Stipsicz}.

\subsection*{Acknowledgements}
This project was initiated during a stay of the authors at the ICERM semester programme \textit{Braids} (February 1--May 6, 2022). We thank ICERM for the invitation, hospitality, and financial support. We thank Daniele Zuddas for early discussions with Marc Kegel and for preliminary explorations related to this project. We are also grateful to Filippo Bianchi for pointing us to~\cite{Gervais}, which turned out to be a crucial reference for this work, and to Russell Avdek for his encouragement and interest throughout the project.

\subsection*{Individual grant support}
Marc Kegel is supported by a Ram\'on y Cajal grant \mbox{(RYC2023-043251-I)} and PID2024-157173NB-I00 funded by MCIN/AEI/10.13039/501100011033, by ESF+, and by FEDER, EU; and by a VII Plan Propio de Investigación y Transferencia (SOL2025-36103) of the University of Sevilla, and was funded by the DFG, German Research Foundation (Project: 561898308). 
The research of Eric Stenhede and Vera Vértesi was supported by the Austrian Science Fund (FWF) projects PAT7436924 and P~34318. Eric Stenhede was also supported by the Vienna School of Mathematics.

\subsection*{Open access}
For the purpose of open access, the authors have applied a CC BY public copyright licence to any author-accepted manuscript version arising from this submission.


\section{Mapping class groups and Dehn twists}\label{sec:MCG}

Let $\Sigma$ be a compact, oriented, smooth surface with nonempty boundary. We write $\mathrm{Diff}^+(\Sigma,\partial\Sigma)$ for the group of orientation-preserving diffeomorphisms of $\Sigma$ that are equal to the identity near $\partial\Sigma$, and $\mathrm{Diff}^+_0(\Sigma,\partial\Sigma)$ for the identity component of this group. The \emph{mapping class group} of $\Sigma$ relative to the boundary is then defined by
\[
\MCG(\Sigma,\partial\Sigma)
=
\mathrm{Diff}^+(\Sigma,\partial\Sigma)/\mathrm{Diff}^+_0(\Sigma,\partial\Sigma).
\]
A theorem of Lickorish says that $\MCG(\Sigma,\partial\Sigma)$ is generated by Dehn twists along simple closed curves~\cite{Lickorish_MCG}. Given such a curve $\aa\subset\Sigma$, we denote by $\tau_\aa^\pm$ the mapping class of the positive or negative Dehn twist along~$\aa$.
We now recall a few standard relations among Dehn twists; see~\cite{FM12}. Since Dehn twists along isotopic curves define the same element of $\MCG(\Sigma,\partial\Sigma)$, the subscript in $\tau_\aa^\pm$ may be regarded as the isotopy class of $\aa$.

\begin{example}[Cancelling Dehn twists]\label{ex:cancellation_DT}
For any simple closed curve $\aa$, we have
\[
\tau_\aa^+\,\tau_\aa^-\;=\;\tau_\aa^-\,\tau_\aa^+\;=\;1.
\]
In particular, $\tau_\aa^-$ is the inverse of $\tau_\aa^+$.
\end{example}

\begin{example}[Commutativity]
If $\aa$ and $\bb$ are disjoint simple closed curves, then the corresponding Dehn twists commute, i.e.\
\[
\tau_\aa^+\,\tau_\bb^+ = \tau_\bb^+\,\tau_\aa^+
\qquad \text{if } \aa\cap\bb=\emptyset.
\]
\end{example}

Now let $\aa$ and $\bb$ be simple closed curves on $\Sigma$ that intersect transversely. We define a new curve $\aa\bb$ by smoothing each intersection point of the graph $\aa\cup\bb$ as follows: when traversing $\bb$ and approaching an intersection point, we turn onto the strand of $\aa$ lying to the right. See Figure~\ref{fig:resolution_curves}.

\begin{figure}[htbp]
    \centering
    \begin{overpic}[scale=1]{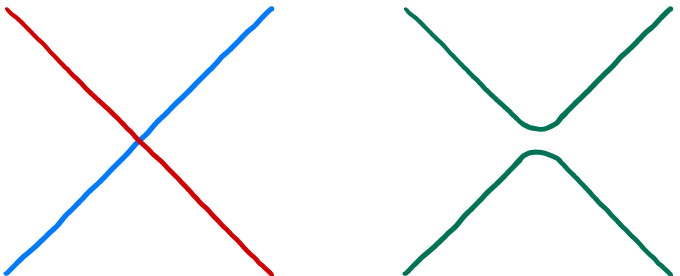}
        \put(7,35){$\aa$}
        \put(30,35){$\bb$}
        \put(66,35){$\aa\bb$}
    \end{overpic}
    \caption{The curves $\aa$, $\bb$ and $\aa\bb$ near the intersection point $\aa\cap\bb$.}\label{fig:resolution_curves}
\end{figure}

If $\aa$ and $\bb$ intersect transversely at a single point, then $\aa\bb$ is isotopic to $\tau_\aa^+(\bb)$.

\begin{example}[Braid relations]\label{ex:braid_relations}
Let $\aa$ and $\bb$ be simple closed curves intersecting transversely at a single point. Then
\[
\tau_\aa^+\;\tau_\bb^+ \;=\; \tau_{\aa\bb}^+\;\tau_\aa^+.
\]
This is a \emph{braid relation}; see Figure~\ref{fig:braid_relation}.
\end{example}

\begin{figure}[htbp]
    \centering
    \begin{overpic}[scale=1]{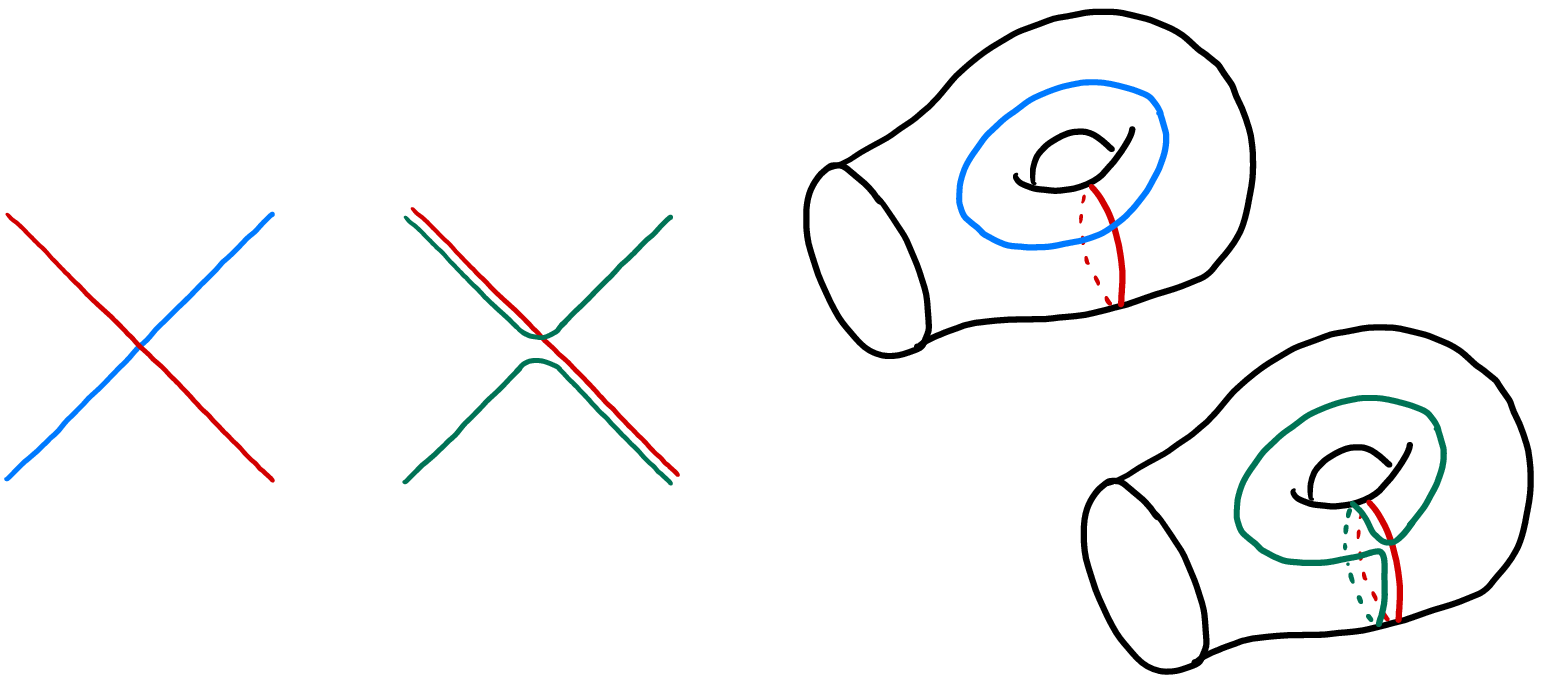}
        \put(4,28){$\aa$}
        \put(13,28){$\bb$}
        \put(30,28){$\aa$}
        \put(38,28){$\aa\bb$}
        \put(77,35){$\bb$}
        \put(74,27){$\aa$}
        \put(95,15){$\aa\bb$}
        \put(92,6){$\aa$}
    \end{overpic}
    \caption{Left: the curves $\aa$, $\bb$ and $\aa\bb$ near their intersection point. Right: a neighborhood of $\aa\cup\bb$ with the three curves drawn on it.}\label{fig:braid_relation}
\end{figure}
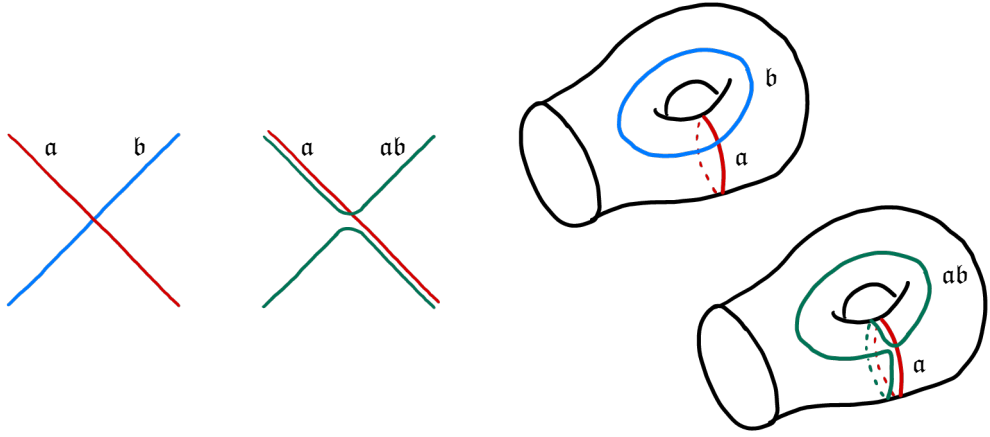

\begin{example}[Lantern relations]\label{ex:lantern_relations}
Let $\aa$ and $\bb$ be simple closed curves that intersect transversely in two points and algebraically zero times. Then
\[
\tau_\aa^+\;\tau_\bb^+\;\tau_{\aa\bb}^+
\;=\;
\tau_{\dd_4}^+\;\tau_{\dd_3}^+\;\tau_{\dd_2}^+\;\tau_{\dd_1}^+,
\]
where $\dd_1,\dd_2,\dd_3,\dd_4$ are the boundary components of a regular neighborhood of $\aa\cup\bb$. This is a \emph{lantern relation}; see Figure~\ref{fig:lantern_relation}.
\end{example}

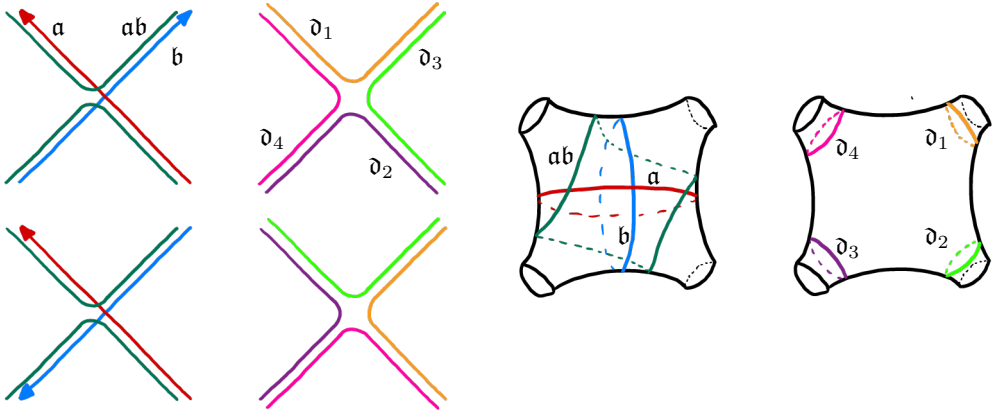
\begin{figure}[htbp]
    \centering
    \begin{overpic}[scale=1]{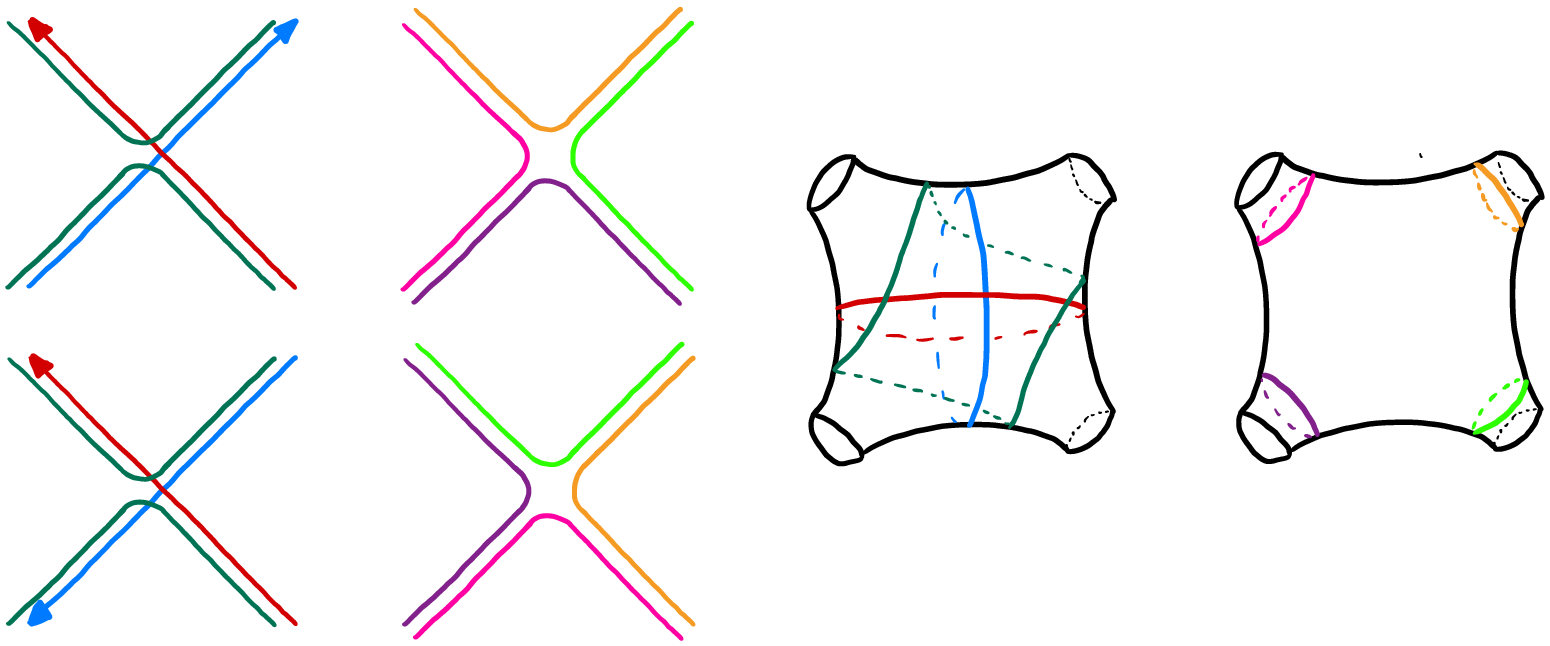}
        \put(5,38){$\aa$}
        \put(17,35){$\bb$}
        \put(12,38){$\aa\bb$}
        \put(31,38){$\dd_1$}
        \put(42,35){$\dd_3$}
        \put(37,24){$\dd_2$}
        \put(26,27){$\dd_4$}
        \put(55,25){$\aa\bb$}
        \put(62,17){$\bb$}
        \put(65,23){$\aa$}
        \put(84,26){$\dd_4$}
        \put(84,16){$\dd_3$}
        \put(93,27){$\dd_1$}
        \put(93,17){$\dd_2$}
    \end{overpic}
    \caption{Left: the curves in a lantern relation near the two intersection points of $\aa$ and $\bb$ (which have opposite sign). Right: a neighborhood of $\aa\cup\bb$ with all the curves drawn.}\label{fig:lantern_relation}
\end{figure}

\begin{example}[Chain relations]\label{ex:chain_relations}
Let $\aa$, $\bb$ and $\cc$ be simple closed curves such that $\aa$ and $\bb$ intersect transversely in one point, $\bb$ and $\cc$ intersect transversely in one point, and $\aa\cap\cc=\emptyset$. Then
\[
\bigl(\tau_\aa^+\;\tau_\bb^+\;\tau_\cc^+\bigr)^4
\;=\;
\tau_{\dd_2}^+\;\tau_{\dd_1}^+,
\]
where $\dd_1$ and $\dd_2$ are the boundary components of a regular neighborhood of $\aa\cup\bb\cup\cc$. This is a \emph{chain relation}; see Figure~\ref{fig:chain_relation}.
\end{example}

\begin{figure}[htbp]
    \centering
    \begin{overpic}[scale=1]{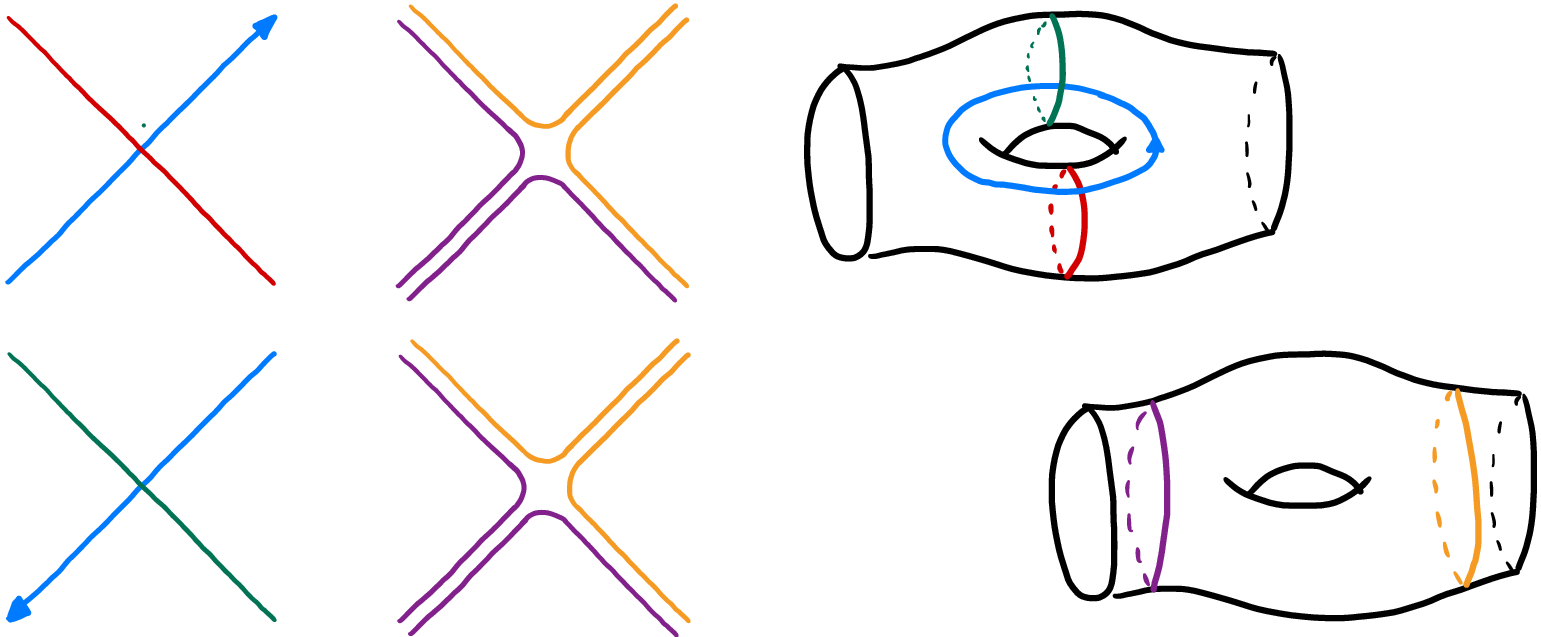}
        \put(4,38){$\aa$}
        \put(13,38){$\bb$}
        \put(4,16){$\cc$}
        \put(31,38){$\dd_1$}
        \put(37,24){$\dd_2$}
        \put(70,37){$\cc$}
        \put(76,32){$\bb$}
        \put(71,26){$\aa$}
        \put(77,12){$\dd_2$}
        \put(90,12){$\dd_1$}
    \end{overpic}
    \caption{Left: the curves in a chain relation near the intersection points of $\aa$ with $\bb$ and of $\bb$ with $\cc$. Right: a neighborhood of $\aa\cup\bb\cup\cc$ with the curves $\aa,\bb,\cc,\dd_1,\dd_2$ drawn.}\label{fig:chain_relation}
\end{figure}
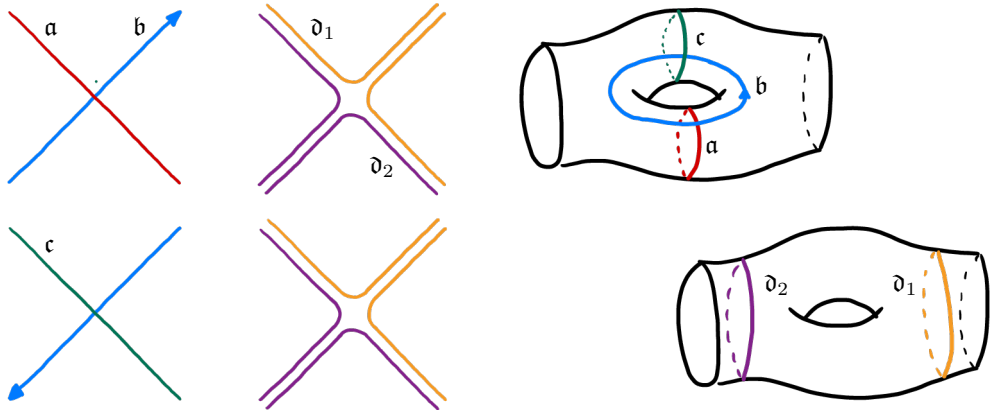

\begin{theorem}[Gervais {\cite[Theorem~B]{Gervais}}]\label{thm:Gervais}
Let $\Sigma$ be a compact, oriented surface of genus $g\geq 2$. Then the mapping class group $\MCG(\Sigma,\partial\Sigma)$ admits the following presentation:
\begin{itemize}
    \item \textbf{Generators:} all Dehn twists $\tau_\aa^\pm$ along nonseparating simple closed curves $\aa$ in $\Sigma$.
    \item \textbf{Relations:}
    \begin{enumerate}[label=(\alph*)]
        \item $\tau_\aa^+ \tau_\aa^- = \tau_\aa^- \tau_\aa^+ = 1$ for all $\aa$ \textup{(cancelling Dehn twists)}.
        \item $\tau_\aa^+ \tau_\bb^+ = \tau_\bb^+ \tau_\aa^+$ whenever $\aa\cap \bb=\emptyset$ \textup{(commutativity relations)}.
        \item All braid relations between Dehn twists along curves intersecting in a single point.
        \item One lantern relation involving only Dehn twists along nonseparating curves.
        \item One chain relation involving only Dehn twists along nonseparating curves.\qed
    \end{enumerate}
\end{itemize}
\end{theorem}

\section{Contact structures and Legendrian graphs}\label{sec:Legendrian_graphs}

Let $M$ be a smooth, oriented $3$--manifold. A $1$--form $\alpha\in\Omega^1(M)$ is called a \emph{contact form} if
\[
\alpha\wedge d\alpha>0.
\]
A \emph{contact structure} on $M$ is a rank--$2$ distribution of the form
\[
\xi=\ker\alpha
\]
for some contact form $\alpha$.
On $\R^3$ we define the \textit{standard contact form}
\[
\alphast = dz + x\,dy
\]
and the corresponding \textit{standard contact structure}
\[
\xist=\ker\alphast.
\]
Likewise, on $S^3\subset\R^4$, we consider the \textit{standard contact structure}
\[
\xist = \ker\bigl(x_1dy_1 - y_1dx_1 + x_2dy_2 - y_2dx_2\bigr)\big|_{TS^3}.
\]
For every $p\in S^3$, the punctured sphere $(S^3\setminus\{p\},\xist|_{S^3\setminus\{p\}})$ is contactomorphic to $(\R^3,\xist)$; see \cite[Proposition~2.1.8]{Geiges_book}.
If $\alpha$ is a contact form, its \emph{Reeb vector field} $R_\alpha$ is characterized by the conditions
\[
d\alpha(R_\alpha,-)\equiv 0,
\qquad
\alpha(R_\alpha)\equiv 1.
\]
For the standard form $\alphast$, one has
\[
R_{\alphast}=\partial_z.
\]
A vector field $\X$ on $(M,\xi)$ is called a \emph{contact vector field} if its flow preserves the contact structure~$\xi$.
Let $(M,\xi)$ be a contact $3$--manifold. A smooth curve
\[
r\colon I\longrightarrow M,
\qquad I\subset\R \text{ or } I=S^1,
\]
is called \emph{Legendrian} if its tangent vector lies in the contact planes at every point, that is,
\[
r'(s)\in \xi_{r(s)}
\qquad\text{for all }s\in I.
\]
A properly embedded $1$--dimensional CW-complex is called a \emph{Legendrian graph} if each edge is Legendrian and, at every vertex, the incident edges determine distinct oriented tangencies inside the contact plane. Two Legendrian graphs are said to be \emph{Legendrian isotopic} if they can be joined by a smooth one-parameter family of Legendrian graphs that preserves the cyclic ordering of the edges at each vertex.

For Legendrian graphs in $(\R^3,\xist)$, two projections are particularly useful: the \emph{front projection} to the $(y,z)$--plane and the \emph{Lagrangian projection} to the $(x,y)$--plane. We use the coordinate convention shown in Figure~\ref{fig:coordinate_system}.
Now let
\[
r=(r_x,r_y,r_z)\colon I\to(\R^3,\xist)
\]
be a Legendrian immersion. Its front projection is the plane curve
\[
r_F=(r_y,r_z).
\]
For a generic front projection, every point has one of the local models displayed in Figure~\ref{fig:Legendrian_projection}: a smooth point, a transverse double point, a cusp, or a vertex of a Legendrian graph. Conversely, any planar graph in the $(y,z)$--plane whose local behavior is modeled on Figure~\ref{fig:Legendrian_projection}(a)--(e) admits a unique lift to a generic Legendrian graph in $(\R^3,\xist)$.

\begin{figure}[htbp]
    \centering
    \begin{overpic}[scale=1]{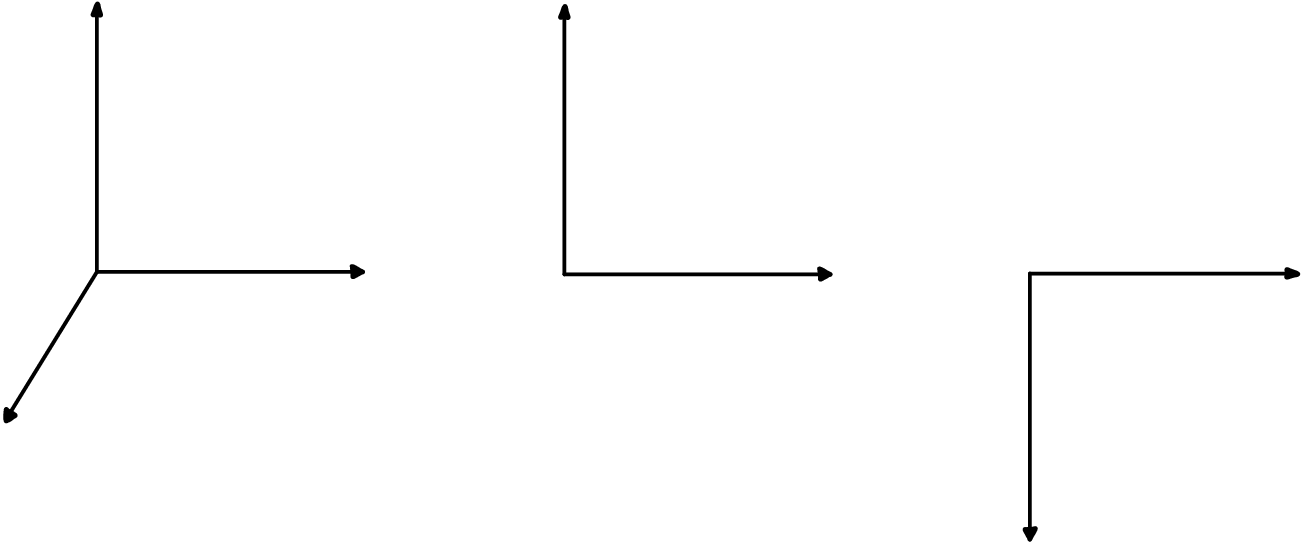}
        \put(-1,11){$x$}
        \put(5,39){$z$}
        \put(25,22){$y$}
        \put(41,39){$z$}
        \put(61,22){$y$}
        \put(77,2){$x$}
        \put(97,22){$y$}
    \end{overpic}
    \caption{Left: coordinates on $\R^3$. Center: the front projection. Right: the Lagrangian projection.}
    \label{fig:coordinate_system}
\end{figure}

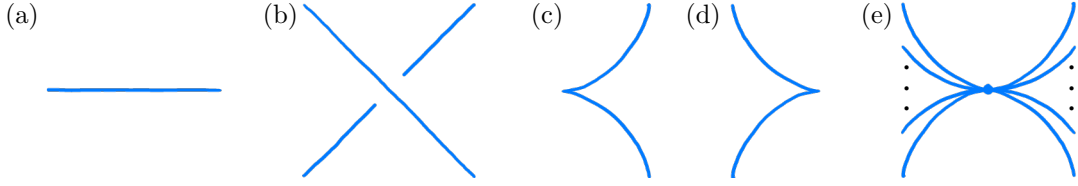
\begin{figure}[htbp]
    \centering
    \begin{overpic}[scale=1]{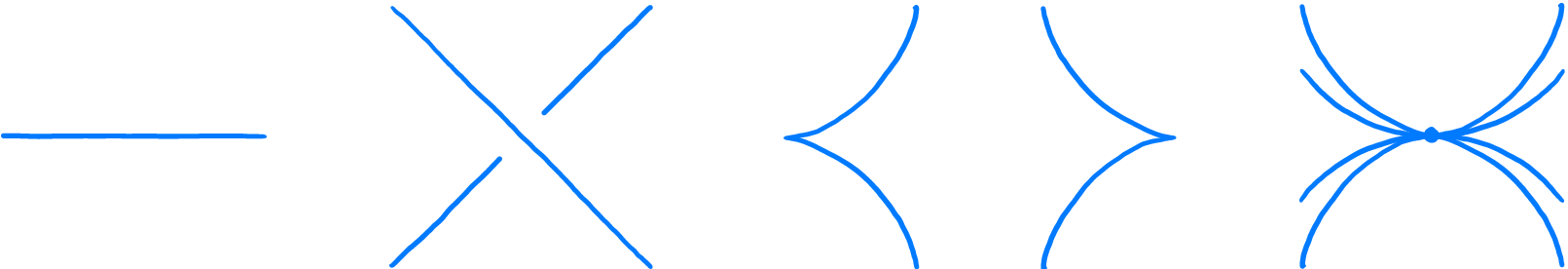}
        \put(-4,15){(a)}
        \put(21,15){(b)}
        \put(47,15){(c)}
        \put(62,15){(d)}
        \put(79,15){(e)}
        \put(83,10){$\boldsymbol\cdot$}
        \put(83,8){$\boldsymbol\cdot$}
        \put(83,6){$\boldsymbol\cdot$}
        \put(99,10){$\boldsymbol\cdot$}
        \put(99,8){$\boldsymbol\cdot$}
        \put(99,6){$\boldsymbol\cdot$}
    \end{overpic}
    \caption{Local models for points in the front projection of a generic Legendrian graph.}
    \label{fig:Legendrian_projection}
\end{figure}

The Lagrangian projection of $r$ is the immersed curve
\[
r_L=(r_x,r_y)
\]
in the $(x,y)$--plane. Once $r_L$ and the value of the $z$--coordinate at one point $r(s_0)$ are known, the full Legendrian curve can be reconstructed from the formula
\[
r_z(s)=r_z(s_0)-\int_{s_0}^s r_x(u)\,r_y'(u)\,du.
\]
In particular, if $r$ is closed, then
\[
\oint_{r_L} x\,dy = 0.
\]
Furthermore, $r$ is embedded precisely when every loop in the Lagrangian projection, except possibly the entire loop in the closed case, encloses nonzero signed area. The local models in the Lagrangian projection corresponding to Figure~\ref{fig:Legendrian_projection} are shown in Figure~\ref{fig:Lagrangian_projection}.

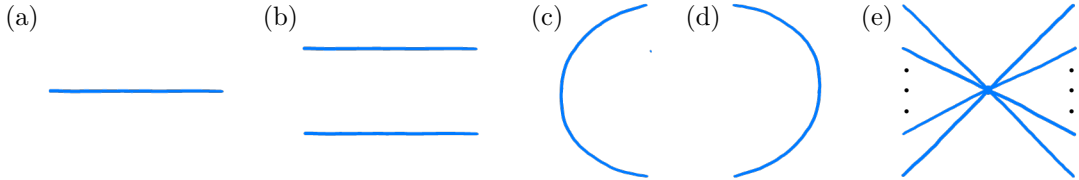
\begin{figure}[htbp]
    \centering
    \begin{overpic}[scale=1]{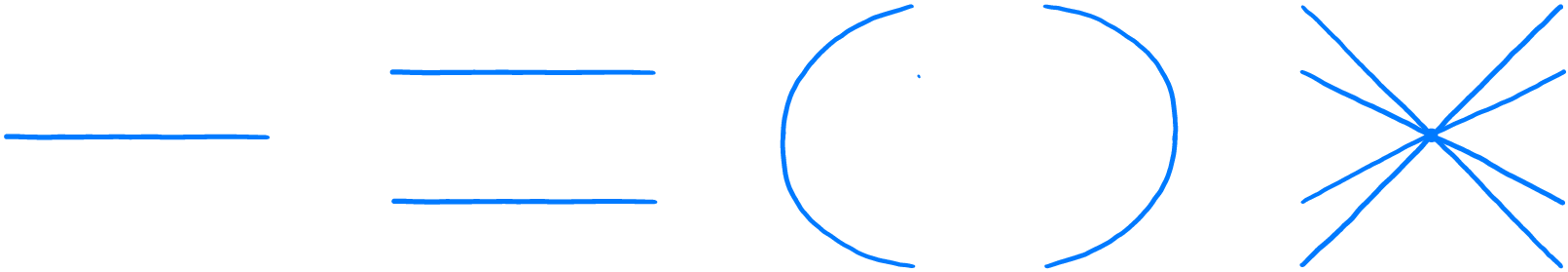}
        \put(-4,15){(a)}
        \put(21,15){(b)}
        \put(47,15){(c)}
        \put(62,15){(d)}
        \put(79,15){(e)}
        \put(83,10){$\boldsymbol\cdot$}
        \put(83,8){$\boldsymbol\cdot$}
        \put(83,6){$\boldsymbol\cdot$}
        \put(99,10){$\boldsymbol\cdot$}
        \put(99,8){$\boldsymbol\cdot$}
        \put(99,6){$\boldsymbol\cdot$}
    \end{overpic}
    \caption{Lagrangian projections corresponding to the front projections in Figure~\ref{fig:Legendrian_projection}. The $y$--coordinate agrees with that of the front projection, and the $x$--coordinate parametrizes the vertical direction (top to bottom is positive).}
    \label{fig:Lagrangian_projection}
\end{figure}

The following two Reidemeister theorems for Legendrian links and Legendrian graphs were first shown in~\cite{Legendrian_Reidemeister} and~\cite[Proposition~4.2]{BM09}.

\begin{theorem}[Swiatkowski \cite{Legendrian_Reidemeister}]\label{thm:Reidemeister_Legendrian}
   Two Legendrian links in $(\R^3,\xist)$ are isotopic if and only if their front projections are related by a sequence of planar isotopies and the Reidemeister moves I, II, III shown in Figure~\ref{fig:Reidemeister_Leg_link}.\qed
\end{theorem}

\begin{figure}[htbp]
    \centering
    \begin{overpic}[scale=1]{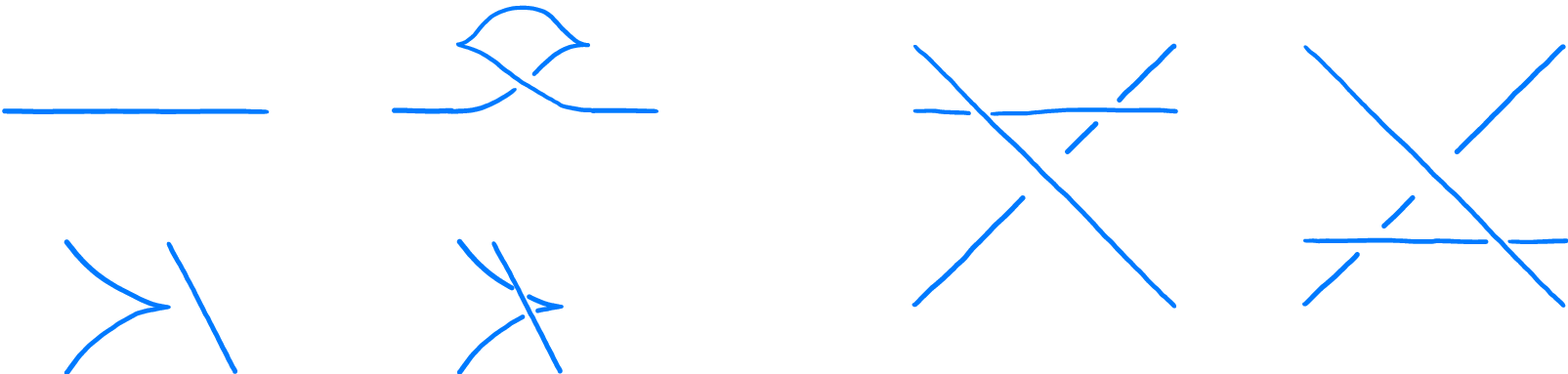}
        \put(20,16){$\xleftrightarrow[]{\mathrm{I}}$}
        \put(20,4){$\xleftrightarrow[]{\mathrm{II}}$}
        \put(77,12){$\xleftrightarrow[]{\mathrm{III}}$}
    \end{overpic}
    \caption{Legendrian Reidemeister moves for Legendrian links. Horizontal and vertical reflections of these moves are also allowed, with crossings adjusted appropriately. No other Legendrian strand of the diagram is allowed to appear in these pictures.}\label{fig:Reidemeister_Leg_link}
\end{figure}

\begin{theorem}[Baader--Ishikawa \cite{BM09}]\label{thm:Reidemeister_graph}
    Two Legendrian graphs in $(\R^3,\xist)$ are isotopic if and only if their front projections are related by a sequence of planar isotopies and the moves I, II, III, together with the additional moves $\mathrm{II}_G$ and $R$ shown in Figure~\ref{fig:Reidemeister_Leg_graph}.\qed
\end{theorem}

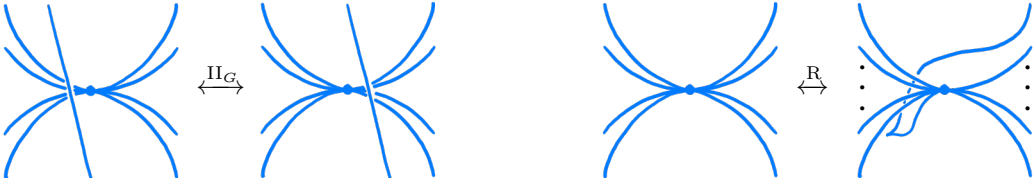
\begin{figure}[htbp]
    \centering
    \begin{overpic}[scale=1]{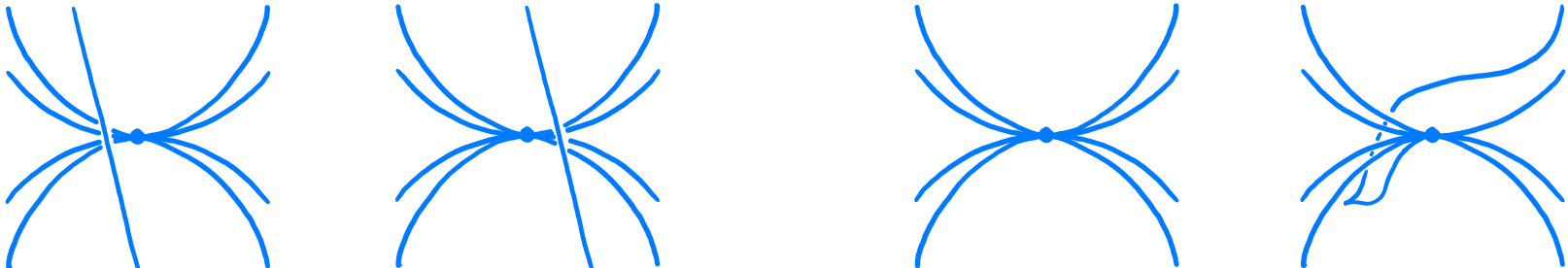}
        \put(19,8){$\xleftrightarrow[]{\mathrm{II}_G}$}
        \put(77,8){$\xleftrightarrow[]{\mathrm{R}}$}
        \put(83,10){$\boldsymbol\cdot$}
        \put(83,8){$\boldsymbol\cdot$}
        \put(83,6){$\boldsymbol\cdot$}
        \put(99,10){$\boldsymbol\cdot$}
        \put(99,8){$\boldsymbol\cdot$}
        \put(99,6){$\boldsymbol\cdot$}
    \end{overpic}
    \caption{Additional Legendrian Reidemeister moves for Legendrian graphs. Horizontal and vertical reflections of these moves are again allowed.}\label{fig:Reidemeister_Leg_graph}
\end{figure}

\section{Convex surfaces}\label{sec:convex}

Let $\Sigma$ be an oriented surface embedded in a contact manifold $(M,\xi)$. We denote by $\Sigma_\xi$ the characteristic foliation induced on~$\Sigma$. The surface $\Sigma$ is called \emph{convex} if there exists a contact vector field $\X$ which is everywhere transverse to~$\Sigma$. An equivalent formulation is that $\Sigma$ admits a neighborhood of the form $\Sigma\times(-\varepsilon,\varepsilon)$ in which the contact structure is invariant under translation in the second factor. In such a vertically invariant neighborhood, one may write the contact structure as the kernel of a $1$--form
\[
\beta+u\,dt,
\]
where $\beta$ and $u$ are the pull-backs, via the projection $\Sigma\times(-\varepsilon,\varepsilon)\to\Sigma$, of a $1$-form and a smooth function on~$\Sigma$, respectively, and where $t\in(-\varepsilon,\varepsilon)$ is the coordinate in the second factor.

The \emph{dividing set} $\Gamma_\Sigma$ consists of those points of $\Sigma$ at which the chosen transverse contact vector field lies in the contact planes. In the vertically invariant model above, this is exactly the zero set of the function~$u$. The \emph{positive region} of $\Sigma$ is where $u>0$, while the \emph{negative region} where $u<0$.

A singular foliation $\mathcal F$ on $\Sigma$ is said to be \emph{divided} by a multicurve $\Gamma$ if $\Gamma$ is transverse to $\mathcal F$ and there exists an area form $\omega$ on $\Sigma$ and a vector field $v$ generating $\mathcal F$ such that $\mathcal L_v\omega$ is nowhere zero on $\Sigma\setminus\Gamma$, with $v$ pointing out of the positive region and into the negative region along~$\Gamma$. Whenever $\Sigma$ is convex, its characteristic foliation is divided by~$\Gamma_\Sigma$.

Giroux's flexibility theorem asserts that, once the dividing set is fixed, the characteristic foliation may be altered arbitrarily within the class of singular foliations divided by it.

\begin{theorem}[Giroux's flexibility theorem \cite{Giroux91,Etnyre_VHM11}]\label{thm:Giroux_flexibility}
Let $\Sigma$ be a convex surface in a contact manifold $(M,\xi)$ with respect to a contact vector field $\X$. If $\Sigma$ has boundary, assume that $\partial\Sigma$ coincides with the dividing set $\Gamma_\Sigma$. Let $\mathcal{F}$ be a singular foliation on $\Sigma$ divided by $\Gamma_\Sigma$, and let $\mathcal{N}(\Sigma)$ be any neighborhood of $\Sigma$ of the form $\Sigma\times(-\varepsilon,\varepsilon)$ in $M$. Then there exists an isotopy
\[
\psi_t\colon \Sigma \longrightarrow \mathcal{N}(\Sigma), \qquad t\in[0,1],
\]
of embeddings such that:
\begin{itemize}
    \item $\psi_0$ is the inclusion $\Sigma\hookrightarrow M$;
    \item for all $t\in[0,1]$, the surface $\psi_t(\Sigma)$ is convex with respect to $\X$ and has dividing set $\psi_t(\Gamma_\Sigma)$;
    \item the characteristic foliation $(\psi_1(\Sigma))_\xi$ coincides with $\psi_1(\mathcal{F})$.
\end{itemize}
Moreover, if $\mathcal{F}$ agrees with $\Sigma_\xi$ on a neighborhood of $\Gamma_\Sigma$ in $\Sigma$, then the isotopy may be chosen to be fixed on a sufficiently small neighborhood of $\Gamma_\Sigma$.\qed
\end{theorem}

Now let $\Sigma$ be a convex surface, and in the case $\partial\Sigma\neq\emptyset$ assume again that $\partial\Sigma=\Gamma_\Sigma$. A graph $C\subset\Sigma$ is called \emph{nonisolating}\footnote{This is the notion of nonisolating that is relevant for our purposes. For convex surfaces with Legendrian boundary, or for convex surfaces with transverse boundary and dividing set disjoint from the boundary, the definition is different.} if it satisfies the following conditions.
\begin{itemize}
    \item $C$ is properly embedded in $\Sigma$;
    \item $C$ is disjoint from $\partial\Sigma$;
    \item $C$ is transverse to the dividing set $\Gamma_\Sigma$ whenever the two intersect; and
    \item every connected component of $\Sigma\setminus C$ has nonempty intersection with $\Gamma_\Sigma$.
\end{itemize}

\begin{theorem}[Legendrian realization principle \cite{Kanda1998,Honda_classification,Etnyre04}]\label{thm:Legendrian_realization}
Let $\Sigma$ be a convex surface in a contact manifold $(M,\xi)$ with respect to a contact vector field $\X$. If $\Sigma$ has boundary, assume that $\partial\Sigma$ coincides with the dividing set $\Gamma_\Sigma$. Let $C\subset\Sigma$ be a nonisolating graph, and let $\mathcal{N}(\Sigma)$ be any neighborhood of $\Sigma$ of the form $\Sigma\times(-\varepsilon,\varepsilon)$ in $M$. Then there exists an isotopy
\[
\psi_t\colon \Sigma \longrightarrow \mathcal{N}(\Sigma),\qquad t\in[0,1],
\]
of embeddings such that
\begin{itemize}
    \item $\psi_0$ is the inclusion $\Sigma\hookrightarrow M$, and $\psi_t$ is fixed on a neighborhood of $\Gamma_\Sigma$;
    \item for all $t\in[0,1]$, the surface $\psi_t(\Sigma)$ is convex with respect to $\X$ and has dividing set $\Gamma_{\psi_t(\Sigma)}=\Gamma_\Sigma$;
    \item the graph $\psi_1(C)\subset \psi_1(\Sigma)$ is Legendrian.\qed
\end{itemize}
\end{theorem}

\begin{theorem}[Giroux's surface neighborhood theorem~\cite{Giroux91}]\label{thm:neigh_thm}
Let $S_i$ be compact oriented surfaces embedded in contact $3$--manifolds $(M_i,\xi_i)$, for $i=0,1$, and let
\[
\phi: S_0 \to S_1
\]
be a diffeomorphism which identifies the characteristic foliations induced by $\xi_0$ and $\xi_1$ (with orientation). Then there is a contactomorphism of neighborhoods
\[
\psi : \mathcal{N}(S_0) \to \mathcal{N}(S_1),
\]
sending $S_0$ to $S_1$, and such that on $S_0$ the restriction $\psi|_{S_0}$ is isotopic to $\phi$ via an isotopy preserving the characteristic foliation.\qed
\end{theorem}

\section{Contact surgery, open books, and Giroux's correspondence}\label{sec:surgery_OB_Giroux}

Let \(L\) be a Legendrian knot in a contact \(3\)-manifold \((M,\xi)\), and let \(\mathcal N(L)\) be a standard neighborhood of \(L\). Thus \(\mathcal N(L)\) is a solid torus with convex boundary, and the dividing set on \(\partial\mathcal N(L)\) determines the \emph{contact framing} of \(L\). Equivalently, the contact framing is the framing induced by a Legendrian push-off of \(L\).

Given a rational number \(r\), one may remove the interior of \(\mathcal N(L)\) and glue back a solid torus by topological surgery with coefficient \(r\) measured with respect to the contact framing. To perform a contact surgery, we want to extend the contact structure from the knot exterior over the newly glued-in solid torus such that the contact structure restricted to the solid torus is tight. For coefficients of the form \(1/n\), with \(n\in\mathbb Z\), this extension is unique up to isotopy, and the resulting contact manifold is therefore well-defined up to contactomorphism. This operation is called \emph{contact \((1/n)\)-surgery} on \(L\).

In this article, most of the time we only use contact \((\pm1)\)-surgery. If a Legendrian knot \(L\) is equipped with contact surgery coefficient \((+1)\) or \((-1)\), we write \(L^+\) or \(L^-\), respectively. When \(L\subset(S^3,\xist)\), the corresponding smooth surgery framing is \(\tb(L)+1\) for \(L^+\), and \(\tb(L)-1\) for \(L^-\).

An \emph{open book} on a closed oriented $3$--manifold $M$ is a pair $(B,\pi)$ where $B\subset M$ is an oriented link, called the \emph{binding}, and $\pi\colon M\setminus B\to S^1$ is a fibration which coincides with the angular coordinate on a neighborhood $B\times D^2$ of the binding and whose fibers are interiors of Seifert surfaces of $B$. The closures of the fibers of $\pi$ are the \emph{pages}. An \emph{abstract open book} is a pair $(\Sigma,\phi)$ where $\Sigma$ is a compact oriented surface with nonempty boundary and $\phi\in\mathrm{Diff}^+(\Sigma,\partial\Sigma)$ is the identity near the boundary. Equivalent abstract open books correspond to diffeomorphic open books. Every open book determines the diffeomorphism class of an abstract open book. Conversely, every abstract open book determines the diffeomorphism type of a smooth $3$-manifold together with an open book.

We say that a contact structure $\xi$ on $M$ is \emph{compatible with} (or \emph{supported by}) an open book $(B,\pi)$ if there exists a contact form $\alpha$ for $\xi$ such that the binding is positively transverse and $d\alpha$ restricts to a positive area form on the interior of every page. The fundamental correspondence due to Giroux can be stated as follows.

\begin{theorem}[Giroux's correspondence \cite{Giroux02,BHH24,LicVer1_24,LicVer2_24}]\label{thm:Giroux_correspondence}
Let $M$ be a closed, oriented $3$--manifold. There is a one-to-one correspondence between positive cooriented contact structures on $M$ up to isotopy and open books on $M$ up to isotopy and positive stabilization.\qed
\end{theorem}

In particular, any abstract open book determines the contactomorphism type of a contact manifold.

\section{Ribbon surfaces}\label{sec:ribbon_surfaces}

A basic role in what follows is played by ribbon surfaces associated with Legendrian graphs.

\begin{definition}\label{def:ribbon}
    Let $\G$ be a Legendrian graph in a contact manifold $(M,\xi)$. A \emph{ribbon surface} for~$\G$ is a compact, oriented surface $\Sigma\subset M$ such that
    \begin{itemize}
        \item the graph $\G$ is contained in $\Sigma$;
        \item there is a contact form $\alpha$ defining $(M,\xi)$ whose Reeb vector field $R_\alpha$ is positively transverse to~$\Sigma$; 
        \item there exists a vector field $v$ on $\Sigma$ directing the characteristic foliation, positively transverse to $\partial\Sigma$, whose time-$t$ flow $\Phi^t_v$ satisfies
        \[
        \bigcap_{t\in(0,\infty)} \Phi_v^{-t}(\Sigma) = \G.
        \]
    \end{itemize}
    In this situation, $\G$ is called the \emph{Legendrian skeleton} of~$\Sigma$.
\end{definition}

It is implicit in~\cite{Avdek_surgery} that every Legendrian graph admits a ribbon surface.
Every ribbon surface is convex, and its dividing set may be taken to be the boundary. Consequently, the converse to the Legendrian realization principle (see, for instance,~\cite{Durst_Kegel_OB}) shows that the curves relevant for Legendrian realization on a ribbon surface are precisely the homologically nontrivial ones.

\begin{lemma}\label{lem:inverse_Leg_real}
    Let $L$ be a Legendrian knot in a contact manifold $(M,\xi)$. Assume that $L$ is contained in a convex surface $\Sigma$ whose dividing set agrees with $\partial\Sigma$. Then $L$ is homologically nontrivial in~$\Sigma$.\qed
\end{lemma}

\begin{definition}\label{def:leg_real}
    Let $\Sigma$ be a compact, oriented surface, and let $\xi$ be a vertically invariant contact structure on $\Sigma\times I$, where either $I=[-\varepsilon,\varepsilon]$ or $I=\R$.
    
    An isotopy $\phi_t$ of $\Sigma=\Sigma\times\{0\}$ inside $\Sigma\times I$ is called \emph{admissible} if each surface $\phi_t(\Sigma)$ is transverse to $\partial_z$ and if $\phi_t$ is the identity on a neighborhood of $\partial\Sigma$ whenever $\partial\Sigma\neq\emptyset$.
    
    Let $\aa$ be the isotopy class of a simple closed curve on $\Sigma$. A Legendrian knot $L\subset\Sigma\times I$ is said to be a \emph{Legendrian realization} of $\aa$ if there exists an admissible isotopy $\phi_t$ such that $L\subset\phi_1(\Sigma)$ and the isotopy class of $\phi_1^{-1}(L)$ on $\Sigma$ is equal to~$\aa$.
\end{definition}

\begin{theorem}[{\cite[Theorem~3.4]{StenhedeAlgorithm}}]\label{thm:uniqueness_ribbon}
    Let $\Sigma$ be a compact, oriented surface with $\partial\Sigma\neq\emptyset$. Let $\xi=\ker\alpha$ be a vertically invariant contact structure on $\Sigma\times I$, where $I=[-\varepsilon,\varepsilon]$ or $I=\R$, such that $\partial_z$ is a Reeb vector field, and assume that $\partial\Sigma$ is a positive transverse link. Then any two Legendrian realizations of the same isotopy class of a simple closed curve on $\Sigma$ are Legendrian isotopic.\qed
\end{theorem}


\section{Legendrian links and ribbon surfaces}\label{sec:Leglinkandribbon}

In this section, we discuss compatibility properties between Legendrian links and ribbon surfaces. 

\begin{convention}
From now on, we will work in $(\R^3,\xist)\subset(S^3,\xist)$. We use the following conventions.
\begin{itemize}
    \item In the introduction, we distinguished carefully between Legendrian links and their front projections. In the rest of this article, we will often suppress this distinction. Thus, when no confusion can arise, we use the same symbol for a Legendrian link and for its front projection, and we refer to figures as showing Legendrian links rather than their front projections. This is only a notational simplification.
    \item We denote by \(R\), or by \(R_{\G}\), a ribbon surface of a Legendrian graph \(\G\).
    \item Every ribbon surface considered below is assumed to be transverse to the Reeb vector field \(\partial_z\).
    \item If \(\Sigma\) is a surface transverse to \(\partial_z\), then
$\Sigma\times[-\varepsilon,\varepsilon]$
    denotes the neighborhood obtained by flowing \(\Sigma\) along \(\partial_z\). In particular, for a ribbon surface \(R\), each slice \(R\times\{t\}\) is again a ribbon surface.
    \item Let \(\alpha\) be a curve on a ribbon surface \(R\), and let
    \[
    R\times[-\varepsilon,\varepsilon]
    \]
    be a Reeb-invariant neighborhood of \(R\) as above. We write
    \[
    \alpha\times\{t\}
    \]
    for the corresponding curve on the slice \(R\times\{t\}\).
    \item If \(\alpha\subset R\) is a Legendrian-realizable, simple closed curve, we denote by \(\alpha(t)\) the Legendrian realization of \(\alpha\times\{t\}\). More precisely, \(\alpha(t)\) is contained in a boundary-relative \(C^0\)-small perturbation of \(R\times\{t\}\) through surfaces transverse to \(\partial_z\).
\end{itemize}
\end{convention}

\begin{definition}\label{def:perturbation}
  Let $\Sigma\subset(\R^3,\xist)$ be a compact surface with nonempty boundary that is transverse to $\partial_z$. A \emph{perturbation} of $\Sigma$ is a surface $\Sigma'$ obtained by isotoping $\Sigma$ relative to $\partial\Sigma$ through surfaces transverse to $\partial_z$.
\end{definition}

Note that every ribbon surface we consider satisfies the hypotheses of Definition~\ref{def:perturbation}.

\begin{remark}
  If $\Sigma$ is a convex surface transverse to a contact vector field $X$, then in Giroux's terminology~\cite{Giroux91} an isotopy of $\Sigma$ through surfaces transverse to $X$ is called \emph{admissible}. In our setting, $\partial_z$ is a contact vector field, so perturbations are admissible isotopies in this sense.

  In \cite{StenhedeAlgorithm}, admissible isotopies were used for what we now call perturbations. We have switched to the term ``perturbation'' because it shortens our sentences (e.g.\ ``perturb $R$'' rather than ``isotope $R$ by an admissible isotopy'').
\end{remark}

\begin{lemma}\label{lem:perturbation}
  Assume that $L$ is a Legendrian link contained in a perturbation $R'$ of a ribbon surface $R$. Let $U$ be a neighborhood of $R$. Then, after a Legendrian isotopy of $L$, we may assume that $L$ is embedded in a perturbation of $R$ contained in $U$.
\end{lemma}

\begin{proof}
  By the Legendrian realization principle (Theorem~\ref{thm:Legendrian_realization}), we can Legendrian realize a curve on $R$ that defines the same isotopy class as $L$ so that the result $L'$ lies in $U$. By Theorem~\ref{thm:uniqueness_ribbon}, $L$ and $L'$ are Legendrian isotopic.
\end{proof}

By Lemma~\ref{lem:perturbation}, whenever multiple Legendrian knots are compatible with different ribbon surfaces, we may assume that each knot lies in a small perturbation of its corresponding ribbon surface, and that these perturbations are supported in disjoint neighborhoods.

\begin{definition}
  Let $L=\bigcup_{i=1}^n L_i$ be a Legendrian link and $R$ a ribbon surface. We say that $L$ is \emph{compatible} with $R$ if there exists a neighborhood $R\times[-\varepsilon,\varepsilon]$ of $R$, with $\varepsilon>0$, such that each component $L_i$ is embedded in a perturbation of $R\times\{t_i\}$ for some $t_i\in(-\varepsilon,\varepsilon)$. Moreover, we assume that the components of $L$ are indexed so that $j>i$ implies $t_j>t_i$.
  When we write $L\subset R\times[-\varepsilon,\varepsilon]$, we implicitly mean that $L$ is compatible with $R$ and contained in this neighborhood.
\end{definition}

\begin{example}
  Let $R$ be a ribbon surface of a Legendrian knot $L$. Then $L$ is compatible with $R$; in fact, $L$ lies in the characteristic foliation of $R$.
\end{example}

\begin{example}\label{compatible_braid1}
  Consider the Legendrian graph $G$ and its ribbon surface $R$ from Figure~\ref{fig:comp_ex_1}. We claim that the two Legendrian links $L=L_1\cup L_2$ and $L'=L'_1\cup L'_2$ in Figure~\ref{fig:comp_ex_2} are both compatible with $R$ (individually, not simultaneously).

\begin{figure}[htbp]
    \centering
    \begin{overpic}[scale=1]{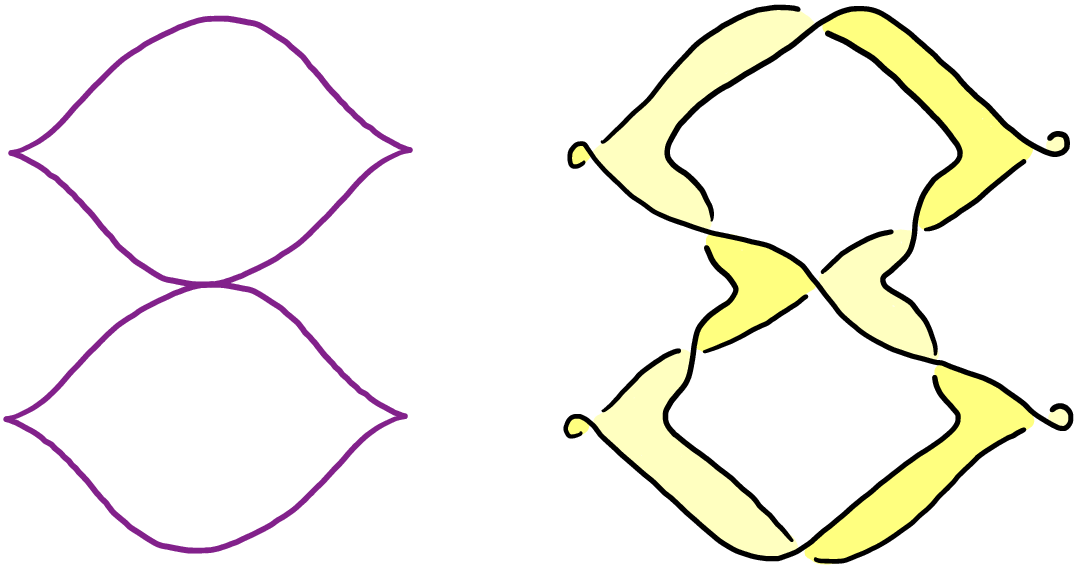}
      \put(32,45){$G$}
      \put(92,45){$R$}
    \end{overpic}
    \caption{Left: a Legendrian graph $G$. Right: a ribbon surface $R$ of $G$.}
    \label{fig:comp_ex_1}
  \end{figure}
  
  \begin{figure}[htbp]
    \centering
    \begin{overpic}[scale=1]{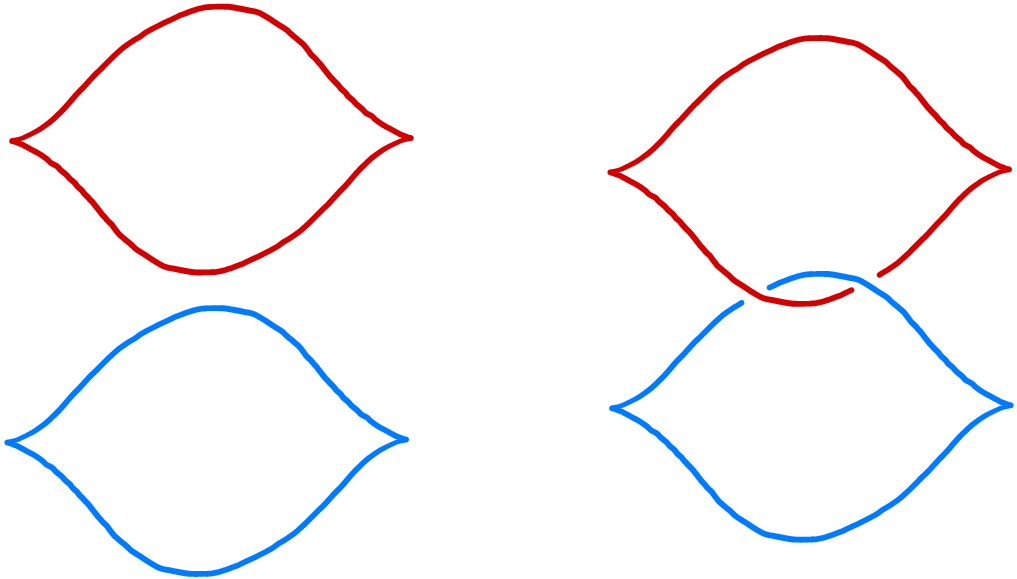}
      \put(35,50){$L_2$}
      \put(35,20){$L_1$}
      \put(95,47){$L'_1$}
      \put(95,23){$L'_2$}
    \end{overpic}
    \caption{Left: a Legendrian link $L=L_1\cup L_2$. Right: another Legendrian link $L'=L'_1\cup L'_2$.}
    \label{fig:comp_ex_2}
  \end{figure}
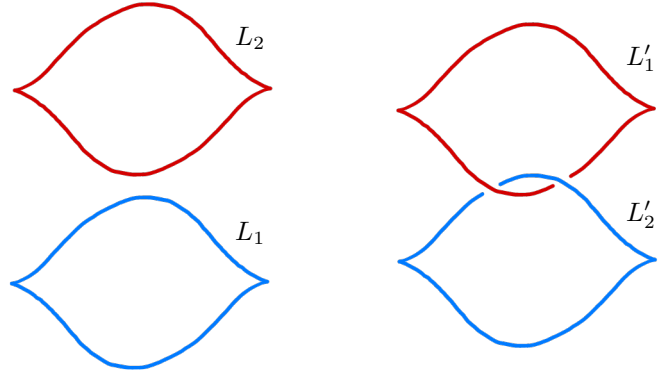

  For $L$, we can embed $L_1$ in a perturbation of $R\times\{t_1\}$ and $L_2$ in a perturbation of $R\times\{t_2\}$ with $t_2>t_1$, as illustrated in Figure~\ref{fig:comp_ex_3}.
  Since $t_2>t_1$, the surfaces $R\times\{t_2\}$ and $R\times\{t_1\}$ are disjoint and the two components of $L$ are unlinked.  
  The link $L'$ is also compatible with $R$. In this case $L'_1$ is embedded in the upper part of $R\times\{t_1\}$ and $L'_2$ in the upper part of $R\times\{t_2\}$ with $t_2>t_1$, so the two components are linked.
\end{example}

\begin{figure}[htbp]
    \centering
    \begin{overpic}[scale=1]{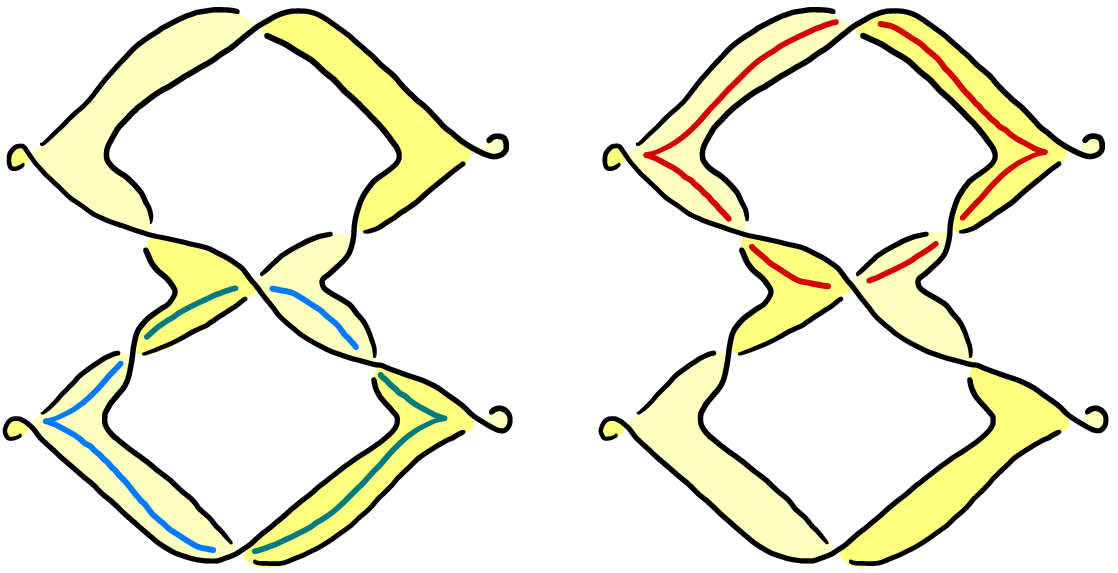}
    \end{overpic}
    \caption{Left: $L_1$ embedded in $R\times\{t_1\}$. Right: $L_2$ embedded in $R\times\{t_2\}$ with $t_2>t_1$.}
    \label{fig:comp_ex_3}
  \end{figure}

\begin{example}
  Let $R$ be a ribbon surface and let $\aa$ be a homologically nontrivial, simple closed curve on $R$. Then the Legendrian realization $\aa(0)$ of $\aa$ is compatible with $R$: by construction, $\aa(0)$ lies on a small perturbation of $R=R\times\{0\}$.  

  More generally, if $\aa_1,\dots,\aa_n$ are homologically nontrivial, simple closed curves on $R$ and $t_1<\dots<t_n$ are numbers in $(-\varepsilon,\varepsilon)$, then the Legendrian link
  \[
    L = \aa_1(t_1)\cup\cdots\cup \aa_n(t_n)
  \]
  is compatible with $R$.
\end{example}

Before we discuss the next lemma, we need a definition. Given a Legendrian graph $G$ whose vertices all have valency $4$ and whose front projections look as on the left of Figure~\ref{fig:resolutions}, a \emph{resolution} of $G$ is a connected component of the Legendrian graph obtained by replacing, at some of the vertices, the local picture by one of the models on the right of Figure~\ref{fig:resolutions}.

\begin{figure}[htbp]
  \centering
  \begin{overpic}[scale=1]{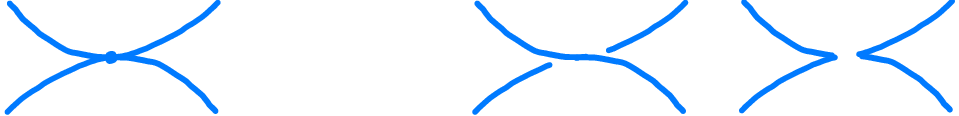}
    \put(10,10){$v$}
  \end{overpic}
  \caption{Left: the local model around a generic valency $4$ vertex $v$. Right: the two possible resolutions of $v$.}
  \label{fig:resolutions}
\end{figure}

\begin{lemma}\label{lem:resolution_comp}
  Let $G$ be a Legendrian graph whose vertices all have valency $4$ and look as in Figure~\ref{fig:resolutions} (left). Assume $G$ is embedded in a perturbation $\widetilde{R}$ of a ribbon surface $R$. Then any resolution of $G$ is also embedded in a (possibly different) perturbation of $R$.
\end{lemma}

\begin{proof}
  Since $\widetilde{R}$ is transverse to $\partial_z$, its Lagrangian projection is a local diffeomorphism. A small neighborhood of a vertex of $G$ in $\widetilde{R}$ projects as on the left of Figure~\ref{fig:resolutions_lag}.

  \begin{figure}[htbp]
    \centering
    \begin{overpic}[scale=1]{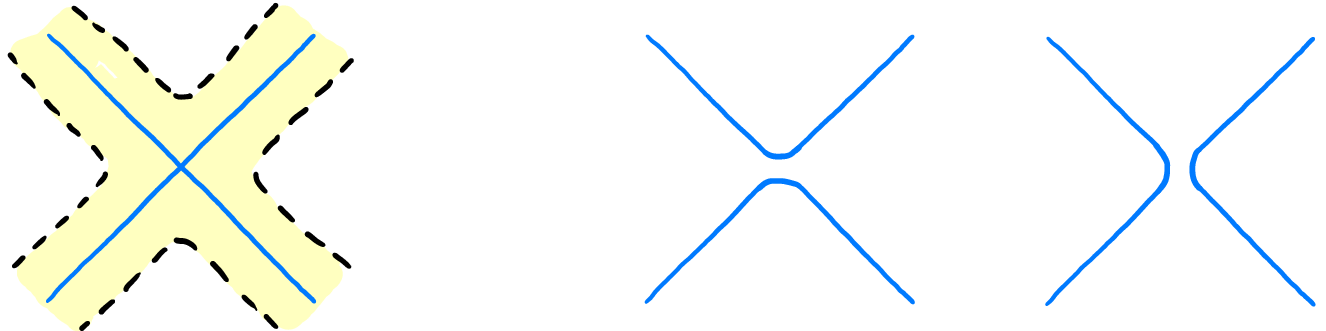}
    \end{overpic}
    \caption{Left: Lagrangian projection of a neighborhood in $\widetilde{R}$ of a vertex of $G$. Right: Lagrangian projections of the two possible resolutions.}
    \label{fig:resolutions_lag}
  \end{figure}

  The right side of Figure~\ref{fig:resolutions_lag} shows the Lagrangian projection of the possible resolutions of $G$. Since these curves in the Lagrangian projection do not (self-)intersect (in the neighborhood considered), we see that we can further perturb the surface $\widetilde{R}$ only in a neighborhood of $G$ through surfaces transverse to $\partial_z$ so that it contains the resolution.
\end{proof}

\begin{remark}\label{rmk:modi_valency_4}
  Generically, a valency $4$ vertex of a Legendrian graph in the front projection looks as in Figure~\ref{fig:Legendrian_projection}(e), but it is not guaranteed that two edges come from the right and two from the left. If $G$ is contained in a perturbation of a ribbon surface $R$, we can modify $G$ by a Legendrian Reidemeister move R (see Section~\ref{thm:Reidemeister_graph} and Figure~\ref{fig:Reidemeister_Leg_graph}) so that each valency $4$ vertex looks as in Figure~\ref{fig:resolutions} (left), without losing compatibility with $R$.

  An example is shown in Figure~\ref{fig:valency4modif}: on the top left, we see the original front projection; the bottom left shows the corresponding Lagrangian projection in the perturbation of $R$. On the top right, we see the front after the Reidemeister move R; the bottom right shows that this modification behaves well with respect to the Lagrangian projection, and so $G$ after this Legendrian isotopy can still be embedded in a perturbation of $R$.
\end{remark}

\begin{figure}[htbp]
    \centering
    \begin{overpic}[scale=1]{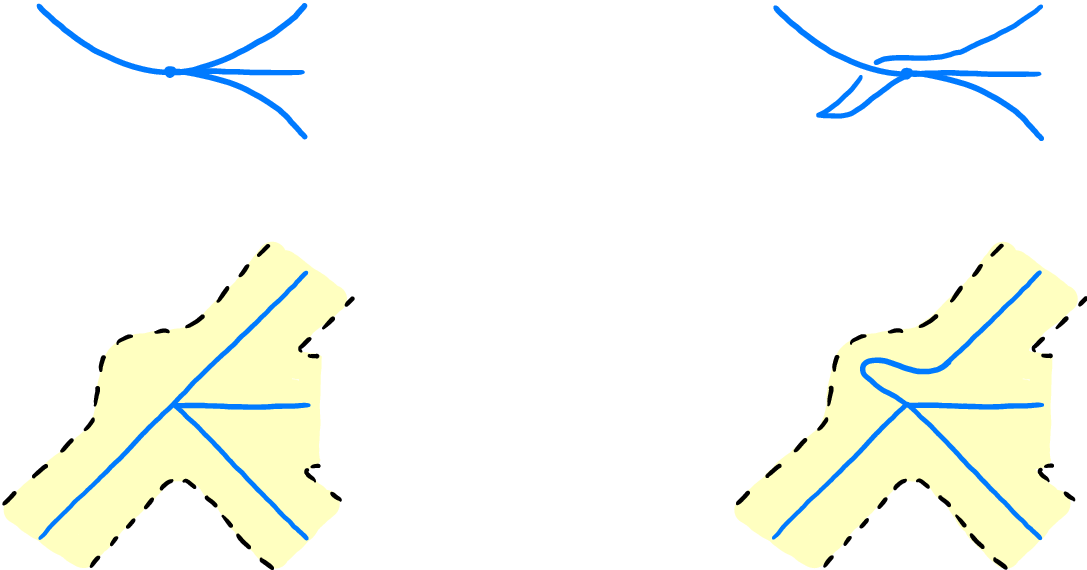}
    \end{overpic}
    \caption{Top: front projections of a valency $4$ vertex before and after a Reidemeister move R. Bottom: corresponding Lagrangian projections.}
    \label{fig:valency4modif}
  \end{figure}

\begin{corollary}\label{cor:resolution_comp}
  Any resolution of a Legendrian graph $G$ whose vertices have valency $4$ is embedded in a perturbation of any ribbon surface of $G$.
\end{corollary}

\begin{proof}
  By definition, $G$ is contained in any ribbon surface $R$ of itself. The claim follows from Lemma~\ref{lem:resolution_comp}.
\end{proof}

Another useful lemma that will be used in the proof of the main theorem is the following.

\begin{lemma}\label{compatibility_extension}
  Let $L$ be a Legendrian link compatible with a ribbon surface $R_G$ of a Legendrian graph $G$. Then $L$ is also compatible with any ribbon surface $R_{G'}$ of any Legendrian graph $G'$ containing $G$.
\end{lemma}

\begin{proof}
  A ribbon surface $R_{G'}$ can be obtained from $R_G$ by attaching embedded $0$- and $1$-handles, so $R_G$ is a subsurface of $R_{G'}$. Any perturbation of $R_G$ that contains $L$ extends (by the identity on $R_{G'}\setminus R_G$) to a perturbation of $R_{G'}$ which also contains $L$. Thus $L$ is compatible with $R_{G'}$.
\end{proof}

\section{Contact handlebodies and half open books}\label{sec:contact_handlebodies}

This section is a digression: the material here is not strictly needed for the proof of Theorem~\ref{thm:main}, but it provides a useful conceptual picture. It explains how ribbon surfaces arise naturally as pages of certain ``half open books'' on contact handlebodies. The ideas come from \cite{Giroux91,Giroux02} and are at the core of many proofs in \cite{Avdek13}. In this section, we drop the standing assumption that we work in $(\R^3,\xist)$.

\begin{definition}
  Let $H$ be a handlebody (i.e.\ a $3$-ball with $3$-dimensional $1$–handles attached) in a $3$-manifold $M$. A \emph{half open book on $H$} is a pair $(B,\pi)$, where
  \begin{itemize}
    \item $B$ is an oriented link in $\partial H$, and
    \item $\pi\colon H\setminus B\to[-\varepsilon,\varepsilon]$ is a fibration whose fibers are the interiors of Seifert surfaces of $B$.
  \end{itemize}
  By analogy with open books, $B$ is called the \emph{binding}, and the closures of the fibers are called \emph{pages}.
\end{definition}

\begin{example}
  Let $(B,\pi)$ be an open book on a closed $3$-manifold $M$, with $\pi\colon M\setminus B\to S^1=\R/\Z$. For $\varepsilon<\tfrac{1}{2}$, the subset
  \[
    H := \pi^{-1}([-\varepsilon,\varepsilon])
  \]
  is a handlebody, and the restriction of $(B,\pi)$ to $H$ is a half open book.
\end{example}

\begin{remark}
  In general, a half open book $(B,\pi)$ on a handlebody $H\subset M$ does not extend to an open book of the whole manifold $M$. For such an extension to exist, the closure of $M\setminus H$ must also be a handlebody carrying a half open book structure with the same binding $B$.
\end{remark}

\begin{definition}
  Let $(H,\xi)$ be a contact handlebody (in the sense of~\cite{Giroux91}). A half open book $(B,\pi)$ on $H$ is said to \emph{support} $\xi$ if
  \begin{itemize}
    \item $B$ is positively transverse to $\xi$; and
    \item there exists a contact form $\lambda$ for $\xi$ such that $d\lambda$ restricts to a positive area form on the interior of each page.
  \end{itemize}
  In this case, we also say that $(H,\xi)$ is \emph{supported} by $(B,\pi)$.
\end{definition}

As in the case of an open book supporting a contact structure, the pages of a half open book supporting $(H,\xi)$ are convex surfaces whose dividing set can be taken to be the boundary of the page. Moreover, we may assume that the dividing set of $\partial H$ is the binding $B$.

\begin{remark}
  Any two half open books supporting a given contact handlebody are isotopic (through half open books supporting the same contact structure).
\end{remark}

If $G$ is a compact connected Legendrian graph in a contact manifold $(M,\xi)$, then a standard regular neighborhood $H$ of $G$ (with the restriction of $\xi$) is a contact handlebody.

\begin{example}\label{half_ob_ribbon}
  Let $\Sigma$ be a ribbon surface of a Legendrian graph $G$ in $(M,\xi)$ (Definition~\ref{def:ribbon}). Using the flow of the Reeb vector field $R_\alpha$ appearing in the definition, we can construct a neighborhood $U$ of $G$ contactomorphic to $(\Sigma\times[-\varepsilon,\varepsilon], dz + \alpha|_{T\Sigma})$, where the $z$–coordinate corresponds to the flow lines of $R_\alpha$; see Figure~\ref{fig:ribbonOB} (left and middle).

  \begin{figure}[htbp]
    \centering
    \begin{overpic}[scale=1]{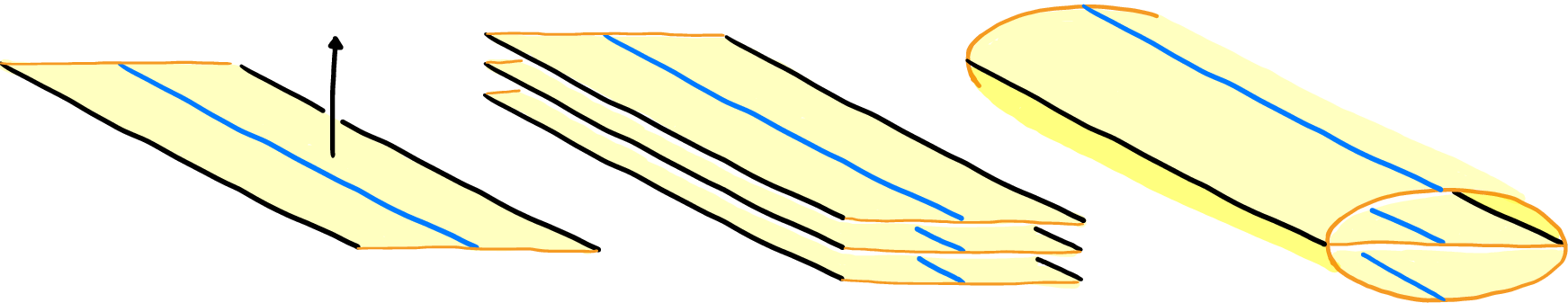}
      \put(22,13){$R_\alpha$}
      \put(31,1){$G$}
    \end{overpic}
    \caption{Left: a ribbon surface $\Sigma$. Middle: the product $\Sigma\times[-\varepsilon,\varepsilon]$. Right: a regular neighborhood $H$ of $G$ equipped with a half open book structure.}
    \label{fig:ribbonOB}
  \end{figure}
  
  \begin{figure}[htbp]
    \centering
    \begin{overpic}[scale=1]{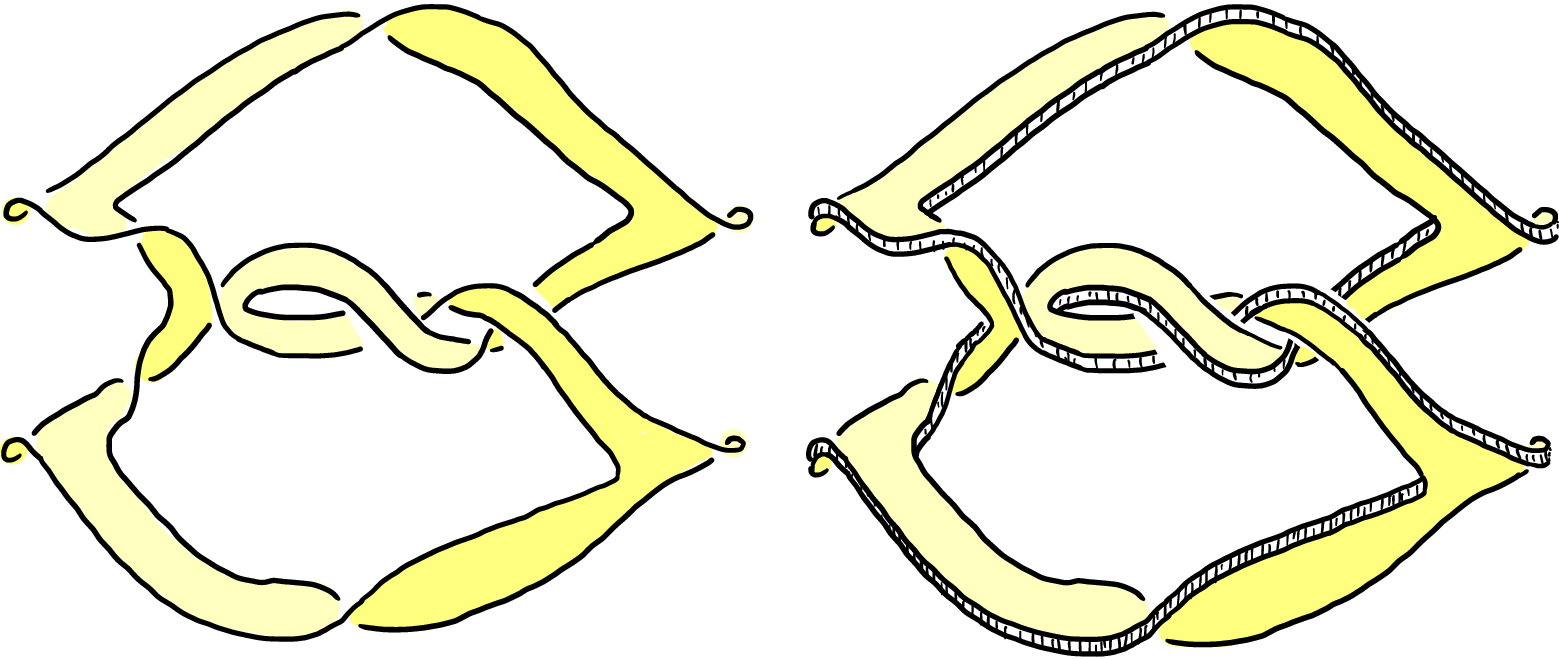}
    \end{overpic}
    \caption{Left: a ribbon surface $R$. Right: the product $R\times[-\varepsilon,\varepsilon]$, viewed as a neighborhood of the Legendrian graph.}
    \label{fig:ribbonOB2}
  \end{figure}
  
  By rounding the edges of $U$, we obtain a standard regular neighborhood $H$ of $G$, which can be seen as
  \[
    H \cong \Sigma\times[-\varepsilon,\varepsilon]/\!\sim,
  \]
  where $(x,z)\sim(x,z')$ if $x\in\partial\Sigma$ (Figure~\ref{fig:ribbonOB}, right). The projection to the second factor defines a half open book $(B,\pi)$ on $H$ whose pages are the images of $\Sigma\times\{t\}$ and whose binding is the image of $\partial\Sigma\times[-\varepsilon,\varepsilon]$. It is straightforward to check that this half open book supports the given contact structure.
Figure~\ref{fig:ribbonOB2} shows part of this construction in $(S^3,\xist)$.
\end{example}

In particular, every ribbon surface appears as a page of a half open book on a contact handlebody.

\begin{theorem}[{\cite{Giroux02}}]\label{thm:contact_cell}
    A ribbon surface $R_G$ on a contact manifold $(M,\xi)$ is the page of an open book compatible with $\xi$ if $G$ is the $1$--skeleton of a \emph{contact cell decomposition}. \qed
\end{theorem}

We refer to~\cite{Etnyre06} for the definition of a contact cell decomposition. 
An important point is that, when working in $(\R^3,\xist)$, it is particularly easy to determine when a Legendrian graph is the $1$--skeleton of such a decomposition (see~\cite{Avdek13}). 
Theorem~\ref{thm:contact_cell} is at the heart of the proof of the main result in \cite[Theorem 1.10]{Avdek13}, which will be our starting point.

\section{Contact surgery links and ribbon surfaces}\label{sec:ribbon_moves_sec}

We now explain how a contact surgery link compatible with a ribbon surface $R$ determines a product of Dehn twists in the mapping class group $\mathrm{MCG}(R,\partial R)$.

Before the next definition, we make a brief convention. Let $L$ be a contact surgery link compatible with a ribbon surface $R$. Since each component $L_i$ lies in a perturbation of some $R\times\{t_i\}$, we can regard $L_i$ as a simple closed curve on $R$, well-defined up to isotopy: pull it back to $R\times\{t_i\}$ via the inverse of the perturbation, then project to $R$.
We will freely switch between these two points of view (Legendrian knot in $\R^3$ vs.\ isotopy class of a curve on $R$); the meaning should always be clear from the context.\footnote{Here, regarding $R$ as an abstract surface is analogous to the familiar process of viewing a page of an open book as the surface of the corresponding abstract open book.}

\begin{convention}
  Given a Legendrian knot $L$ compatible with a ribbon surface $R$, we denote by $\overline{L}$ the simple closed curve on $R$ defined by $L$ up to isotopy.
\end{convention}

\begin{definition}[{\cite[Definition~6.1]{Avdek13}}]\label{def:mapping_class_determined}
  Let $R$ be a ribbon surface and
  \[
    \boldsymbol{L}
    =
    \bigcup_{i=1}^n L_i^{\delta_i}, 
    \qquad
    \delta_i\in\{+,-\},
  \]
  be a contact surgery link compatible with $R$.  
The \emph{mapping class determined by $\boldsymbol{L}$} is
  \[
    \tau_{\scriptscriptstyle{\boldsymbol{L}}}
    :=
    \tau_{\scriptscriptstyle{\overline{L_n}}}^{-\delta_n}
    \cdots
    \tau_{\scriptscriptstyle{\overline{L_1}}}^{-\delta_1}
    \in \mathrm{MCG}(R,\partial R).
  \]
\end{definition}

Note that a contact surgery link compatible with a ribbon surface determines more than just an element of the mapping class group. It provides a specific
factorization in terms of Dehn twists, where each Dehn twist corresponds to a component of the contact surgery link. Since, in the following sections, the actual
factorization will be important; we sometimes think of
the “mapping class determined” by a contact surgery link as the data of the
factorization.
The motivation for Definition~\ref{def:mapping_class_determined} is the following well-known result.

\begin{proposition}[{\cite{Gay,Etnyre06}}]\label{prop:Dehn_surgery_Dehn_twist}
  Let
  \[
    \boldsymbol{L}
    =
    \bigcup_{i=1}^n L_i^{\delta_i},
    \qquad
    \delta_i\in\{+,-\},
  \]
  be a contact surgery link whose components are embedded in different pages of an open book $(B,\pi)$ supporting a contact structure $\xi$ in a manifold $M$. Let $(\Sigma,h)$ be the corresponding abstract open book, where $\Sigma$ is a page chosen just before the page containing $L_1$. Then the contact manifold obtained by contact surgery on $\boldsymbol{L}$ is contactomorphic to the contact manifold supported by the abstract open book $(\Sigma, h\circ\tau_{\scriptscriptstyle{\boldsymbol{L}}})$.\qed
\end{proposition}

\section{Ribbon moves}\label{sec:ribbonmovesadvek}

In this section, we recall the definition of ribbon moves given in \cite{Avdek13}, and we state the main result there, namely Theorem~\ref{Vague_Kirby_Theorem}.

\begin{definition}[{\cite[Definition~6.2]{Avdek13}}]\label{RequivalentV1}
  Let $\boldsymbol{L}$ and $\boldsymbol{L}'$ be contact surgery links in $(S^3,\xist)$. We say that $\boldsymbol{L}$ and $\boldsymbol{L}'$ are \emph{$R$–equivalent} if
  \begin{itemize}
    \item both links are compatible with the same ribbon surface $R$, and
    \item $\tau_{\scriptscriptstyle{\boldsymbol{L}}}=\tau_{\scriptscriptstyle{\boldsymbol{L}'}}$ in $\mathrm{MCG}(R,\partial R)$.
  \end{itemize}
\end{definition}

\begin{example}\label{ex:R-eq}
  Consider the two contact surgery links $\boldsymbol{L}=L_1^-\cup L_2^-$ on the left and $\boldsymbol{L}'=L_1'^-\cup L_2'^-$ on the right of Figure~\ref{fig:ex_R_eq1}. In Example~\ref{compatible_braid1} we already saw that $\boldsymbol{L}$ is compatible with the ribbon $R$ from Figure~\ref{fig:comp_ex_1}; the same is true for $\boldsymbol{L}'$ (note that in this case, the component $L'_2$ is a resolution of $G$ and so is compatible with $R$ because of Corollary~\ref{cor:resolution_comp}). As the indices suggest, we have that $L'_2\subset R\times\{t_2\}$ and $L'_1\subset R\times\{t_1\}$, with $t_2>t_1$.
  We have that
  \[
    \tau_{\scriptscriptstyle{\boldsymbol{L}}}
    = \tau_{\scriptscriptstyle{\overline{L_2}}}^+ \tau_{\scriptscriptstyle{\overline{L_1}}}^+,
    \qquad
    \tau_{\scriptscriptstyle{\boldsymbol{L}'}}
    = \tau_{\scriptscriptstyle{\overline{L'_2}}}^+ \tau_{\scriptscriptstyle{\overline{L'_1}}}^+.
  \]
  The ribbon $R$ can be identified with the surface in Figure~\ref{fig:braid_relation} so that $\overline{L_1}\mapsto\bb$, $\overline{L_2}\mapsto\aa$, $\overline{L'_1}\mapsto\aa$, and $\overline{L'_2}\mapsto\aa\bb$. Under this identification, the two factorizations differ by a braid relation, hence
  \[
    \tau_{\scriptscriptstyle{\boldsymbol{L}}}
    = \tau_{\scriptscriptstyle{\boldsymbol{L}'}} \in \mathrm{MCG}(R,\partial R),
  \]
  and $\boldsymbol{L},\boldsymbol{L}'$ are $R$–equivalent.
\end{example}

\begin{figure}[htbp]
    \centering
    \begin{overpic}[scale=1]{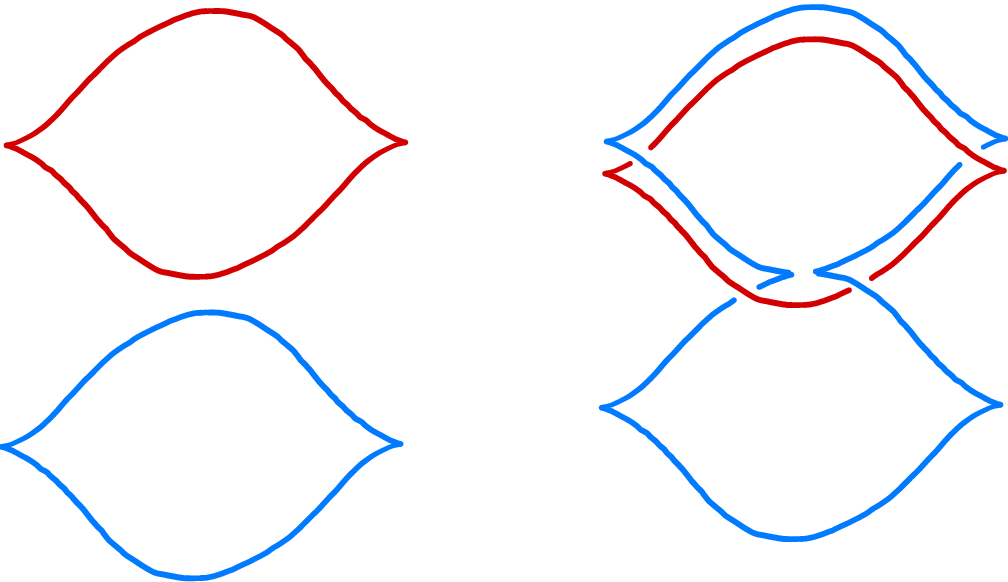}
      \put(34,53){$L_2$}
      \put(34,20){$L_1$}
      \put(41,42.5){$-$}
      \put(41,12.5){$-$}
      \put(92,53){$L'_2$}
      \put(92,30){$L'_1$}
      \put(101,42.5){$-$}
      \put(101,40){$-$}
    \end{overpic}
    \caption{Left: the contact surgery link $\boldsymbol{L}=L_1^-\cup L_2^-$. Right: the contact surgery link $\boldsymbol{L}'=L_1'^-\cup L_2'^-$.}
    \label{fig:ex_R_eq1}
  \end{figure}

\begin{lemma}\label{lem:Req_after_enlarg}
  Let $\boldsymbol{L}$ and $\boldsymbol{L}'$ be two $R_G$–equivalent contact surgery links. Then $\boldsymbol{L}$ and $\boldsymbol{L}'$ are also $R_{G'}$–equivalent for every Legendrian graph $G'$ containing $G$.
\end{lemma}

\begin{proof}
  By Lemma~\ref{compatibility_extension}, $\boldsymbol{L}$ and $\boldsymbol{L}'$ are both compatible with $R_{G'}$. Moreover, viewing $R_G$ as a subsurface of $R_{G'}$, any diffeomorphism of $R_G$ fixing $\partial R_G$ extends by the identity on $R_{G'}\setminus R_G$, yielding a homomorphism
  \[
    \mathrm{MCG}(R_G,\partial R_G)
    \longrightarrow
    \mathrm{MCG}(R_{G'},\partial R_{G'}).
  \]
  If $\tau_{\scriptscriptstyle{\boldsymbol{L}}}
  =\tau_{\scriptscriptstyle{\boldsymbol{L}'}}$ in $\mathrm{MCG}(R_G,\partial R_G)$, the same holds in $\mathrm{MCG}(R_{G'},\partial R_{G'})$, so $\boldsymbol{L}$ and $\boldsymbol{L}'$ are $R_{G'}$–equivalent.
\end{proof}

\begin{definition}[{\cite[Definition~6.4]{Avdek13}}]\label{def:ribbon_moves}
  Let $\boldsymbol{L}\subset(S^3,\xist)$ be a contact surgery link and $\boldsymbol{l}$ a contact surgery sublink. Suppose $\boldsymbol{l}$ is $R$–equivalent to another contact surgery link $\boldsymbol{l}'$.  
  A \emph{ribbon move} performed on $\boldsymbol{l}$ consists of replacing $\boldsymbol{l}$ by $\boldsymbol{l}'$ in $\boldsymbol{L}$.
\end{definition}

\begin{remark}
Definition~\ref{def:ribbon_moves} is stated in the form used in~\cite{Avdek13}. Strictly speaking, however, when a ribbon move is performed on a sublink
\[
\boldsymbol l\subset \boldsymbol L,
\]
the \(R\)-equivalence between \(\boldsymbol l\) and \(\boldsymbol l'\) should take place in the complement of the remaining components of \(\boldsymbol L\). In other words, one should require the ribbon surface \(R\), and the corresponding neighborhood
\[
R\times[-\varepsilon,\varepsilon],
\]
to be disjoint from \(\boldsymbol L\setminus \boldsymbol l\). Equivalently, the sublinks \(\boldsymbol l\) and \(\boldsymbol l'\) are regarded as \(R\)-equivalent inside
\[
\bigl(S^3\setminus(\boldsymbol L\setminus \boldsymbol l),
\xist|_{S^3\setminus(\boldsymbol L\setminus \boldsymbol l)}\bigr).
\]

Informally, this condition says that the components of \(\boldsymbol L\) not involved in the move should not interfere with the \(R\)-equivalence. This is also the reason for the shaded regions in the diagrammatic moves shown in Figures~\ref{fig:cancelling_intro}--\ref{fig:chain_intro}: the shaded region is the front projection of the supporting handlebody $R\times [-\varepsilon,\varepsilon]$, which should be disjoint from all components of the contact surgery link except those involved in the move.
\end{remark}

\begin{example}
  Consider the contact surgery links $\boldsymbol{L}:=L_1^-\cup L_2^-\cup K^+$ and $\boldsymbol{L}':=L_1'^-\cup L_2'^-\cup K^+$ in Figure~\ref{fig:ribbon_move_ex}.
By Example~\ref{ex:R-eq}, the sublinks $L_1^-\cup L_2^-$ and $L_1'^-\cup L_2'^-$ are $R$–equivalent, so $\boldsymbol{L}$ and $\boldsymbol{L}'$ differ by a ribbon move.
\end{example}

  \begin{figure}[htbp]
    \centering
    \begin{overpic}[scale=1]{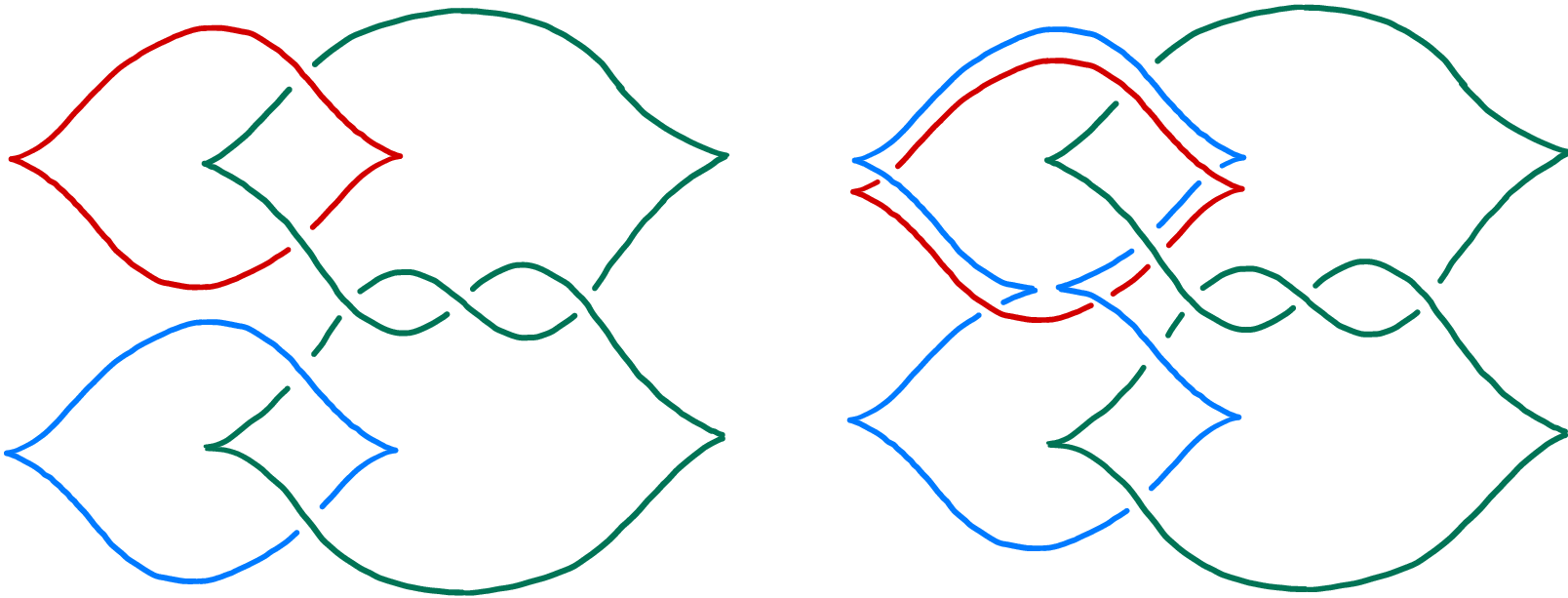}
      \put(26,27.5){$-$}
      \put(26,8.5){$-$}
      \put(47,27.5){$+$}
      \put(80,27.5){$-$}
      \put(80,25.5){$-$}
      \put(101,27.5){$+$}
      \put(2,23){$L_2$}
      \put(2,14){$L_1$}
      \put(54,21){$L'_1$}
      \put(54,15){$L'_2$}
      \put(41,32){$K$}
      \put(95,32){$K$}
    \end{overpic}
    \caption{Left: the contact surgery link $\boldsymbol{L}=L_1^-\cup L_2^-\cup K^+$. Right: the contact surgery link $\boldsymbol{L}'=L_1'^-\cup L_2'^-\cup K^+$.}
    \label{fig:ribbon_move_ex}
  \end{figure}

The point of introducing ribbon moves is the following result.

\begin{proposition}[{\cite[Proposition~6.3]{Avdek13}}]\label{ribbon_moves_prop}
  Suppose that a contact surgery link $\boldsymbol{L}$ in $(S^3,\xist)$ can be transformed into another contact surgery link $\boldsymbol{L}'$ by a sequence of ribbon moves. Then $\boldsymbol{L}$ and $\boldsymbol{L}'$ describe contactomorphic contact manifolds.\qed
\end{proposition}

The main result of~\cite{Avdek13} is a converse to Proposition~\ref{ribbon_moves_prop}, which we will use as the starting point for our refinement.

\begin{theorem}[{\cite[Theorem~1.10]{Avdek13}}]\label{Vague_Kirby_Theorem}
  Let $\boldsymbol{L}$ and $\boldsymbol{L}'$ be contact surgery links in $(S^3,\xist)$ that describe contactomorphic contact manifolds. Then there exists a sequence of ribbon moves and Legendrian isotopies transforming $\boldsymbol{L}$ into $\boldsymbol{L}'$.\qed
\end{theorem}

\section{Elementary \texorpdfstring{$R$}{R}-equivalences}\label{sec:elementary_R_equivalences}

We now isolate those $R$–equivalences whose mapping class factorizations differ by a \emph{single} relation in the presentation of the mapping class group (Theorem~\ref{thm:Gervais}). These will correspond to the moves appearing in our contact version of Kirby's theorem.

\begin{definition}
  Let $\boldsymbol{L}$ and $\boldsymbol{L}'$ be two $R$–equivalent contact surgery links, such that the factorizations of the mapping classes $\tau_{\scriptscriptstyle{\boldsymbol{L}}}$ and $\tau_{\scriptscriptstyle{\boldsymbol{L}'}}$ (as in Definition~\ref{def:mapping_class_determined}) differ by one of the following:
  \begin{itemize}
    \item insertion or deletion of a pair of cancelling Dehn twists;
    \item commuting two consecutive Dehn twists along disjoint curves;
    \item a braid relation;
    \item a lantern relation; or
    \item a chain relation.
  \end{itemize}
  Then we say that $\boldsymbol{L}$ and $\boldsymbol{L}'$ differ by an \emph{elementary $R$–equivalence}.
If two contact surgery links $\boldsymbol{L}$ and $\boldsymbol{L}'$ differ by a ribbon move performed on sublinks $\boldsymbol{l}\subset\boldsymbol{L}$ and $\boldsymbol{l}'\subset\boldsymbol{L}'$ which are elementary $R$–equivalent, we say that $\boldsymbol{L}$ and $\boldsymbol{L}'$ differ by an \emph{elementary ribbon move}.
\end{definition}

For instance, the two links in Example~\ref{ex:R-eq} differ by an elementary $R$–equivalence corresponding to a braid relation.
We continue by presenting several examples of elementary \(R\)-equivalences. These examples, in the more general context where other components might also be involved, will constitute the moves used in Theorem~\ref{thm:main}.

\subsection{Addition or deletion of cancelling Dehn twists}\label{cancellationEX}
  Consider two Legendrian knots that are vertical push–offs of each other in the $\partial_z$-direction, decorated with opposite surgery coefficients, as in Figure~\ref{fig:cancelling_intro}. Denote this contact surgery link by $\boldsymbol{L}$.
  It is easy to see that $\boldsymbol{L}$ is compatible with the Legendrian graph $G$ in Figure~\ref{fig:cancellationR} and any of its ribbon surfaces. The Legendrian graph $G$ is just a copy of either of the two Legendrian knots in Figure~\ref{fig:cancelling_intro}.

  \begin{figure}[htbp]
    \centering
    \begin{overpic}[scale=1]{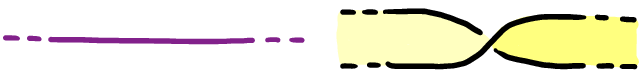}
      \put(30,8){$G$}
    \end{overpic}
    \caption{Left: a Legendrian graph $G$ which is a copy of either component of Figure~\ref{fig:cancelling_intro}. Right: a ribbon $R$ of $G$.}
    \label{fig:cancellationR}
  \end{figure}

  The mapping class associated with $\boldsymbol{L}$ is the product of two cancelling Dehn twists, hence trivial in $\mathrm{MCG}(R,\partial R)$ (see Example~\ref{ex:cancellation_DT}). The empty contact surgery link is also compatible with $R$ and also determines the identity in $\mathrm{MCG}(R,\partial R)$. Thus $\boldsymbol{L}$ and the empty diagram differ by an elementary $R$–equivalence, and we have recovered the \emph{standard cancelling pair} as in Figure~\ref{fig:cancelling_intro} from the introduction. The interpretation here via mapping class factorizations is due to Avdek~\cite{Avdek13}.

\subsection{Braid relation}\label{ex:braidLeg}
  Let $\boldsymbol{L}=L_1^-\cup L_2^-$ be a contact surgery link. Up to Legendrian isotopy, we may assume that there is a ball in which $\boldsymbol{L}$ looks as on the left of Figure~\ref{fig:braidLeg1}. Let $\boldsymbol{L}'=L_1'^-\cup L_2'^-$ be as on the right of Figure~\ref{fig:braidLeg1}. Both links are compatible with the ribbon $R$ shown in Figure~\ref{fig:braidLegR}.
We have
  \[
    \tau_{\scriptscriptstyle{\boldsymbol{L}}}
    = \tau_{\scriptscriptstyle{\overline{L_2}}}^+ \tau_{\scriptscriptstyle{\overline{L_1}}}^+,
    \qquad
    \tau_{\scriptscriptstyle{\boldsymbol{L}'}}
    = \tau_{\scriptscriptstyle{\overline{L'_2}}}^+ \tau_{\scriptscriptstyle{\overline{L'_1}}}^+.
  \]
  Identifying $R$ with the surface in Figure~\ref{fig:braid_relation} so that $\overline{L_2}$ and $\overline{L'_1}$ map to $\aa$, $\overline{L_1}$ to $\bb$, and $\overline{L'_2}$ to $\aa\bb$, we see that the two factorizations differ by a braid relation (see Example~\ref{ex:braid_relations}). Hence $\boldsymbol{L}$ and $\boldsymbol{L}'$ differ by an elementary $R$–equivalence.
A similar argument shows that the two contact surgery links on the left and on the right of Figure~\ref{fig:braidLeg2} also differ by an elementary $R$-equivalence.

The two $R$-equivalences described in Section~\ref{ex:braidLeg} and shown in Figures~\ref{fig:braidLeg1} and~\ref{fig:braidLeg2} are the \emph{standard contact handle slides} from the introduction. This way of presenting them again comes from~\cite{Avdek13}.

 \begin{figure}[htbp]
    \centering
    \begin{overpic}[scale=1]{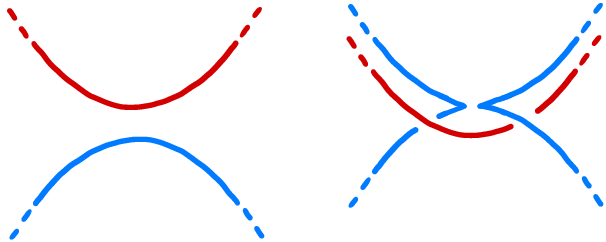}
      \put(35,25){$-$}
      \put(35,12){$-$}
      \put(10,30){$L_2$}
      \put(10,5){$L_1$}
      \put(95,12){$-$}
      \put(95,25){$-$}
      \put(55,20){$L'_1$}
      \put(65,32){$L'_2$}
    \end{overpic}
    \caption{Left: the contact surgery link $\boldsymbol{L}=L_1^-\cup L_2^-$. Right: the contact surgery link $\boldsymbol{L}'=L_1'^-\cup L_2'^-$.}
    \label{fig:braidLeg1}
  \end{figure}

  \begin{figure}[htbp]
    \centering
    \begin{overpic}[scale=1]{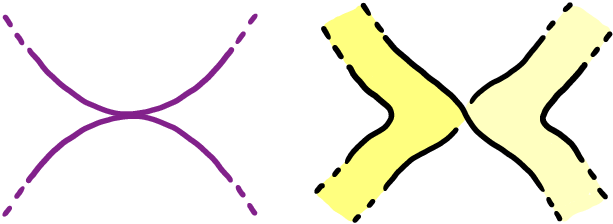}
      \put(18,5){$G$}
      \put(73,5){$R$}
    \end{overpic}
    \caption{Left: the Legendrian graph $G$ underlying the braid relation. Right: a ribbon surface $R$ of $G$.}
    \label{fig:braidLegR}
  \end{figure}
  
  \begin{figure}[htbp]
    \centering
    \begin{overpic}[scale=1]{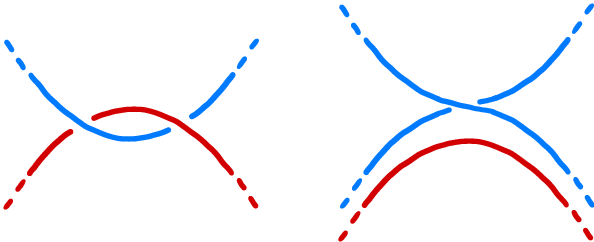}
      \put(38,23){$-$}
      \put(38,14){$-$}
      \put(85,6){$-$}
      \put(92,25){$-$}
    \end{overpic}
    \caption{Another pair of contact surgery links differing by an elementary $R$–equivalence corresponding to a braid relation.}
    \label{fig:braidLeg2}
  \end{figure}
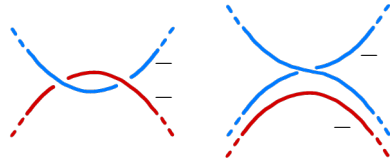
  
\subsection{Lantern relation}\label{ex:lanternLeg}
  Let $\boldsymbol{L}=L_1^-\cup L_2^-\cup L_3^-$ be a contact surgery link. Suppose that there are two balls in which $\boldsymbol{L}$ looks as in the left column of Figure~\ref{fig:ex_lant_Leg1}. The arrows in the figure indicate that if we orient $L_2$ and $L_3$ as in the top ball, we want these orientations to match those drawn in the bottom ball.

  \begin{figure}[htbp]
    \centering
    \begin{overpic}[scale=1]{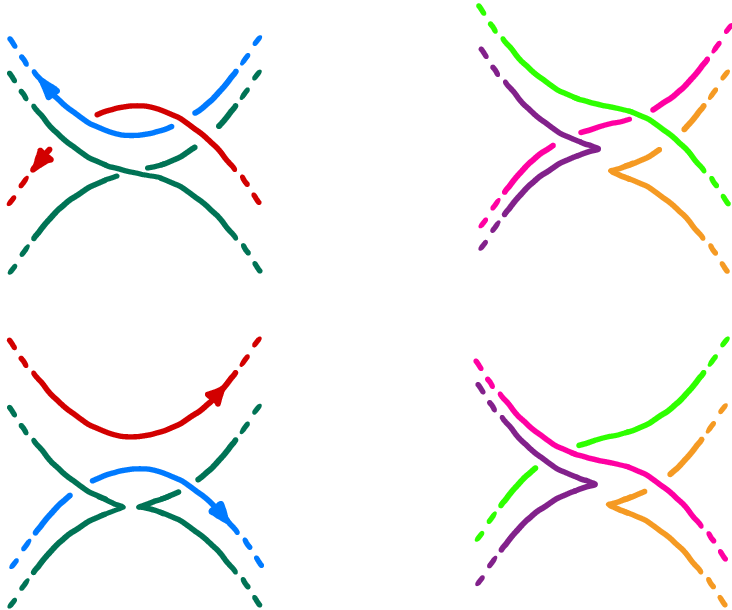}
      \put(23,5){$L_1$}
      \put(23,50){$L_1$}
      \put(32,12){$L_2$}
      \put(25,75){$L_2$}
      \put(32,61){$L_3$}
      \put(25,33){$L_3$}
      \put(-5,78){$-$}
      \put(-5,72){$-$}
      \put(-5,54){$-$}
      \put(59,81){$-$}
      \put(59,76){$-$}
      \put(59,51){$-$}
      \put(59,47){$-$}
      \put(73,74){$L_4'$}
      \put(86,73){$L_3'$}
      \put(71,50){$L_2'$}
      \put(84,50){$L_1'$}
    \end{overpic}
    \caption{Left: the contact surgery link $\boldsymbol{L}=L_1^-\cup L_2^-\cup L_3^-$. Right: the contact surgery link $\boldsymbol{L}'=L_1'^-\cup L_2'^-\cup L_3'^-\cup L_4'^-$ produced by a lantern move.}
    \label{fig:ex_lant_Leg1}
  \end{figure}
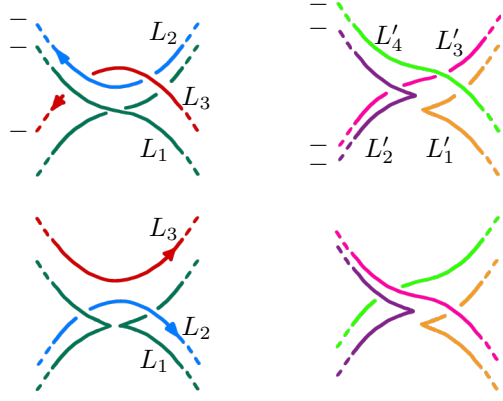

  Let $\boldsymbol{L}' = L_1'^-\cup L_2'^-\cup L_3'^-\cup L_4'^-$ be as in the right column of Figure~\ref{fig:ex_lant_Leg1}. As usual, both links are compatible with the ribbon surface $R$ depicted in Figure~\ref{fig:ex_lant_Leg2}, for instance by using Corollary~\ref{cor:resolution_comp}.
    Their associated mapping classes are
  \[
    \tau_{\scriptscriptstyle{\boldsymbol{L}}}
    = \tau_{\scriptscriptstyle{\overline{L_3}}}^+
      \tau_{\scriptscriptstyle{\overline{L_2}}}^+
      \tau_{\scriptscriptstyle{\overline{L_1}}}^+,
    \qquad
    \tau_{\scriptscriptstyle{\boldsymbol{L}'}}
    = \tau_{\scriptscriptstyle{\overline{L_4'}}}^+
      \tau_{\scriptscriptstyle{\overline{L_3'}}}^+
      \tau_{\scriptscriptstyle{\overline{L_2'}}}^+
      \tau_{\scriptscriptstyle{\overline{L_1'}}}^+.
  \]
  Identifying $R$ with the surface in Figure~\ref{fig:lantern_relation} so that
  \[
    \overline{L_1}\mapsto \aa\bb,\quad
    \overline{L_2}\mapsto \bb,\quad
    \overline{L_3}\mapsto \aa, \quad\textrm{and}\quad
    \overline{L_1'},\overline{L_2'},\overline{L_3'},\overline{L_4'}
    \mapsto \dd_1,\dd_2,\dd_3,\dd_4,
  \]
  we see that the two factorizations differ by a lantern relation (see Example~\ref{ex:lantern_relations}). Hence $\boldsymbol{L}$ and $\boldsymbol{L}'$ differ by an elementary $R$–equivalence.
We call this $R$-equivalence a \emph{lantern move}. Similar $R$-equivalences, also called lantern moves, are collected in Section~\ref{sec:moves} and the standard lantern move (Figure~\ref{fig:lantern_intro}) is a special case.

\begin{figure}[htbp]
    \centering
    \begin{overpic}[scale=1]{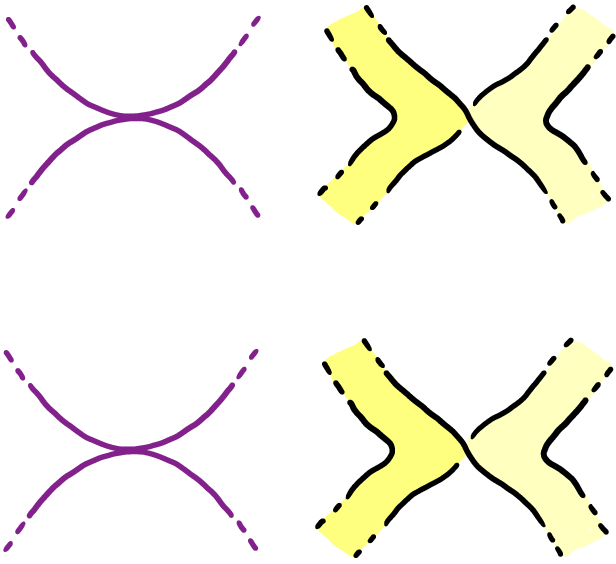}
    \end{overpic}
    \caption{A ribbon surface $R$ compatible with the contact surgery links appearing in the lantern move.}
    \label{fig:ex_lant_Leg2}
  \end{figure}

\subsection{Chain relation}\label{ex:chainLeg}
  We consider the contact surgery links
  \[
    \boldsymbol{L}= \bigcup_{i=1}^{12} L_i^-,
    \qquad
    \boldsymbol{L}'=L_1'^-\cup L_2'^-
  \]
  whose local pictures are as in Figure~\ref{fig:ex_chain_Leg1}. It is straightforward to check that $\boldsymbol{L}$ and $\boldsymbol{L}'$ are compatible with the ribbon surface $R$ in Figure~\ref{fig:ex_chain_R}. (For $\boldsymbol{L}'$, note that $L_1'$ and $L_2'$ are the Legendrian realizations of the two boundary components of $R$ given by the algorithm of~\cite{StenhedeAlgorithm}.)
 The associated mapping classes are
  \[
    \tau_{\scriptscriptstyle{\boldsymbol{L}}}
    = \prod_{i=1}^{12}\tau_{\scriptscriptstyle{\overline{L_i}}}^+,
    \qquad
    \tau_{\scriptscriptstyle{\boldsymbol{L}'}}
    = \tau_{\scriptscriptstyle{\overline{L_2'}}}^+\tau_{\scriptscriptstyle{\overline{L_1'}}}^+.
  \]
  Identifying $R$ with the surface in Figure~\ref{fig:chain_relation} so that
  \[
    \overline{L_1},\overline{L_4},\overline{L_7},\overline{L_{10}}\mapsto\cc,\quad
    \overline{L_2},\overline{L_5},\overline{L_8},\overline{L_{11}}\mapsto\bb,\quad
    \overline{L_3},\overline{L_6},\overline{L_9},\overline{L_{12}}\mapsto\aa,\quad\textrm{and}\quad
    \overline{L_1'},\overline{L_2'}\mapsto \dd_1,\dd_2,
  \]
  we see that the two factorizations differ by a chain relation (see Example~\ref{ex:chain_relations}). Hence $\boldsymbol{L}$ and $\boldsymbol{L}'$ differ by an elementary $R$–equivalence.
We call this $R$-equivalence a \emph{chain move}. Similar $R$-equivalences, also called chain moves, are collected in Section~\ref{sec:moves} and the standard chain move (Figure~\ref{fig:chain_intro}) is a special case.

 \begin{figure}[htbp]
    \centering
    \begin{overpic}[scale=1]{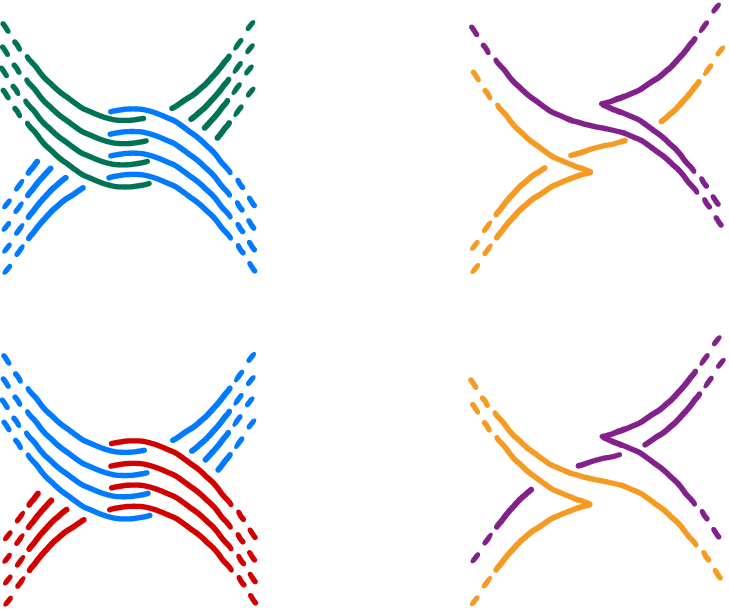}
      \put(-6,79){$-$}
      \put(-6,76){$-$}
      \put(-6,73){$-$}
      \put(-6,70){$-$}
      \put(36,70){$\rightarrow L_1$}
      \put(-6,33){$-$}
      \put(-6,30){$-$}
      \put(-6,27){$-$}
      \put(-6,24){$-$}
      \put(36,24){$\rightarrow L_2$}
      \put(-6,8){$-$}
      \put(-6,5){$-$}
      \put(-6,2){$-$}
      \put(-6,-1){$-$}
      \put(36,-1){$\rightarrow L_3$}
      \put(59,79){$-$}
      \put(59,72){$-$}
      \put(72,72){$L_2'$}
      \put(75,50){$L_1'$}
    \end{overpic}
    \caption{Left: the contact surgery link $\boldsymbol{L}=\bigcup_{i=1}^{12}L_i^-$. Right: the contact surgery link $\boldsymbol{L}'=L_1'^-\cup L_2'^-$.}
    \label{fig:ex_chain_Leg1}
  \end{figure}

  \begin{figure}[htbp]
    \centering
    \begin{overpic}[scale=1]{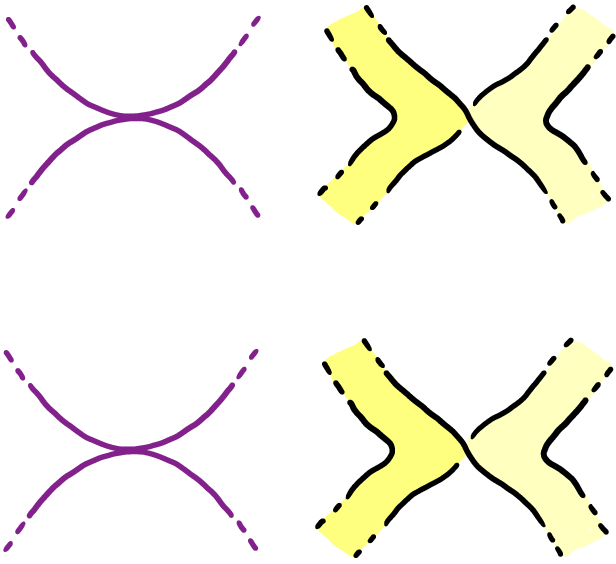}
    \end{overpic}
    \caption{A ribbon surface $R$ compatible with the contact surgery links appearing in the chain move.}
    \label{fig:ex_chain_R}
  \end{figure}

\begin{remark}\label{rem:lantern_chain_alternative}
There is also a more direct way to see that, for both the standard lantern move and the standard chain move, the two contact surgery links involved represent contactomorphic contact manifolds. We briefly outline the argument.

One first identifies the underlying surgered \(3\)-manifolds by smooth Kirby calculus. Equivalently, this identification can be checked using Regina~\cite{regina} together with SnapPy~\cite{SnapPy}. In this way, both sides of the standard lantern move are seen to describe the Seifert fibered space
\[
M_1 = SFS[S^2:(2,1)(3,2)(-3,1)],
\]
while both sides of the standard chain move describe the Seifert fibered space
\[
M_2 = SFS[S^2:(2,1)(5,1)(-5,4)],
\]
where we use Regina's notation. Since all contact surgery coefficients in these local diagrams are equal to \((-1)\), the resulting contact structures are Stein fillable, and hence tight.

Using~\cite{Lisac_StipsiczIII,Lisca_Matic04}, one checks that both \(M_1\) and \(M_2\) are \(L\)-spaces. The tight contact structures on these Seifert fibered spaces were classified by Ghiggini~\cite{Ghiggini08}. For \(M_2\), there is a unique tight contact structure. It follows immediately that the two contact structures obtained from the two sides of the standard chain move are contactomorphic.

For the standard lantern move, the Seifert fibered space \(M_1\) carries exactly two tight contact structures. To distinguish between them, one can compute the Euler classes of the contact structures arising from the two surgery diagrams. These Euler classes agree, and Ghiggini's classification shows that the Euler class distinguishes the two tight contact structures on \(M_1\). Hence the two contact structures obtained from the two sides of the standard lantern move are contactomorphic as well.
\end{remark}

\subsection{Commutativity}\label{ex:commutativity}
  Let $G$ be a Legendrian graph consisting of a chain $K_1,\dots,K_{2n+1}$ of Legendrian knots: each $K_j$ intersects $K_{j-1}$ and $K_{j+1}$ once, and is disjoint from all other components (with obvious exceptions at the ends of the chain). The Legendrian skeleton of the ribbon $R$ in Section~\ref{ex:chainLeg} is such a graph with $n=1$.

  Let $R$ be the ribbon of $G$; it is a genus $n$ surface with two boundary components $\dd_1$ and $\dd_2$. Let
  \[
    \boldsymbol{L}
    := \dd_1(-\varepsilon)^- \cup \dd_2(\varepsilon)^-,
    \qquad
    \boldsymbol{L}'
    := \dd_1(\varepsilon)^- \cup \dd_2(-\varepsilon)^-,
  \]
  be the contact surgery links obtained by Legendrian realizing $\dd_1,\dd_2$ at different heights (using the algorithm of~\cite{StenhedeAlgorithm}). The two contact surgery links $\boldsymbol{L}$ and $\boldsymbol{L}'$ differ by an elementary $R$-equivalence corresponding to exchanging two commuting Dehn twists.

  We show that $L$ and $L'$ are, in fact, Legendrian isotopic. For concreteness, we assume that all intersections between $K_j$ and $K_{j+1}$ look like the left of Figure~\ref{fig:comm1}: the front projection of $K_{j+1}$ lies above that of $K_j$, and the local orientations are coherent as indicated. The general case is analogous (moreover, it is also true that we can Legendrian isotope $G$ to achieve these properties).

  On the right of Figure~\ref{fig:comm1}, we see how $\boldsymbol{L}$ and $\boldsymbol{L}'$ look like in this case. We now describe a Legendrian isotopy between them.
    Around the intersection of $K_1$ and $K_2$, the isotopy is as in Figure~\ref{fig:comm4}: the intersection point $p$ between the two surgery components travels once around $K_1$ and reappears on the other side. Around the intersection of $K_{2n}$ and $K_{2n+1}$ the isotopy is as in Figure~\ref{fig:comm2}: the point $p$ travels around $K_{2n+1}$. Around an intermediate pair of intersections, say between $K_{2k}$ and $K_{2k+1}$ and between $K_{2k+1}$ and $K_{2k+2}$, the isotopy is as in Figure~\ref{fig:comm3}: the two intersection points $p$ and $q$ slide along $K_{2k+1}$ from one ball to the other. This step is vacuous when $n=1$. Concatenating these local isotopies shows that $L$ and $L'$ are Legendrian isotopic.

\begin{figure}[htbp]
    \centering
    \begin{overpic}[scale=1]{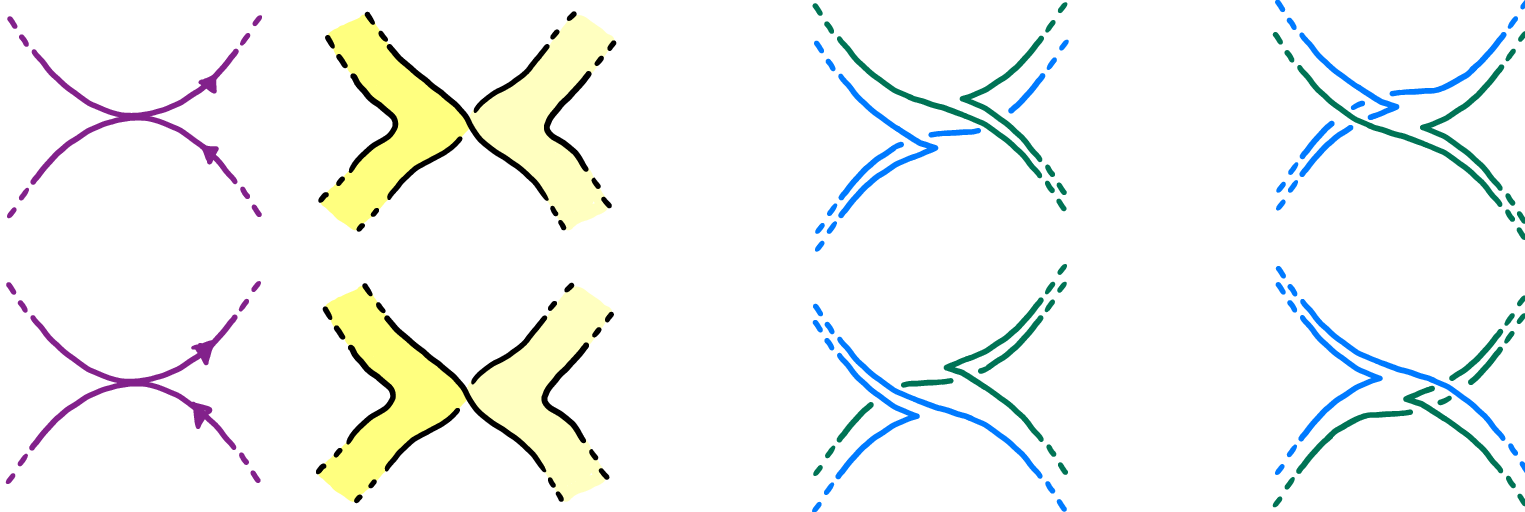}
      \put(3,30){$K_{2k+1}$}
      \put(10,20){$K_{2k}$}
      \put(57,30){$\dd_2(\varepsilon)$}
      \put(57,19){$\dd_1(-\varepsilon)$}
      \put(90,30){$\dd_1(\varepsilon)$}
      \put(89,19){$\dd_2(-\varepsilon)$}
      \put(50,32){$-$}
      \put(50,30){$-$}
      \put(80,32.5){$-$}
      \put(80,30.5){$-$}
      \put(3,12){$K_{2k}$}
      \put(10,0){$K_{2k-1}$}
    \end{overpic}
    \caption{Left: local model of an intersection between successive knots in the chain. The ribbon surface $R$ is also shown here. Right: local pictures of $\boldsymbol{L}$ (first column) and $\boldsymbol{L}'$ (second column) near a vertex; top row for intersections between $K_{2k}$ and $K_{2k+1}$, bottom row for intersections between $K_{2k-1}$ and $K_{2k}$ for $k=1,\dots,n$.}
    \label{fig:comm1}
  \end{figure}

  \begin{figure}[htbp]
    \centering
    \begin{overpic}[scale=1]{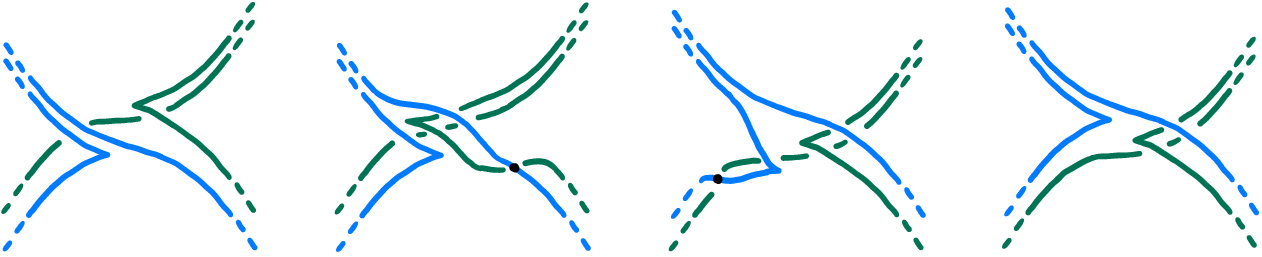}
      \put(40,4){$p$}
      \put(57,3){$p$}
    \end{overpic}
    \caption{Legendrian isotopy of $L$ near the intersection between $K_1$ and $K_2$: the point $p$ travels around $K_1$.}
    \label{fig:comm4}
  \end{figure}

  \begin{figure}[htbp]
    \centering
    \begin{overpic}[scale=1]{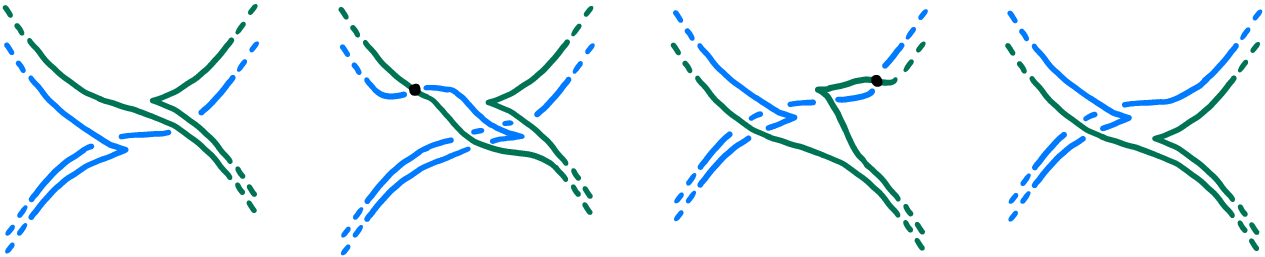}
      \put(33,15){$p$}
      \put(68,16){$p$}
    \end{overpic}
    \caption{Legendrian isotopy near the intersection between $K_{2n}$ and $K_{2n+1}$.}
    \label{fig:comm2}
  \end{figure}

  \begin{figure}[htbp]
    \centering
    \begin{overpic}[scale=1]{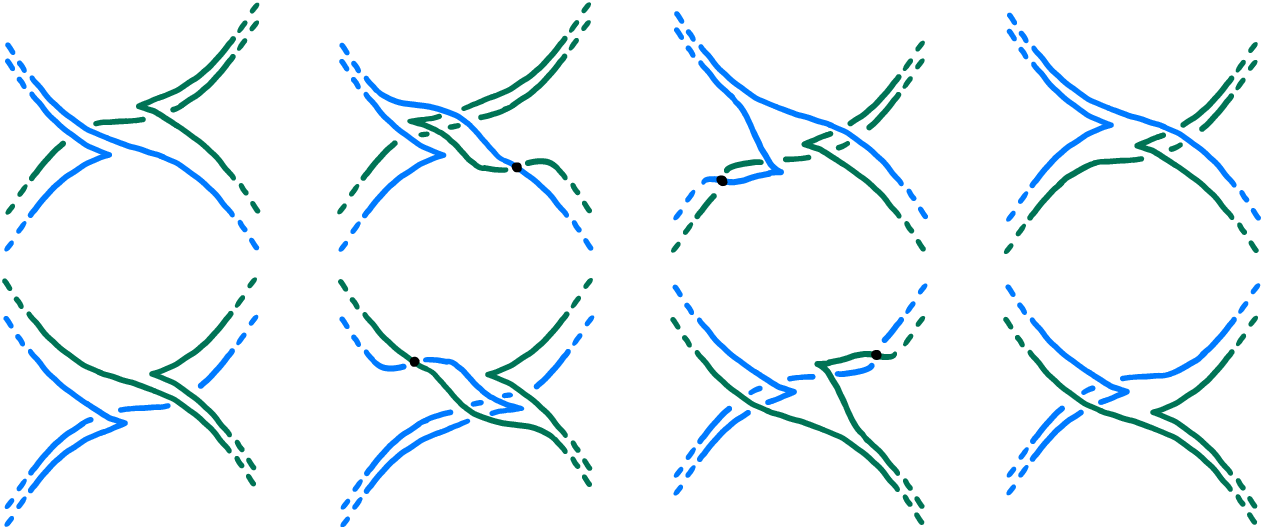}
      \put(40,26){$p$}
      \put(33,15){$q$}
      \put(57,25){$q$}
      \put(68,16){$p$}
    \end{overpic}
    \caption{Legendrian isotopy near consecutive intersections between $K_{2k},K_{2k+1}$ and $K_{2k+1},K_{2k+2}$: the intersection points $p$ and $q$ slide along $K_{2k+1}$.}
    \label{fig:comm3}
  \end{figure}
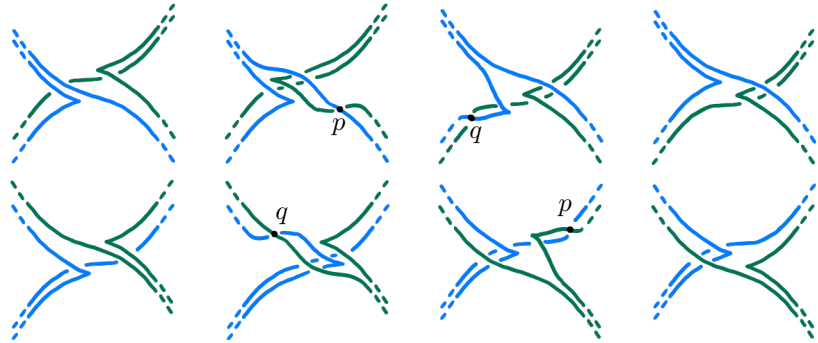

\section{Uniqueness of elementary \texorpdfstring{$R$}{R}-equivalences}\label{sec:uniqueness_R_eq}

In the previous section, we presented examples of elementary \(R\)-equivalences. We now prove a converse statement: any elementary \(R\)-equivalence is, up to Legendrian isotopy, one of the models described above.
We begin with a uniqueness statement for Legendrian links carried by a fixed ribbon surface. This is the link version of Theorem~\ref{thm:uniqueness_ribbon}.

\begin{lemma}\label{lem:uniq_Leg_link}
Let
\[
L=\bigcup_{i=1}^n L_i
\qquad\text{and}\qquad
L'=\bigcup_{i=1}^n L'_i
\]
be Legendrian links, both compatible with the same ribbon surface \(R\). Assume that the isotopy class of \(\overline{L_i}\) equals the isotopy class of \(\overline{L'_i}\) in \(R\) for each \(i\). Then \(L\) and \(L'\) are Legendrian isotopic by an isotopy supported in any prescribed neighborhood of the form $R\times[-\varepsilon,\varepsilon]$
containing both \(L\) and \(L'\).
\end{lemma}

\begin{proof}
Choose a neighborhood $
R\times[-\varepsilon,\varepsilon]$
containing both links. Subdivide the interval \([-\varepsilon,\varepsilon]\) into subintervals
\[
[t_{i-1},t_i],
\qquad i=1,\ldots,n,
\]
with \(t_0=-\varepsilon\) and \(t_n=\varepsilon\). Up to a Legendrian isotopy supported in \(R\times[-\varepsilon,\varepsilon]\), we may assume that each slab
\[
R\times[t_{i-1},t_i]
\]
contains precisely the two components \(L_i\) and \(L'_i\), and no other components of \(L\cup L'\).

By assumption, \(\overline{L_i}\) and \(\overline{L'_i}\) represent the same isotopy class of simple closed curve on \(R\). Hence \(L_i\) and \(L'_i\) are Legendrian realizations of the same isotopy class inside the vertically invariant contact manifold
\[
R\times[t_{i-1},t_i].
\]
By Theorem~\ref{thm:uniqueness_ribbon}, \(L_i\) and \(L'_i\) are Legendrian isotopic inside this slab. Since the slabs are disjoint, these isotopies can be performed independently for all \(i\). Combining them gives a Legendrian isotopy from \(L\) to \(L'\), supported in the prescribed neighborhood \(R\times[-\varepsilon,\varepsilon]\).
\end{proof}

For the next proposition, we use the following observation: the braid, the lantern, and the chain relations in the mapping class group described in Section~\ref{sec:MCG} have an underlying graph from which one can recover all the curves involved in the relation by resolving this graph in different ways. Using the notation of Section~\ref{sec:MCG}, this graph is $\aa\cup\bb$. We call it the \textit{graph of the relation}.

\begin{proposition}\label{uniq_cont_Kirby}
  Let $\boldsymbol{L}$ and $\boldsymbol{L}'$ be contact surgery links which differ by an elementary $R$–equivalence. Then, up to Legendrian isotopy, $\boldsymbol{L}$ and $\boldsymbol{L}'$ differ by one of the following moves (defined in Section~\ref{sec:elementary_R_equivalences}):
  \begin{itemize}
    \item insertion or removal of a cancelling pair;
    \item a standard contact handle slide;
    \item a lantern move;
    \item a chain move.
  \end{itemize}
\end{proposition}

\begin{proof}
  By definition, the mapping class factorizations $\tau_{\scriptscriptstyle{\boldsymbol{L}}}$ and $\tau_{\scriptscriptstyle{\boldsymbol{L}'}}$ differ by a single relation from Theorem~\ref{thm:Gervais}. Up to restricting to sublinks of $\boldsymbol{L}$ and $\boldsymbol{L}'$ by discarding components that are unchanged by the $R$-equivalence, we may assume that every Dehn twist appearing in the factorizations is involved in the relation.

  \medskip\noindent
  \emph{Case 1: cancelling pair.}  
  Suppose the factorizations differ by inserting or deleting a pair of cancelling Dehn twists. Up to swapping $\boldsymbol{L}$ and $\boldsymbol{L}'$, we can assume that $\boldsymbol{L}=L_1^{\pm}\cup L_2^{\mp}$ and $\boldsymbol{L}'$ is the empty diagram.
Let $\widetilde{L}_2$ be the Legendrian knot obtained from $L_1$ by a small vertical translation. The contact surgery link
  \[
    \widetilde{\boldsymbol{L}} := L_1^{\pm}\cup \widetilde{L}_2^{\mp}
  \]
  is a cancelling pair as in Example~\ref{cancellationEX} and is compatible with $R$ (if $L_1$ lies in a perturbation of $R\times\{t_1\}$, then $\widetilde{L}_2$ lies in a perturbation of $R\times\{t_1+\varepsilon\}$). By Theorem~\ref{lem:uniq_Leg_link}, $\boldsymbol{L}$ is Legendrian isotopic to $\widetilde{\boldsymbol{L}}$, so this elementary $R$–equivalence is realized by insertion or removal of a standard cancelling pair.

  \medskip\noindent
  \emph{Case 2: braid, lantern, or chain relation.}  
  In these cases, all curves occurring in the relation are homologically nontrivial in $R$ (by Lemma~\ref{lem:inverse_Leg_real}). Since $R$ is convex with dividing set $\partial R$, a homologically nontrivial curve is nonisolating, so the graph $\overline{G}$ of the relation is nonisolating.  
  By the Legendrian realization principle (Theorem~\ref{thm:Legendrian_realization}), we can Legendrian realize $\overline{G}$ to a Legendrian graph $G$ embedded in a perturbation of $R$. Using Remark~\ref{rmk:modi_valency_4}, we may further assume that all vertices of $G$ have the standard local front from Figure~\ref{fig:resolutions} (left).

  We now compare this with Sections~\ref{ex:braidLeg},~\ref{ex:lanternLeg}, and~\ref{ex:chainLeg}. In each of those cases, the Legendrian skeleton of the ribbon is precisely the graph of the relation, its vertices have the standard form, and every surgery component is a resolution of the Legendrian skeleton, possibly translated slightly in the $\partial_z$-direction. Thus, for each type of relation, we can construct from $G$ contact surgery links $\widetilde{\boldsymbol{L}}$ and $\widetilde{\boldsymbol{L}}'$ which
  \begin{itemize}
    \item differ by the corresponding move (handle slide, lantern, or chain move),
    \item are compatible with $R$ (since each component is a resolution of $G$, by Lemma~\ref{lem:resolution_comp}), and
    \item induce the same mapping class factorizations as $\boldsymbol{L}$ and $\boldsymbol{L}'$.
  \end{itemize}

  More concretely, the choice of which explicit pair $(\widetilde{\boldsymbol{L}},\widetilde{\boldsymbol{L}}')$ to use among the models in the examples is dictated by the front projection of $G$ around the vertices (orientation and relative position of the pieces of $G$ corresponding to $\aa$ and $\bb$).  
  Since the factorization of $\tau_{\scriptscriptstyle{\boldsymbol{L}}}$ corresponds to that of $\tau_{\scriptscriptstyle{\widetilde{\boldsymbol{L}}}}$, the components of $\boldsymbol{L}$ and $\widetilde{\boldsymbol{L}}$ define the same curves on $R$, and the same for $\boldsymbol{L}'$ and $\widetilde{\boldsymbol{L}}'$. By Theorem~\ref{lem:uniq_Leg_link}, $\boldsymbol{L}$ is Legendrian isotopic to $\widetilde{\boldsymbol{L}}$ and $\boldsymbol{L}'$ to $\widetilde{\boldsymbol{L}}'$. Therefore, the elementary $R$–equivalence is realized, up to Legendrian isotopy, by one of the models described above.

  \medskip\noindent
  \emph{Case 3: commuting Dehn twists.}  
  Suppose $\boldsymbol{L}=L_1^-\cup L_2^-$ and $\boldsymbol{L}'=L_1'^-\cup L_2'^-$ correspond to two factorizations differing by commuting two Dehn twists along disjoint simple closed curves $\aa$ and $\bb$ on $R$. The two components $L_1,L_2$ (and $L_1',L_2'$) represent $\aa,\bb$ respectively.
  If $\aa\cup\bb$ is nonisolating, we can Legendrian realize it as $\aa(0)\cup\bb(0)$, and then consider the links
  \[
    \widetilde{\boldsymbol{L}} = \aa(-\varepsilon)^-\cup \bb(\varepsilon)^-,
    \qquad
    \widetilde{\boldsymbol{L}}' = \aa(\varepsilon)^-\cup \bb(-\varepsilon)^-.
  \]
  The contact surgery link $\boldsymbol{\widetilde{L}}$ is Legendrian isotopic to $\boldsymbol{\widetilde{L}}'$ by translating $\aa(-\varepsilon)$ vertically up and $\bb(\varepsilon)$ vertically down by an amount $2\varepsilon$. These two Legendrian knots do not intersect during the isotopy, because we are assuming that $\aa(0)$ and $\bb(0)$ both are embedded simultaneously in a perturbation of $R$ since $\aa\cup\bb$ is Legendrian realizable. By Theorem~\ref{lem:uniq_Leg_link}, $\boldsymbol{L}$ is Legendrian isotopic to $\widetilde{\boldsymbol{L}}$ and $\boldsymbol{L}'$ to $\widetilde{\boldsymbol{L}}'$, concluding the proof in this subcase as well.

  If $\aa\cup\bb$ is isolating, then it bounds a subsurface $\Sigma\subset R$ of genus $n\ge 1$. As in Figure~\ref{fig:surface_comm}, choose a chain $\widetilde{G}$ of $2n+1$ simple closed curves $\cc_1,\dots,\cc_{2n+1}$ in $\Sigma$ such that $\cc_j$ intersects $\cc_{j-1}$ and $\cc_{j+1}$ once and no other component, and such that $\Sigma$ is a regular neighborhood of the union $G:=\bigcup_j \cc_j$.
  The graph $\widetilde{G}$ is nonisolating, and so we can Legendrian realize it and get a Legendrian graph $G$ embedded in a perturbation of $R$. As before, we can think of $G$ as the graph in Section~\ref{ex:commutativity}, where the Legendrian realization of the curve $\cc_j$ corresponds to the Legendrian knot $K_j$. Reasoning as above for the other relations, we see that the two Legendrian links coming from Section~\ref{ex:commutativity} are compatible with $R$ and differ by a Legendrian isotopy. Moreover, again by Theorem~\ref{lem:uniq_Leg_link}, these two Legendrian links are Legendrian isotopic to $\boldsymbol{L}$ and $\boldsymbol{L}'$. This completes the proof.
\end{proof}

\begin{figure}[htbp]
    \centering
    \begin{overpic}[scale=1]{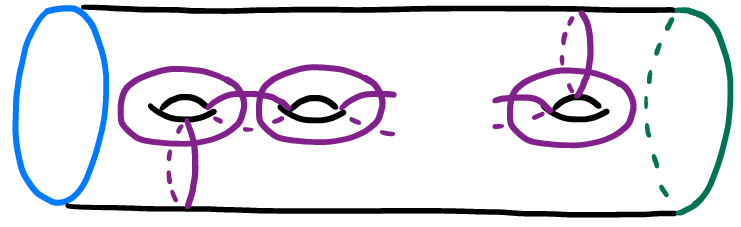}
      \put(-2,10){$\aa$}
      \put(57,6){$\Sigma$}
      \put(100,20){$\bb$}
      \put(55,14){$\boldsymbol\cdot$}
      \put(59,14){$\boldsymbol\cdot$}
      \put(63,14){$\boldsymbol\cdot$}
    \end{overpic}
    \caption{The surface $\Sigma$ bounded by $\aa\cup\bb$ and the chain $\widetilde{G}$ (in purple) whose regular neighborhood is $\Sigma$.}
    \label{fig:surface_comm}
  \end{figure}
  
\section{Proof of the contact Kirby theorem}\label{sec:main_theorem_body}

In this section, we prove Theorem~\ref{thm:main}. The main ingredients have already been established: Avdek's theorem reduces the problem to ribbon moves, and the preceding sections identify the elementary \(R\)-equivalences that are needed. We now explain how these ribbon moves translate into the diagrammatic moves appearing in the statement of the theorem.

First, we fix the terminology. Suppose that a ribbon move replaces a sublink \(\boldsymbol l\) of a contact surgery link by an \(R\)-equivalent sublink \(\boldsymbol l'\). We name the corresponding ribbon move according to the diagrammatic form of the elementary \(R\)-equivalence involved.

\begin{itemize}
    \item If the two sublinks differ by insertion or removal of a standard cancelling pair, then we call the corresponding ribbon move an \emph{insertion or removal of a standard cancelling pair}.

    \item If the two sublinks differ by one of the standard contact handle slides, then we call the corresponding ribbon move a \emph{standard contact handle slide}.

    \item If the two sublinks differ by a lantern move, then we call the corresponding ribbon move a \emph{lantern move}. If they differ by the standard lantern move, then we call the corresponding ribbon move the \emph{standard lantern move}.

    \item If the two sublinks differ by a chain move, then we call the corresponding ribbon move a \emph{chain move}. If they differ by the standard chain move, then we call the corresponding ribbon move the \emph{standard chain move}.
\end{itemize}

Here ``differ by'' means that the two sublinks agree with the two sides of the corresponding model from Section~\ref{sec:elementary_R_equivalences}. The figures in that section are drawn without shaded regions because there we only consider the sublinks involved in the elementary \(R\)-equivalence. When these models are used as ribbon moves inside a larger contact surgery link, one must also keep track of the supporting neighborhood. In the diagrammatic versions of the moves, such as those shown in Figures~\ref{fig:cancelling_intro}--\ref{fig:chain_intro}, this neighborhood is indicated by the shaded region. More precisely, the shaded region is the front projection of the handlebody $R\times[-\varepsilon,\varepsilon]$, and it is understood to be disjoint from all components of the ambient contact surgery link not involved in the move.

Finally, there is one minor notational point. Theorem~\ref{thm:main} is stated in terms of front projections, whereas in the preceding sections we have often suppressed the distinction between Legendrian links and their front projections. This causes no ambiguity. By definition, every ribbon move above is defined for contact surgery links, but it is represented by an explicit and well-defined move in front projection, and the supporting ribbon neighborhood is disjoint from all components not involved in the move. Therefore, to prove Theorem~\ref{thm:main}, it is enough to prove the following link-level statement: any two contact surgery links in \((S^3,\xist)\) representing contactomorphic contact manifolds are related by a sequence of Legendrian isotopies, insertions or removals of standard cancelling pairs, standard contact handle slides, standard lantern moves, and standard chain moves. Taking front projections of such a sequence gives precisely the diagrammatic sequence of moves appearing in Theorem~\ref{thm:main}.

\begin{proof}[Proof of Theorem~\ref{thm:main}]
  We already know that if two contact surgery links differ by Legendrian isotopies, insertions or removals of standard cancelling pairs, standard contact handle slides, standard lantern moves, and standard chain moves, then they describe contactomorphic contact manifolds: each move is either a ribbon move (Proposition~\ref{ribbon_moves_prop}) or a Legendrian isotopy.

  For the converse, let $\boldsymbol{L}$ and $\boldsymbol{L}'$ describe contactomorphic contact $3$-manifolds. By Theorem~\ref{Vague_Kirby_Theorem}, there is a sequence of ribbon moves and Legendrian isotopies taking $\boldsymbol{L}$ to $\boldsymbol{L}'$. Thus, it suffices to show the result for the case where $\boldsymbol{L}$ and $\boldsymbol{L}'$ differ by a single ribbon move. After restricting to sublinks, we can further assume that $\boldsymbol{L}$ and $\boldsymbol{L}'$ are $R$-equivalent.
  Let
  \[
    \tau_{\scriptscriptstyle{\boldsymbol{L}}}
    = \tau_{\scriptscriptstyle{\overline{L_n}}}^{-\delta_n}\cdots\tau_{\scriptscriptstyle{\overline{L_1}}}^{-\delta_1},
    \qquad
    \tau_{\scriptscriptstyle{\boldsymbol{L}'}}
    = \tau_{\scriptscriptstyle{\overline{L'_m}}}^{-\delta'_m}\cdots\tau_{\scriptscriptstyle{\overline{L'_1}}}^{-\delta'_1},
  \]
  be the mapping classes in $\mathrm{MCG}(R,\partial R)$ determined by $\boldsymbol{L}$ and $\boldsymbol{L}'$. By Lemma~\ref{lem:inverse_Leg_real}, every curve $\overline{L_i}$ and $\overline{L'_j}$ is homologically nontrivial in $R$.

  We would like to use the mapping class group presentation in Theorem~\ref{thm:Gervais}, which requires genus at least $2$ and nonseparating curves. To arrange this, let $G$ be the Legendrian skeleton of $R$ and enlarge it in two steps:
  \begin{itemize}
    \item First, attach to $G$ the purple Legendrian graph in Figure~\ref{fig:special_chain_lantern} along a Legendrian arc. The ribbon of this purple graph has genus $2$ (this can be checked by constructing its ribbon explicitly), and the resulting ribbon $\widehat{R}$ is the boundary connected sum of $R$ with this genus $2$ surface. Hence $\widehat{R}$ has genus at least $2$.
    \item Second, following~\cite[Lemma~4.11]{Avdek13}, add further Legendrian segments to the enlarged graph to obtain a graph $\widetilde{G}$ whose ribbon $\widetilde{R}$ has connected boundary (and still genus at least $2$, as attaching segments corresponds to adding $1$–handles to the ribbon).
  \end{itemize}

  \begin{figure}[htbp]
    \centering
    \begin{overpic}[scale=1]{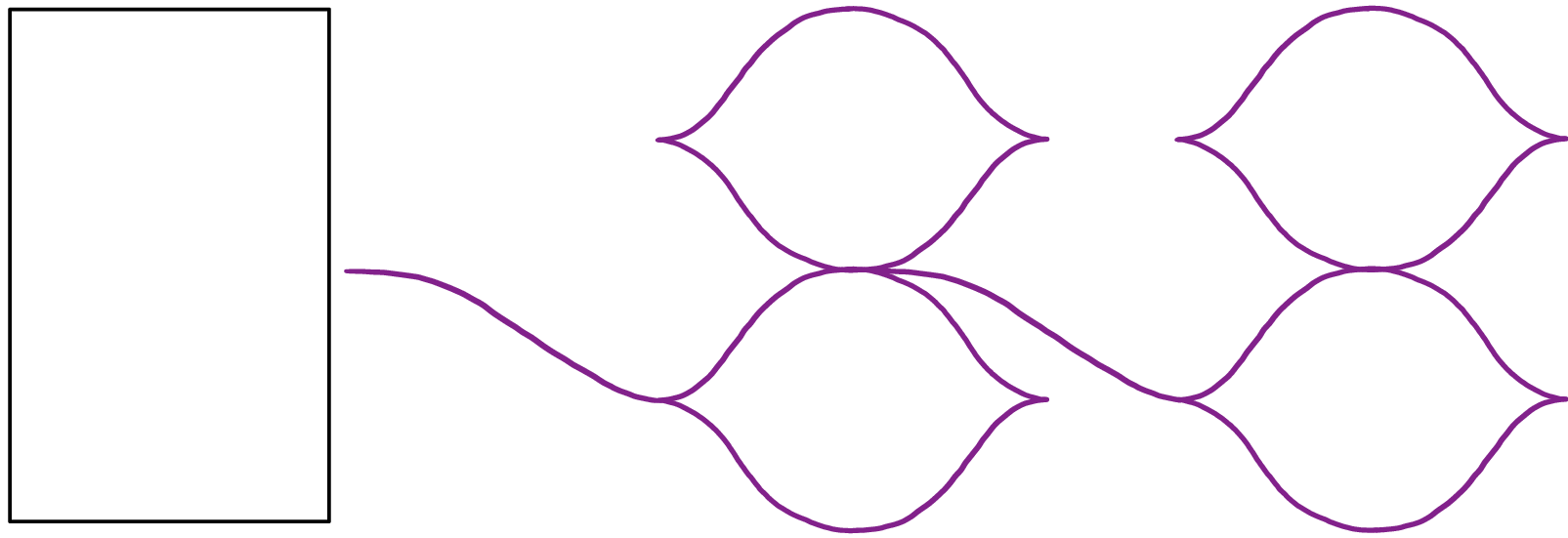}
      \put(9,17){$G$}
    \end{overpic}
    \caption{The box labelled $G$ on the left symbolizes the Legendrian graph $G$. We are attaching to $G$ (using a Legendrian arc) the purple Legendrian graph on the right of this figure.}\label{fig:special_chain_lantern}
  \end{figure}

  Since $\widetilde{G}$ contains $G$, Lemma~\ref{lem:Req_after_enlarg} implies that $\boldsymbol{L}$ and $\boldsymbol{L}'$ are also $\widetilde{R}$–equivalent. We will still denote the corresponding mapping classes by $\tau_{\scriptscriptstyle{\boldsymbol{L}}}$ and $\tau_{\scriptscriptstyle{\boldsymbol{L}'}}$, now viewed as elements in $\mathrm{MCG}(\widetilde{R},\partial\widetilde{R})$.
  The advantage of passing to $\widetilde{R}$ is that its genus is at least $2$ and its boundary is connected. Since each curve $\overline{L_i}$ and $\overline{L'_j}$ is homologically nontrivial, it is nonseparating on $\widetilde{R}$. Thus, all Dehn twists in the factorizations of $\tau_{\scriptscriptstyle{\boldsymbol{L}}}$ and $\tau_{\scriptscriptstyle{\boldsymbol{L}'}}$ are along nonseparating curves, and we are in the hypotheses of Theorem~\ref{thm:Gervais}.

  Next, consider a sequence of relations $r_1,\dots,r_k$ from Theorem~\ref{thm:Gervais} that transforms $\tau_{\scriptscriptstyle{\boldsymbol{L}}}$ into $\tau_{\scriptscriptstyle{\boldsymbol{L}'}}$. After each relation $r_i$, the intermediate factorization still involves only Dehn twists along nonseparating curves. 
  By the Legendrian realization principle and Proposition~\ref{prop:Dehn_surgery_Dehn_twist}, each intermediate factorization is realized by some contact surgery link compatible with $\widetilde{R}$. Thus, we obtain a sequence of contact surgery links
  \[
    \boldsymbol{L}=\boldsymbol{L}^0,\boldsymbol{L}^1,\dots,\boldsymbol{L}^k=\boldsymbol{L}'
  \]
  such that
  \begin{itemize}
    \item each $\boldsymbol{L}^i$ is compatible with $\widetilde{R}$;
    \item the mapping class of $\boldsymbol{L}^i$ is the one obtained from $\tau_{\scriptscriptstyle{\boldsymbol{L}}}$ after the relations $r_1,\dots,r_i$; and
    \item consecutive links $\boldsymbol{L}^i,\boldsymbol{L}^{i+1}$ differ by an elementary $R$–equivalence.
  \end{itemize}
    By Proposition~\ref{uniq_cont_Kirby}, up to Legendrian isotopy, each elementary $R$–equivalence is realized by inserting or deleting a cancelling pair, a contact handle slide, a lantern move, or a chain move.  

  Finally, to ensure that all lantern and chain moves are the \emph{standard} ones, we modify the initial enlargement of $G$ slightly. Along with the genus–$2$ piece in Figure~\ref{fig:special_chain_lantern}, we also attach two Legendrian graphs $G_{\mathrm{chain}}$ and $G_{\mathrm{lantern}}$ whose ribbons support the specific chain and lantern relations underlying the standard chain and lantern moves, see Figure~\ref{fig:genus2}. In this way, we can arrange that every lantern or chain relation in the sequence $r_1,\dots,r_k$ is supported in one of these subsurfaces, and therefore corresponds to the standard chain or lantern move on the surgery diagram.
  This concludes the proof.
\end{proof}

\begin{figure}[htbp]
    \centering
    \begin{overpic}[scale=1]{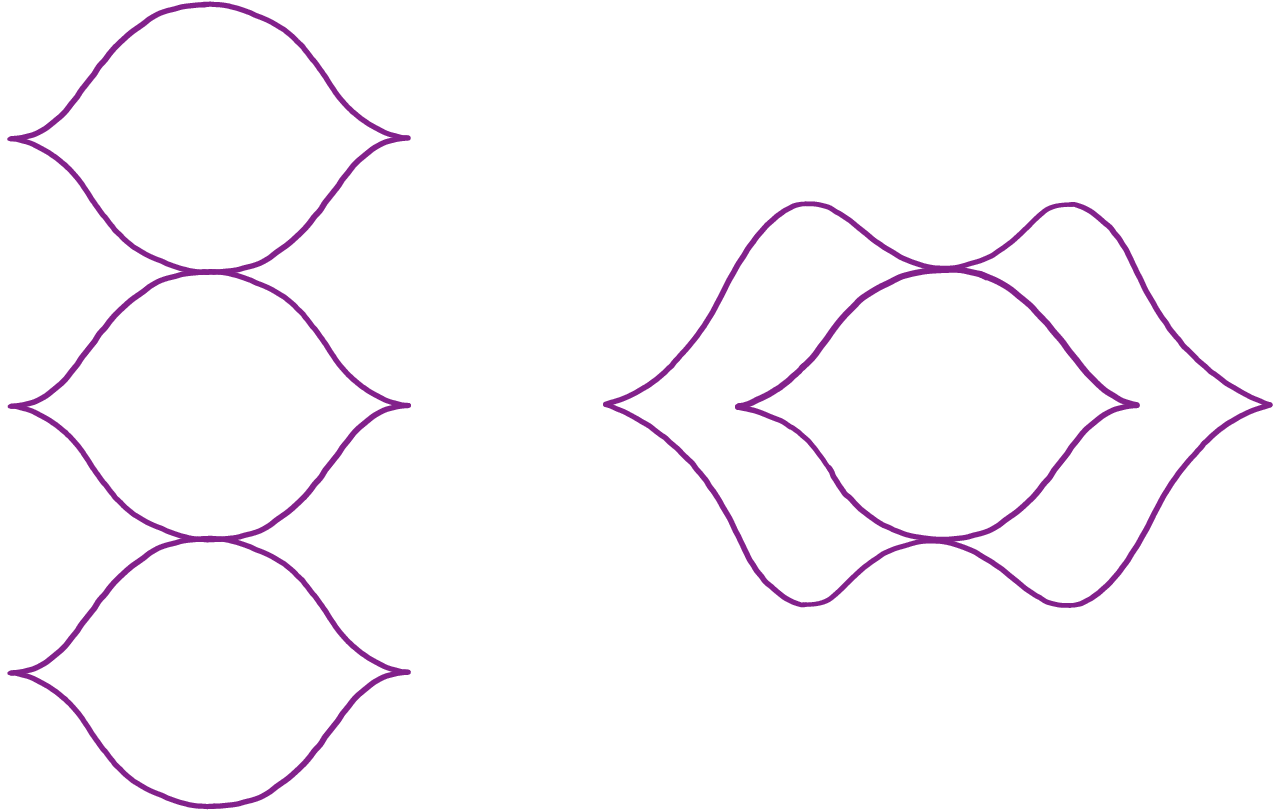}
      \put(30,45){$G_{\mathrm{chain}}$}
      \put(65,50){$G_{\mathrm{lantern}}$}
    \end{overpic}
    \caption{Legendrian graphs $G_{\mathrm{chain}}$ and $G_{\mathrm{lantern}}$ whose ribbons support the standard chain and lantern relations.}
    \label{fig:genus2}
\end{figure}


\section{A toolbox of contact Kirby moves}\label{sec:moves}

In this section, we collect contact Kirby moves from the literature, their diagrammatic representatives, and the new moves introduced in this article. The section is intended as a ready-to-use toolbox for readers who want to manipulate contact surgery diagrams without having to reconstruct the moves from the proofs in the main body of the paper. Some of these moves are most naturally formulated three-dimensionally, at the level of contact surgery links, rather than as moves on front projections. In this section, we use the term \emph{contact Kirby move} in a broader sense: it means an operation on a contact surgery link, or on a contact surgery diagram, which does not change the contactomorphism type of the represented contact manifold. 
As in the introduction, the shaded region indicates the front projection of the supporting surface or handlebody on which the move is performed. The corresponding supporting region is assumed to be disjoint from all components of the contact surgery link except those explicitly involved in the move. All components not involved in the move are left unchanged.

\subsection{Legendrian isotopy and Legendrian Reidemeister moves}
The most basic operation is Legendrian isotopy of the contact surgery link. In the front projection, two Legendrian links in \((S^3,\xist)\) are Legendrian isotopic if and only if their front projections are related by planar isotopies and the finitely many Legendrian Reidemeister moves, shown in Figure~\ref{fig:Reidemeister_Leg_link}; see Theorem~\ref{thm:Reidemeister_Legendrian}. This was first proved in~\cite{Legendrian_Reidemeister}. Since contact surgery on Legendrian-isotopic contact surgery links yields contactomorphic contact manifolds, these moves may be regarded as the most elementary contact Kirby moves. 

\subsection{Cancelling pairs}
The cancellation lemma is one of the classical moves in contact surgery. In its intrinsic form, it says that if a contact surgery link contains two components consisting of a Legendrian knot and a Legendrian push-off, equipped with opposite contact surgery coefficients, then this pair may be inserted or removed without changing the contactomorphism type of the surgered contact manifold, provided that the annulus between the two components is disjoint from the rest of the surgery link. This move first appeared in~\cite{Ding_Geiges_fillable}; see also~\cite{Ding_Geiges_Stipsicz}. Other proofs and interpretations were later given in~\cite{Avdek13,Casals_Etnyre_Kegel}. The cancelling pair used in Theorem~\ref{thm:main} is a particular diagrammatic representative of this move: the two components differ by translation in the \(\partial_z\)-direction in front projection. We call this a \emph{standard cancelling pair}; see Figure~\ref{fig:cancelling_intro}.

\subsection{Contact handle slides}\label{sec:ctchadleslidegeneral}

Contact handle slides are contact analogues of the smooth handle slide. They were first discussed by Ding and Geiges~\cite{Ding_Geiges_slides}, and further proofs and variants appear in~\cite{Avdek13,Casals_Etnyre_Kegel}. A contact handle slide is the following operation. Let \(L\) be a Legendrian component of a contact surgery link with surgery coefficient \((\pm1)\), and let \(K\) be a Legendrian knot in the complement of~\(L\). After contact surgery on $L$, the knot $K$ can be viewed as a Legendrian knot in the contact manifold obtained by contact $(\pm1)$-surgery on $L$. In the surgered manifold, $K$ is Legendrian isotopic to a band connected sum of $K$ with the Legendrian $(\pm1,1)$-cable of $L$. Thus, if $K$ is equipped with a contact surgery coefficient, then that coefficient is preserved under this isotopy. If \(K\) is unframed, the same operation gives a move for Legendrian knots in the surgered contact manifold. This latter version is used in Theorem~\ref{prop:Leg_links_Kirby_thm}. The two handle slides appearing in Theorem~\ref{thm:main} are the \emph{standard contact handle slides}. They are the two specific front-projection representatives shown in Figure~\ref{fig:braid_intro}. Figures~\ref{fig:contact_handle_slides1} and~\ref{fig:contact_handle_slides2} show further diagrammatic incarnations of contact handle slides.

\begin{figure}[htbp]
    \centering
    \begin{overpic}[scale=1]{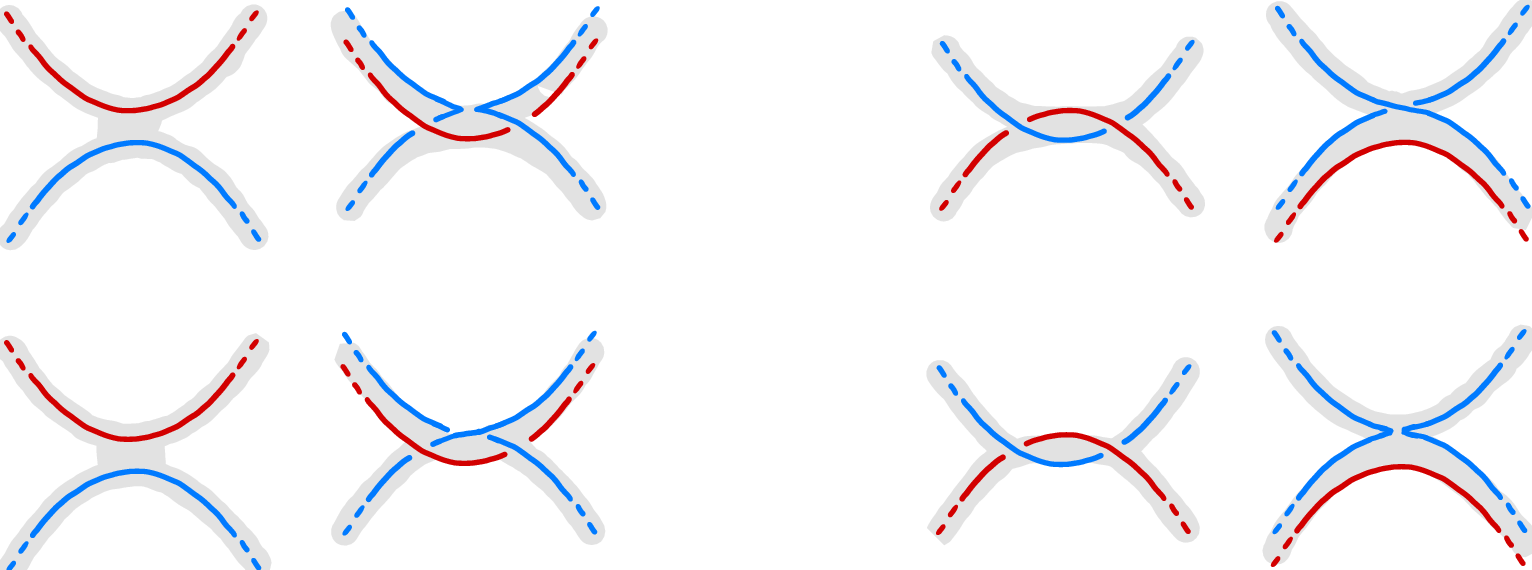}
    \put(0,32){$-$}
    \put(20,32){$-$}
    \put(64,25){$-$}
    \put(88,25){$-$}
    \put(0,11){$+$}
    \put(20,11){$+$}
    \put(64,4){$+$}
    \put(88,4){$+$}
    \end{overpic}
    \caption{Contact handle slides. The blue component may be unframed; if it is a surgery component, its surgery coefficient is preserved.}\label{fig:contact_handle_slides1}
\end{figure}

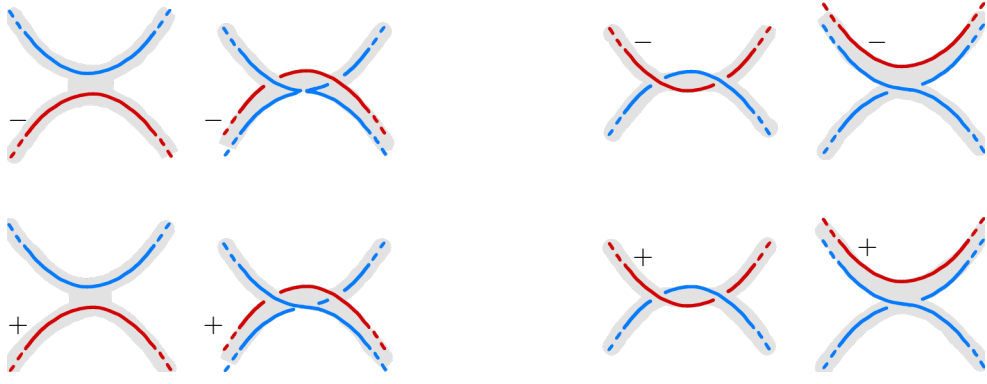
\begin{figure}[htbp]
    \centering
    \begin{overpic}[scale=1]{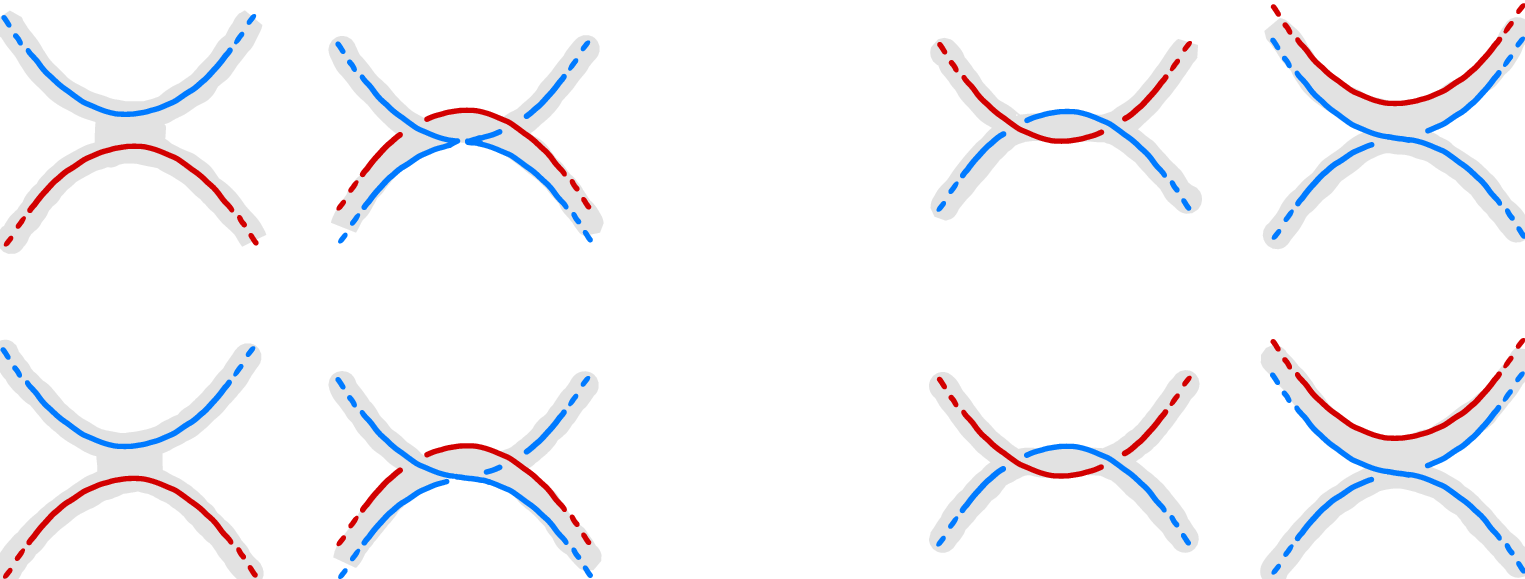}
    \put(0,25){$-$}
    \put(20,25){$-$}
    \put(64,33){$-$}
    \put(88,33){$-$}
    \put(0,4){$+$}
    \put(20,4){$+$}
    \put(64,11){$+$}
    \put(87,12){$+$}
    \end{overpic}
    \caption{Further contact handle slides. Again, the blue component may be unframed; if it is a surgery component, its surgery coefficient is preserved.}\label{fig:contact_handle_slides2}
\end{figure}

\subsection{Lantern destabilizations}

Lantern destabilizations were introduced by Lisca and Stipsicz~\cite{Lisca_Stipsicz_lantern}. They form a family of contact Kirby moves, indexed by the number of additional strands passing through the picture. Their proof proceeds by translating from the relevant surgery links to compatible abstract open books and observing that the two open books differ by two positive stabilizations together with a lantern relation. We refer to Proposition~2.4 of~\cite{Lisca_Stipsicz_lantern} and Lemma~2.9 of~\cite{Etnyre_Kegel_Onaran} for details. A diagrammatic version of the family is shown in Figure~\ref{fig:lantern_destabilization_general}. In Proposition~\ref{prop:relations_known_moves} below, we give an explicit expression of the lantern destabilization move in terms of the other contact Kirby moves.

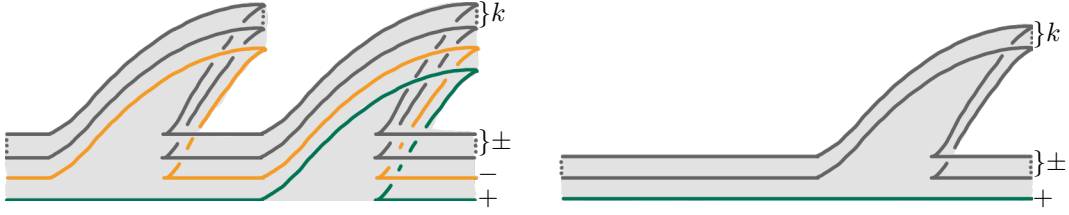
\begin{figure}[htbp]
    \centering
    \begin{overpic}[scale=1]{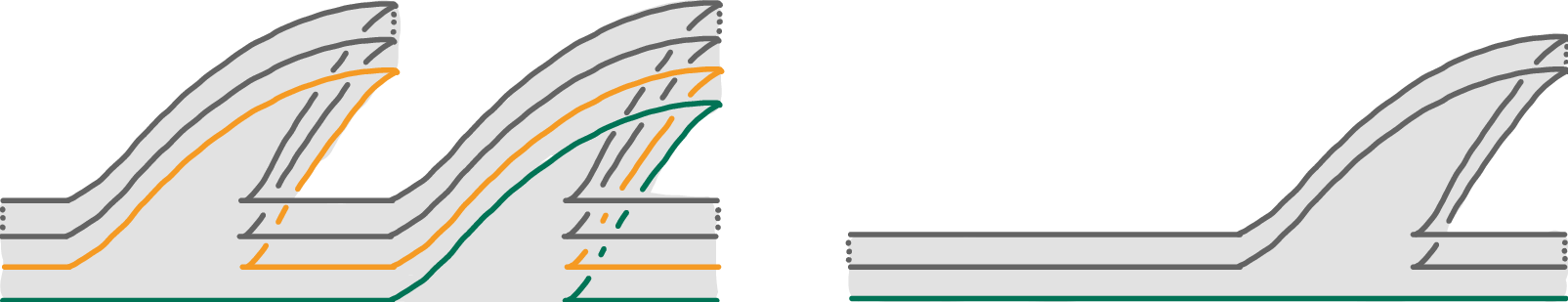}
    \put(46,-0.4){$+$}
    \put(46,1.9){$-$}
    \put(46,5){$\}\pm$}
    \put(46,17.5){$\}k$}
    \put(100,-0.4){$+$}
    \put(100,3){$\}\pm$}
    \put(100,15.5){$\}k$}
    \end{overpic}
    \caption{The lantern destabilizations with $k\in\Z_{\geq 0}$ strands, shown in grey.}\label{fig:lantern_destabilization_general}
\end{figure}

\subsection{Lantern and chain moves}

The lantern and chain moves considered in this article arise from the corresponding lantern and chain relations in the mapping class group. They are constructed in Section~\ref{sec:elementary_R_equivalences} by translating these mapping class group relations into contact surgery diagrams using compatible ribbon surfaces and open books. Special cases of the lantern moves already occur in~\cite{Avdek13}, and the standard chain move appears there as well. The standard lantern and chain moves from Figures~\ref{fig:lantern_intro} and~\ref{fig:chain_intro} are particular members of the larger families shown in Figures~\ref{fig:lantern_moves_app1}--\ref{fig:chain_moves_app4}. In these figures, the left column and the right column represent the two sides of the move, and the arrows on the red and blue strands indicate how the strands leaving one ball reconnect to the strands entering the other ball. Since all parallel strands are understood to remain vertical translates of one another outside the depicted balls, these arrows also determine the reconnection data for the remaining strands.

For completeness, we also mention that contact handle slides, lantern moves, and chain moves admit positive analogues. These are obtained by considering the inverses of the braid, lantern and chain relations and constructing the corresponding contact Kirby moves in front projections as in Section~\ref{sec:elementary_R_equivalences}. We do not depict these moves separately, since they can be expressed explicitly in terms of the negative versions together with insertions and removals of cancelling pairs.

\begin{figure}[htbp]
    \centering
    \begin{overpic}[scale=1]{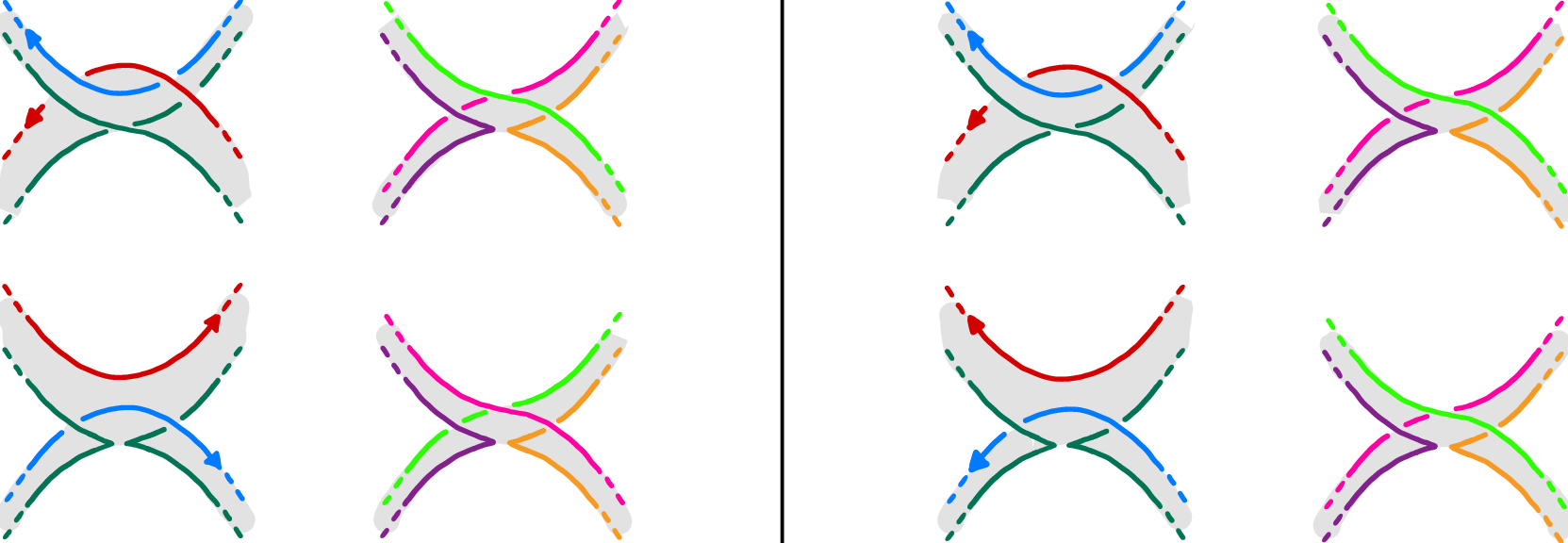}
    \end{overpic}
    \caption{Two examples of lantern moves. All components are decorated with a $-$ sign.}\label{fig:lantern_moves_app1}
\end{figure}

\begin{figure}[htbp]
    \centering
    \begin{overpic}[scale=1]{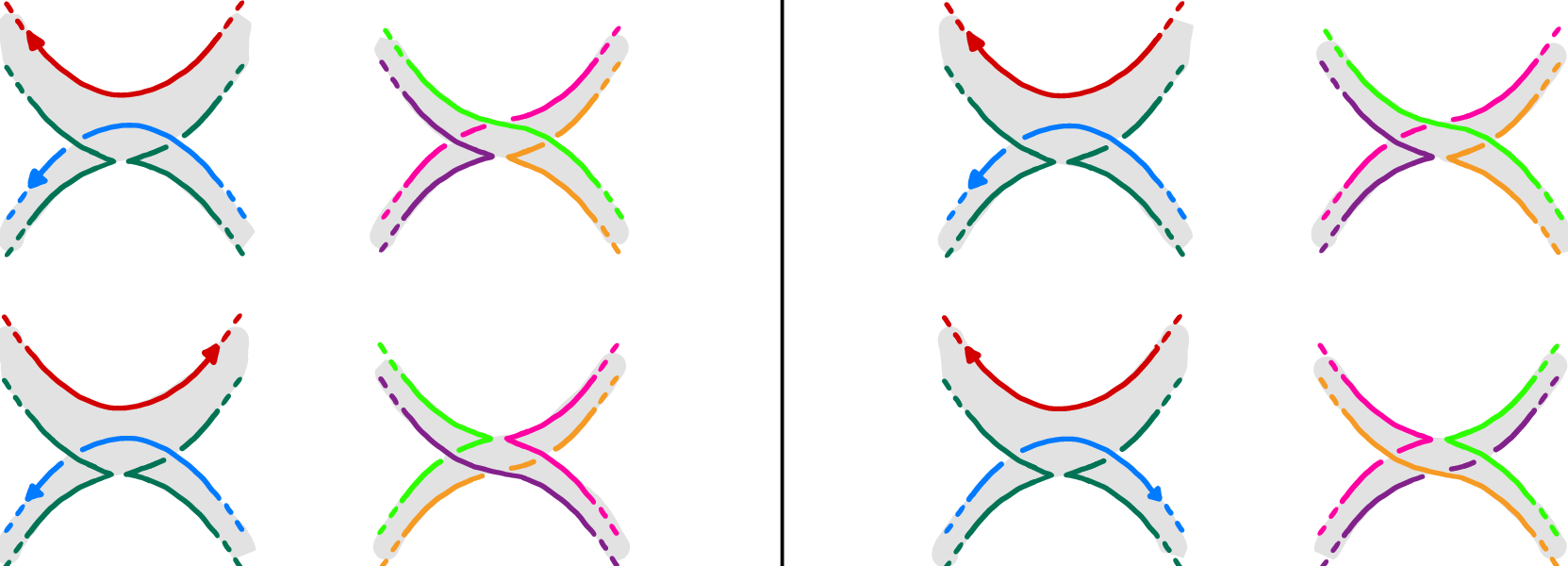}
    \end{overpic}
    \caption{Two further examples of lantern moves. All components are decorated with a $-$ sign.}\label{fig:lantern_moves_app2}
\end{figure}

\begin{figure}[htbp]
    \centering
    \begin{overpic}[scale=1]{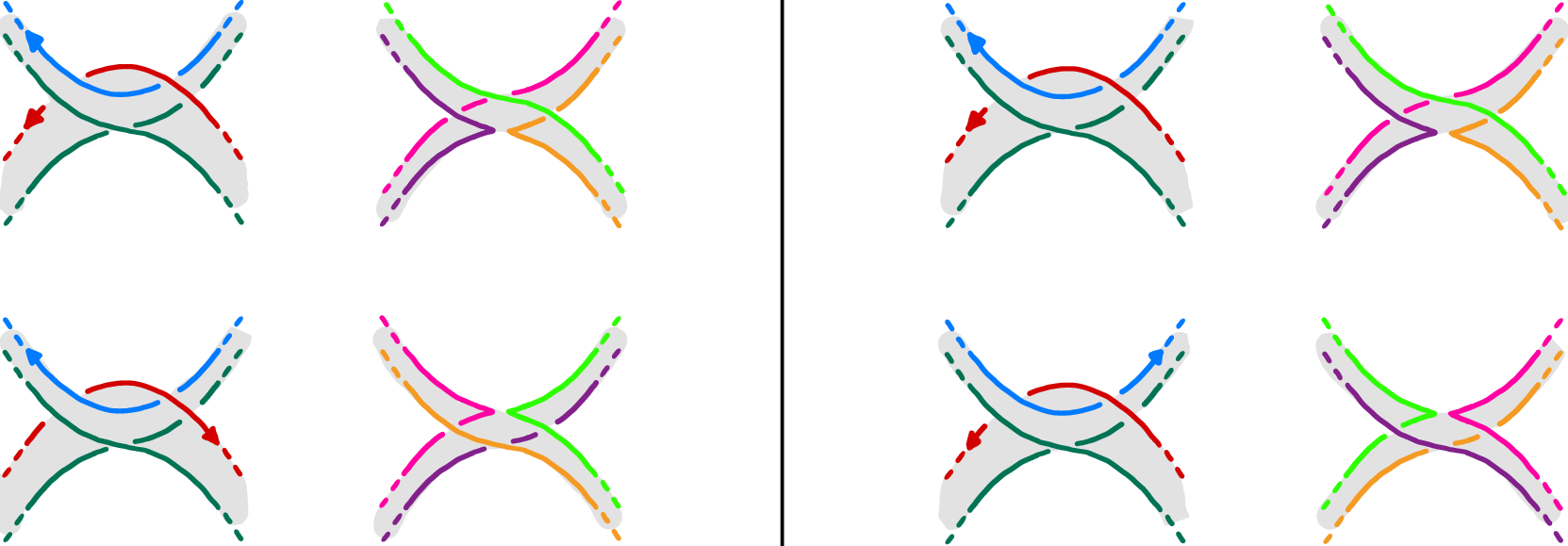}
    \end{overpic}
    \caption{Two more examples of lantern moves arising from the lantern relation. All components are decorated with a $-$ sign.}\label{fig:lantern_moves_app3}
\end{figure}

\begin{figure}[htbp]
    \centering
    \begin{overpic}[scale=0.85]{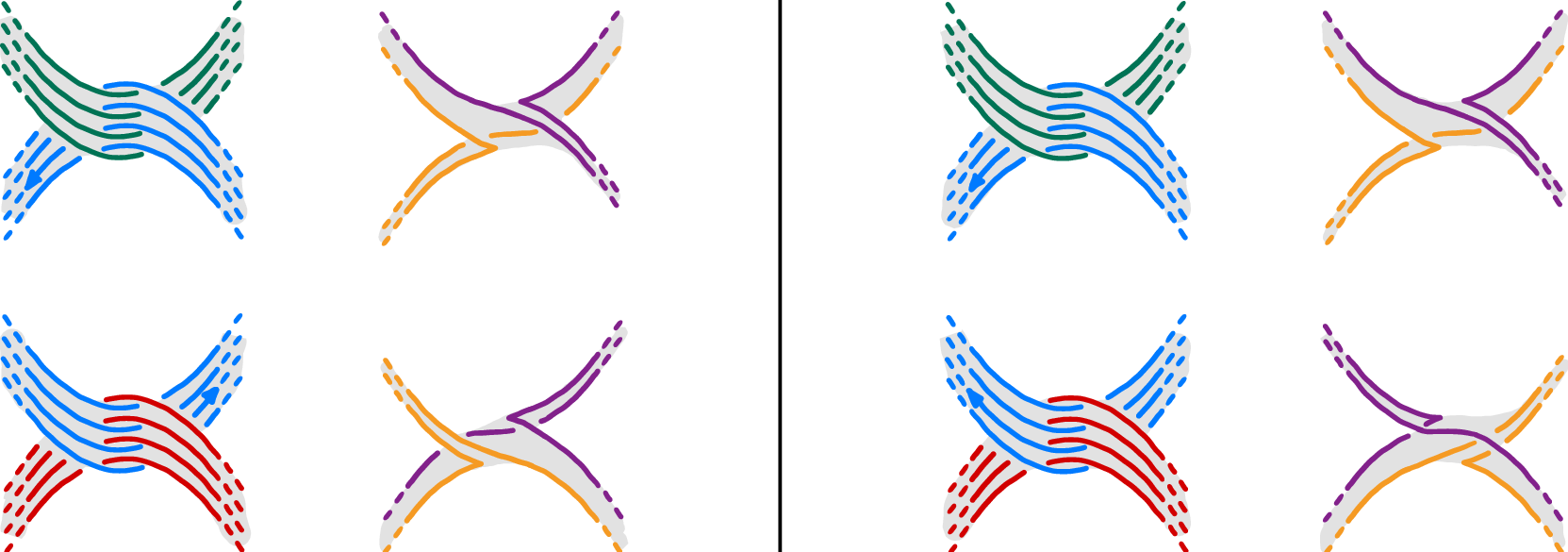}
    \end{overpic}
    \caption{Two examples of chain moves. All components are decorated with a $-$ sign.}\label{fig:chain_moves_app1}
\end{figure}

\begin{figure}[htbp]
    \centering
    \begin{overpic}[scale=0.85]{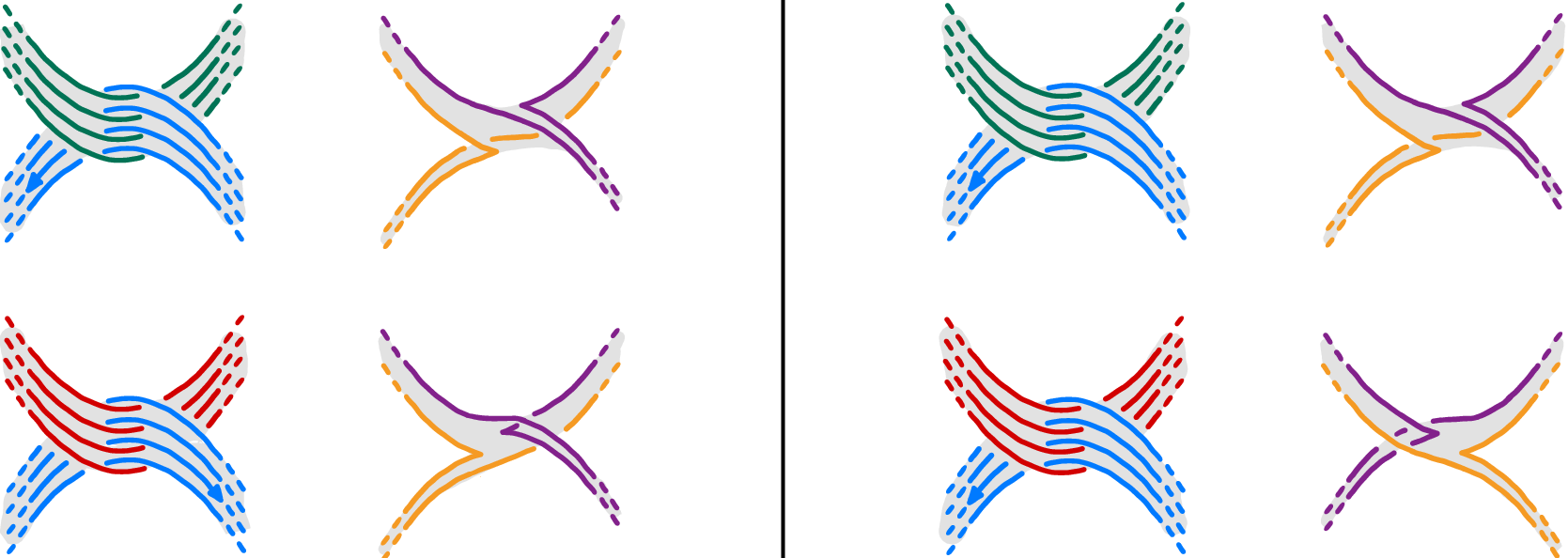}
    \end{overpic}
    \caption{Two further examples of chain moves. All components are decorated with a $-$ sign.}\label{fig:chain_moves_app2}
\end{figure}

\begin{figure}[htbp]
    \centering
    \begin{overpic}[scale=0.85]{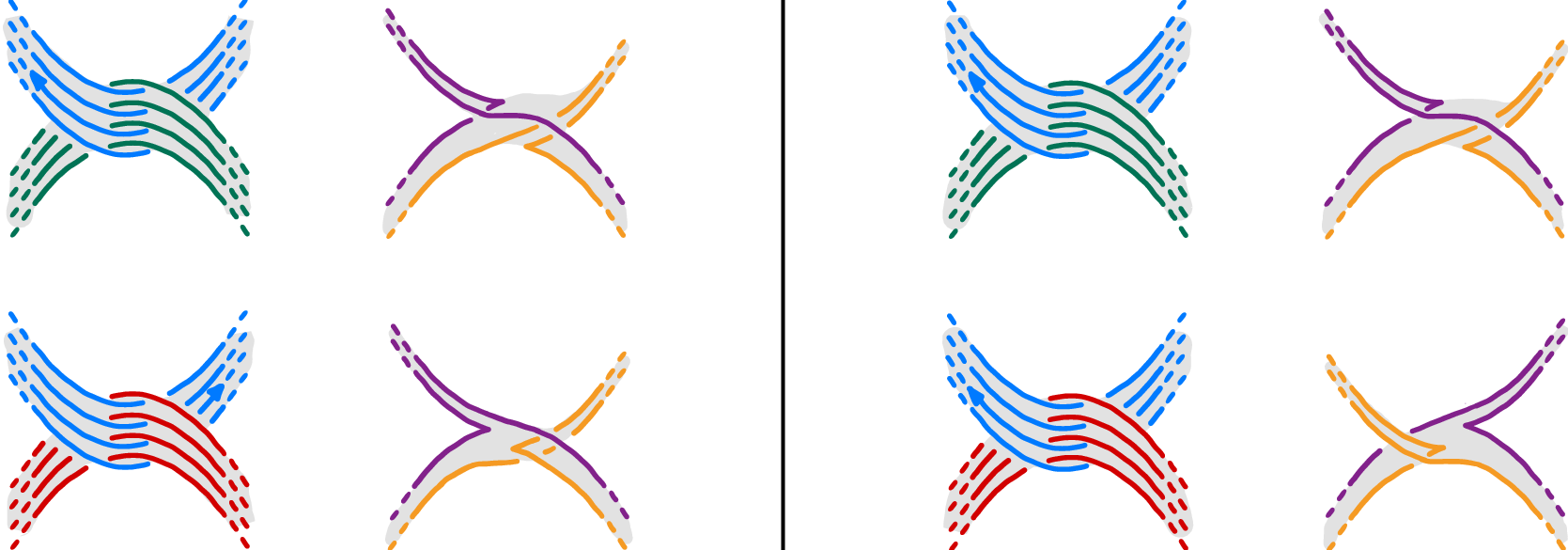}
    \end{overpic}
    \caption{Two more examples of chain moves. All components are decorated with a $-$ sign.}\label{fig:chain_moves_app3}
\end{figure}

\begin{figure}[htbp]
    \centering
    \begin{overpic}[scale=0.85]{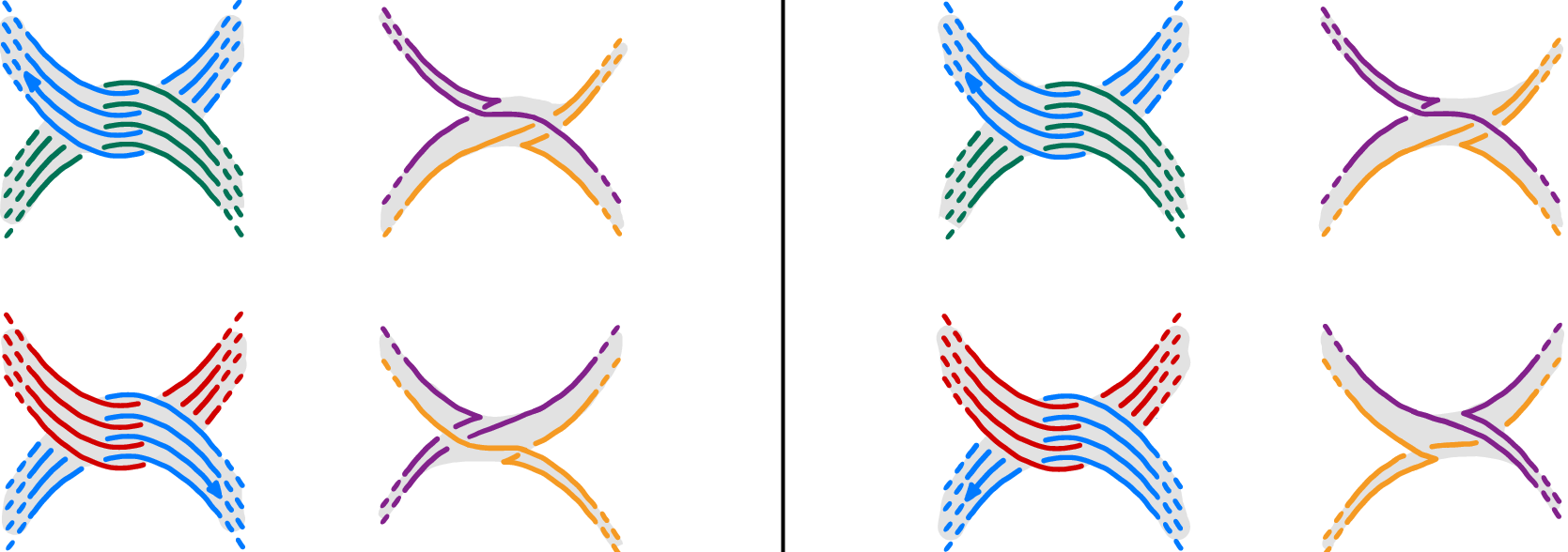}
    \end{overpic}
    \caption{Additional examples of chain moves; again all components have coefficient $-$.}\label{fig:chain_moves_app4}
\end{figure}

\subsection{Other moves in the literature}

There are other contact Kirby moves in the literature. Examples include the contact annulus twist of~\cite{Casals_Etnyre_Kegel} and various contact Rolfsen twists~\cite{Ding_Geiges_slides,Kegel_Legendrian_complement,Kegel_thesis,Kegel_Onaran}. These are useful operations in their own right. However, for the purposes of this paper, we do not include them as basic moves in the toolbox below since the contact annulus twist can be expressed as an explicit sequence of contact handle slides followed by a cancellation~\cite{Casals_Etnyre_Kegel}, while contact Rolfsen twists can be explicitly expressed in terms of contact handle slides and lantern destabilizations~\cite{Kegel_Onaran}.


\section{Homotopical invariants via contact surgery diagrams}
\label{sec:diagrammatic_d3}

In this section, we discuss how to define homotopical invariants of contact structures by verifying their invariance under the contact Kirby moves.

\subsection{The $d_3$-invariant}
We prove Proposition~\ref{prop:d3-invariance} by showing that the surgery formula for Gompf's \(d_3\)-invariant is unchanged under the contact Kirby moves, and can therefore be used as a diagrammatic definition.

We first recall the geometric origin of the $d_3$-invariant. Let \((M,\xi)\) be a contact \(3\)-manifold such that its first Chern class \(c_1(\xi)\in H^2(M;\mathbb Z)\) is torsion. Let \((X,J)\) be a compact almost complex \(4\)-manifold with \(\partial X=M\), such that the field of complex tangencies along \(\partial X\) is \(\xi\) (for the existence of such an $X$ we refer, for example, to~\cite{Gompf}). With our normalization taken into account\footnote{Recall from the conventions in the introduction that our normalization of the $d_3$-invariant differs from Gompf's original definition by a summand of $1/2$.}, Gompf defines
\[
d_3(\xi)
=
\frac14\left(c_1(J)^2-3\sigma(X)-2\chi(X)\right)+\frac{1}{2},
\]
where \(\sigma(X)\) is the signature of the intersection form of \(X\), \(\chi(X)\) is its Euler characteristic, and \(c_1(J)^2\) is interpreted as an absolute-relative pairing. Gompf proved that, under the condition that the first Chern class is torsion, this rational number is independent of the choice of \((X,J)\) and is an invariant of the underlying tangential $2$-plane field of \(\xi\).

Building on Gompf's work, Ding--Geiges--Stipsicz~\cite{Ding_Geiges_Stipsicz} give a formula for \(d_3(\xi)\) in terms of a contact surgery presentation of \((M,\xi)\). We describe the formula in our shifted normalization in the following lemma and include a proof sketch for completeness.

\begin{proposition}[Ding--Geiges--Stipsicz~\cite{Ding_Geiges_Stipsicz}]\label{lem:d3_surgery_formula}
    Let \((M,\xi)\) be a contact \(3\)-manifold with \(c_1(\xi)\) torsion and let \(\boldsymbol L\) be an oriented contact surgery link for $(M,\xi)$ with components $L_1,\ldots,L_n$. We write \(q=q(\boldsymbol L)\) for the number of components of \(\boldsymbol L\) with a \((+1)\) contact surgery coefficient, $\boldsymbol r
=
(\rot(L_1),\ldots,\rot(L_n))^T$ 
for the vector of rotation numbers of the components of \(\boldsymbol L\), and we denote by \(Q=Q_{\boldsymbol L}\) the linking matrix of $\boldsymbol{L}$. 
Then $Q\boldsymbol x=\boldsymbol r$ has a rational solution $\boldsymbol{x}$. For any such solution we define $c_{\boldsymbol L}^2:=\boldsymbol x^TQ\boldsymbol x$. The surgery formula for \(d_3\) is then
\[
d_3(\xi)
=
\frac14\left(c_{\boldsymbol L}^2-3\sigma(Q_{\boldsymbol L})-2n\right)
+q(\boldsymbol L).
\]
\end{proposition}

\begin{proof}[Proof sketch]
First, we give the argument in the case that all contact surgery coefficients are $(-1)$, i.e.\ $q(\boldsymbol{L})=0$. This case is already contained in~\cite{Gompf}. Let \(X=X_{\boldsymbol L}\) denote the smooth surgery trace of \(\boldsymbol L\), i.e.\ the compact $4$-manifold obtained from \(D^4\) by attaching \(2\)-handles along the components \(L_i\) with smooth framings \(\tb(L_i)\pm1\), where the sign is the corresponding contact surgery coefficient. Since all contact surgery coefficients are $(-1)$, the surgery trace carries a preferred Stein structure and is, in particular, an almost complex manifold. From the handle decomposition, we compute \(\chi(X)=n+1\), and the homology \(H_2(X;\mathbb Z)\cong\Z^n\) is generated by the classes \([\widehat{\Sigma}_i]\), where \(\widehat{\Sigma}_i\) is obtained by capping a Seifert surface \(\Sigma_i\) in $D^4$ for \(L_i\) with the core of the corresponding \(2\)-handle.
With respect to this basis, the intersection form of $X$ is represented by the linking matrix \(Q=Q_{\boldsymbol L}\), whose entries are
\[
[\widehat{\Sigma}_i]\cdot[\widehat{\Sigma}_j]=
   \begin{cases}
      \lk(L_i,L_j), & \text{if } i\neq j,\\
      \tb(L_i)\pm1, & \text{if } i=j.
   \end{cases}
\]
The rotation vector determines a cohomology class \(c_{\boldsymbol L}\in H^2(X;\mathbb Z)\cong \operatorname{Hom}(H_2(X;\mathbb Z),\mathbb Z)\) by
\[
\langle c_{\boldsymbol L},[\widehat{\Sigma}_i]\rangle=\rot(L_i).
\]
Under the restriction map \(H^2(X;\mathbb Z)\to H^2(M;\mathbb Z)\), the class \(c_{\boldsymbol L}\) maps to \(c_1(\xi)\), see for example~\cite{Ding_Geiges_Stipsicz}. 
It remains to explain how to interpret \(c_{\boldsymbol L}^2\). Since \(X\) has boundary, the square of \(c_{\boldsymbol L}\) is not obtained by simply pairing \(c_{\boldsymbol L}\smile c_{\boldsymbol L}\) with a fundamental class of \(X\). Instead, one needs an absolute-relative pairing. More precisely, one needs a class $\widetilde c_{\boldsymbol L}\in H^2(X,\partial X;\mathbb Q)$ whose image under the natural map \(H^2(X,\partial X;\mathbb Q)\to H^2(X;\mathbb Q)\) is \(c_{\boldsymbol L}\otimes 1\). Then one sets
\[
c_{\boldsymbol L}^2
=
\left\langle c_{\boldsymbol L}\smile \widetilde c_{\boldsymbol L},[X,\partial X]\right\rangle.
\]
The obstruction to the existence of such a rational relative lift is precisely the restriction of \(c_{\boldsymbol L}\) to the boundary. In our situation, this restriction is \(c_1(\xi)\). Thus the assumption that \(c_1(\xi)\) is torsion implies that \(c_{\boldsymbol L}\) vanishes on the boundary after tensoring with \(\mathbb Q\), and hence admits a rational relative lift.

Let us now translate this into the coordinates coming from the surgery diagram. We identify \(H_2(X;\mathbb Q)\) with \(\mathbb Q^n\) using the basis \([\widehat{\Sigma}_1],\ldots,[\widehat{\Sigma}_n]\). Under Poincaré--Lefschetz duality, a rational relative cohomology class may be represented by a vector \(\boldsymbol x\in\mathbb Q^n\). In these coordinates, the map \(H^2(X,\partial X;\mathbb Q)\to H^2(X;\mathbb Q)\) is represented by the linking matrix \(Q\). Therefore the condition that \(\boldsymbol x\) represents a rational relative lift of \(c_{\boldsymbol L}\) is exactly $Q\boldsymbol x=\boldsymbol r$.
For any rational solution $\boldsymbol{x}$ of this equation, $c_{\boldsymbol L}^2
:=
\boldsymbol x^TQ\boldsymbol x
=
\boldsymbol x^T\boldsymbol r
$ is independent of the chosen solution. Indeed, if \(\boldsymbol x'\) is another solution, then \(\boldsymbol y:=\boldsymbol x-\boldsymbol x'\) lies in \(\ker Q\). Hence
\[
\boldsymbol r^T\boldsymbol y
=
(Q\boldsymbol x)^T\boldsymbol y
=
\boldsymbol x^TQ\boldsymbol y
=
0.
\]
Thus \(\boldsymbol x^T\boldsymbol r=(\boldsymbol x')^T\boldsymbol r\), and so \(c_{\boldsymbol L}^2\) is well defined and computes $c_1(J)^2$. Substituting these diagrammatic quantities into the definition of $d_3$ gives the claimed surgery formula.

If $\boldsymbol{L}$ contains $q$ components with contact surgery coefficient $(+1)$, then the surgery trace $X$ is not an almost complex filling of $(M,\xi)$. As explained in~\cite{Ding_Geiges_Stipsicz}, one can replace $X$ by the $q$-fold blow-up $X\#_q\mathbb C P^2$ of $X$ which carries an almost complex structure inducing $\xi$ on $M$. Then the corresponding computation as above yields the claimed formula with the correction term $q(\boldsymbol{L})$. For details, we refer to~\cite{Ding_Geiges_Stipsicz}.
\end{proof}

In what follows, we take the right-hand side of the surgery formula from Proposition~\ref{lem:d3_surgery_formula} as a diagrammatic definition and write
\[
d_3^{\mathrm{surg}}(\boldsymbol L)
:=
\frac14\left(c_{\boldsymbol L}^2-3\sigma(Q_{\boldsymbol L})-2n\right)
+q(\boldsymbol L).
\]
We now prove directly that this diagrammatic quantity is invariant under the moves of Theorem~\ref{thm:main}.

First, we record a simple linear-algebra observation, which will be used repeatedly. It shows, in particular, that the formula for \(d_3^{\mathrm{surg}}\) is invariant under integral changes of basis. Taking \(P\) to be diagonal with diagonal entries \(\pm1\) also shows that the expression is independent of the orientations chosen on the components \(L_i\).

\begin{lemma}\label{lem:d3-basis}
Let \(P\in GL(n;\mathbb Z)\), and set \(Q'=P^TQP\) and \(\boldsymbol r'=P^T\boldsymbol r\). Then the quantity \(\boldsymbol x^TQ\boldsymbol x\), computed from any rational solution of \(Q\boldsymbol x=\boldsymbol r\), agrees with the corresponding quantity \((\boldsymbol x')^TQ'\boldsymbol x'\), computed from any rational solution of \(Q'\boldsymbol x'=\boldsymbol r'\). In particular, the value of \(d_3^{\mathrm{surg}}\) is unchanged under such a change of basis.
\end{lemma}

\begin{proof}
If \(Q\boldsymbol x=\boldsymbol r\), then \(\boldsymbol x'=P^{-1}\boldsymbol x\) satisfies
\[
Q'\boldsymbol x'
=
P^TQPP^{-1}\boldsymbol x
=
P^TQ\boldsymbol x
=
P^T\boldsymbol r
=
\boldsymbol r'.
\]
Moreover, $(\boldsymbol x')^TQ'\boldsymbol x'=\boldsymbol x^TQ\boldsymbol x$.
Since both quantities are independent of the chosen solutions, this proves the first claim. The signature of \(Q\) is unchanged under congruence by \(P\), and the numbers \(n\) and \(q\) are unchanged as well. Hence \(d_3^{\mathrm{surg}}\) is unchanged.
\end{proof}

We now record how the linking matrix and the rotation vector change under the standard contact Kirby moves from Theorem~\ref{thm:main}.

\subsubsection*{Legendrian isotopy}

Legendrian isotopy does not change any of the data entering the formula.

\subsubsection*{Standard contact handle slides}

A standard contact handle slide corresponds, at the level of the surgery trace, to an integral basis change. More precisely, it is represented by an elementary matrix \(P\in GL(n;\mathbb Z)\), with \(1\)'s on the diagonal and one additional off-diagonal entry equal to \(\pm1\), recording the slide. Thus
\[
Q'=P^TQP,
\qquad
\boldsymbol r'=P^T\boldsymbol r.
\]
The number of components and the number of contact \((+1)\)-surgeries are unchanged. Hence the invariance of \(d_3^{\mathrm{surg}}\) under standard contact handle slides follows from Lemma~\ref{lem:d3-basis}.

We now turn to the remaining moves. Write $\boldsymbol L=\boldsymbol L_{\mathrm{loc}}\cup \boldsymbol L_{\mathrm{old}}$, where \(\boldsymbol L_{\mathrm{loc}}\) consists of the components involved in the move, and \(\boldsymbol L_{\mathrm{old}}\) denotes the remaining components. With respect to this decomposition, we write
\[
Q=
\begin{pmatrix}
A & B^T \\
B & C
\end{pmatrix},
\qquad
\boldsymbol r=\binom{\boldsymbol a}{\boldsymbol b}.
\]
Here \(A\) is the linking matrix of the components of $\boldsymbol L_{\mathrm{loc}}$, \(C\) is the linking matrix of the unchanged components $\boldsymbol L_{\mathrm{old}}$, and \(B\) records the linking between these two sets of components. Similarly, \(\boldsymbol a\) and \(\boldsymbol b\) are the corresponding parts of the rotation vector.
After performing the move, the data \((A,\boldsymbol a,B)\) is replaced by \((A',\boldsymbol a',B')\), while \(C\) and \(\boldsymbol b\) remain unchanged. 

For each generating move, we now specify the matrices \(A,A'\), the vectors \(\boldsymbol a,\boldsymbol a'\), and the off-diagonal blocks \(B,B'\). The orientation of the components will be prescribed when discussing the move. Changing these orientations changes the matrices and rotation vectors by an integral change of basis, and therefore does not affect \(d_3^{\mathrm{surg}}\), by Lemma~\ref{lem:d3-basis}.

\subsubsection*{Standard cancelling pair}
We write $\boldsymbol L=\boldsymbol L_{\mathrm{old}}\cup\{L_1^+,L_2^-\}$, where \(L_1\) and \(L_2\) form a standard cancelling pair as in
Figure~\ref{fig:cancelling_intro}. Since the pair is standard, \(L_2\) is obtained from \(L_1\) by translation in the \(\partial_z\)-direction. If we orient both components in the same direction, then we can write $t=\tb(L_1)=\tb(L_2)$ and $r=\rot(L_1)=\rot(L_2)$ and the local data is
\[
A=
\begin{pmatrix}
t+1 & t\\
t & t-1
\end{pmatrix},
\qquad
\boldsymbol a=
\binom{r}{r},
\qquad
B=
\begin{pmatrix}
\boldsymbol\ell & \boldsymbol\ell
\end{pmatrix},
\]
where \(\boldsymbol\ell\) is the column vector whose entries are the linking
numbers of \(L_1\), equivalently of \(L_2\), with the components of
\(\boldsymbol L_{\mathrm{old}}\). Removing the cancelling pair deletes the
local data \((A,\boldsymbol a,B)\).
  
\subsubsection*{Standard lantern move} We choose the ordering and orientation of the standard lantern move as shown in Figure~\ref{fig:d3_lantern_move} and write
\[
\boldsymbol L
=
\boldsymbol L_{\mathrm{old}}
\cup
\{L_1^-,L_2^-,L_3^-\},
\qquad
\boldsymbol L'
=
\boldsymbol L_{\mathrm{old}}
\cup
\{L_1'^-,L_2'^-,L_3'^-,L_4'^-\}.
\]
Then the local matrices and rotation vectors are
\[
A=
\begin{pmatrix}
-5&0&0\\
0&-2&1\\
0&1&-2
\end{pmatrix},
\qquad
\boldsymbol a=
\begin{pmatrix}
-1\\
0\\
0
\end{pmatrix},
\qquad
A'=
\begin{pmatrix}
-3&0&-1&1\\
0&-3&1&-1\\
-1&1&-2&1\\
1&-1&1&-2
\end{pmatrix},
\qquad
\boldsymbol a'=
\begin{pmatrix}
1\\
-1\\
0\\
0
\end{pmatrix}.
\]
It remains to describe the off-diagonal blocks. These record the linking
of the components involved in the move with the components of
\(\boldsymbol L_{\mathrm{old}}\). Let $\boldsymbol w_2^{\mathrm{left}}, \boldsymbol w_2^{\mathrm{right}}, \boldsymbol w_3^{\mathrm{left}}, \boldsymbol w_3^{\mathrm{right}}$
be the column vectors whose entries are the signed crossing contributions
between all the components of \(\boldsymbol L_{\mathrm{old}}\) and, respectively,
the left and right parts of \(L_2\) and \(L_3\), as indicated in
Figure~\ref{fig:d3_lantern_move}. With this notation, we obtain
\begin{align*}
    B=&
\frac12
\begin{pmatrix}
(\boldsymbol w_2^{\mathrm{right}}-\boldsymbol w_2^{\mathrm{left}})
+
(\boldsymbol w_3^{\mathrm{right}}-\boldsymbol w_3^{\mathrm{left}})
&
\boldsymbol w_2^{\mathrm{left}}+\boldsymbol w_2^{\mathrm{right}}
&
\boldsymbol w_3^{\mathrm{left}}+\boldsymbol w_3^{\mathrm{right}}
\end{pmatrix},\\
B'=&
\frac12
\begin{pmatrix}
\boldsymbol w_2^{\mathrm{left}}-\boldsymbol w_3^{\mathrm{right}}
&
\boldsymbol w_2^{\mathrm{right}}-\boldsymbol w_3^{\mathrm{left}}
&
\boldsymbol w_2^{\mathrm{left}}+\boldsymbol w_3^{\mathrm{left}}
&
\boldsymbol w_2^{\mathrm{right}}+\boldsymbol w_3^{\mathrm{right}}
\end{pmatrix}.
\end{align*}

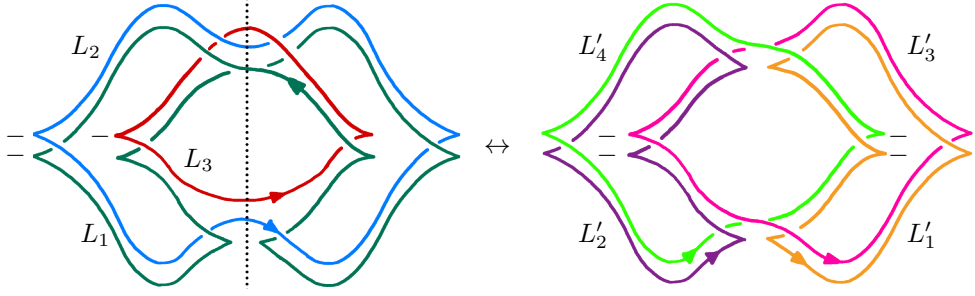
\begin{figure}[htbp]
    \centering
    \begin{overpic}[scale=1]{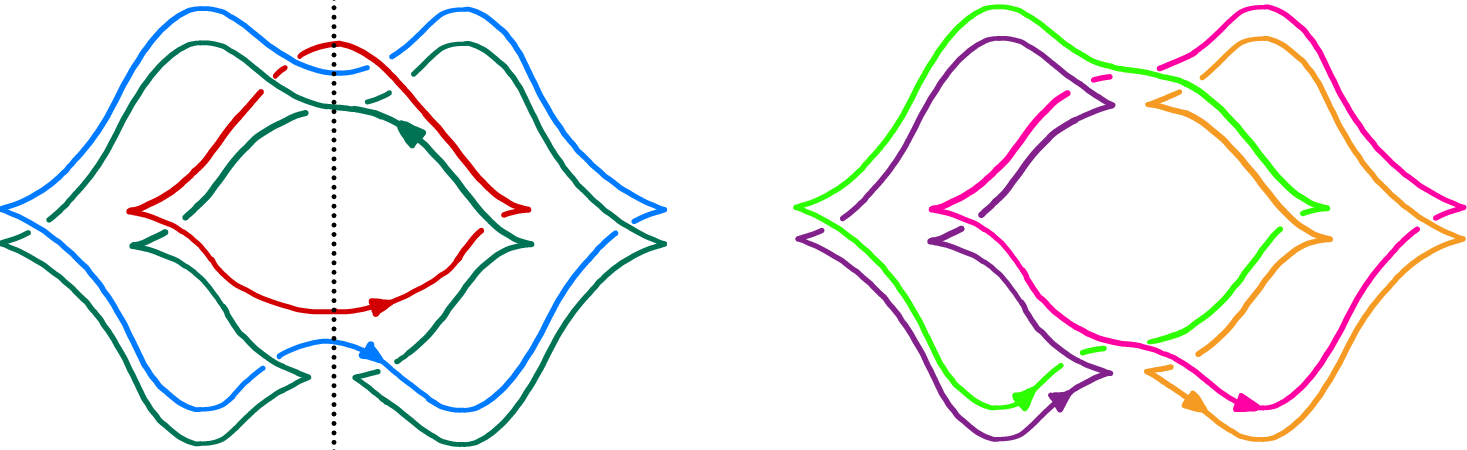}
    \put(5,5){$L_1$}
    \put(4,25){$L_2$}
    \put(16,13){$L_3$}

    \put(93,5){$L'_1$}
    \put(58,5){$L'_2$}
    \put(58,25){$L'_4$}
    \put(93,25){$L'_3$}

    \put(-3,15.5){$-$}
    \put(-3,13.5){$-$}
    \put(6,15.5){$-$}
    \put(60,15.5){$-$}
    \put(60,13.5){$-$}
    \put(91,15.5){$-$}
    \put(91,13.5){$-$}
    \put(48,14.5){$\leftrightarrow$}
    
    \end{overpic}
    \caption{The oriented standard lantern move. The labels and orientations determine the matrices \(A,A'\) and the rotation vectors \(\boldsymbol a,\boldsymbol a'\). The left and right parts of \(L_2\) and \(L_3\) are the parts of these knots on the left and on the right of the dotted line.}\label{fig:d3_lantern_move}
  \end{figure}

\subsubsection*{Standard chain move}
Finally, we consider the standard chain move along the oriented link illustrated in Figure~\ref{fig:d3_chain_move}.
On the left side, the part of the diagram involved in the move consists of twelve Legendrian unknots. With respect to the ordering
\[
(L_1^-,L_4^-,L_7^-,L_{10}^-\mid L_2^-,L_5^-,L_8^-,L_{11}^-\mid L_3^-,L_6^-,L_9^-,L_{12}^-),
\]
and with the orientations shown in the figure, the local linking matrix and rotation vector are
\[
\setlength{\arraycolsep}{3pt}
\renewcommand{\arraystretch}{1.1}
A=
\left(
\begin{array}{cccc|cccc|cccc}
-2&-1&-1&-1&-1&-1&-1&-1&0&0&0&0\\
-1&-2&-1&-1&0&-1&-1&-1&0&0&0&0\\
-1&-1&-2&-1&0&0&-1&-1&0&0&0&0\\
-1&-1&-1&-2&0&0&0&-1&0&0&0&0\\ \hline
-1&0&0&0&-2&-1&-1&-1&-1&-1&-1&-1\\
-1&-1&0&0&-1&-2&-1&-1&0&-1&-1&-1\\
-1&-1&-1&0&-1&-1&-2&-1&0&0&-1&-1\\
-1&-1&-1&-1&-1&-1&-1&-2&0&0&0&-1\\ \hline
0&0&0&0&-1&0&0&0&-2&-1&-1&-1\\
0&0&0&0&-1&-1&0&0&-1&-2&-1&-1\\
0&0&0&0&-1&-1&-1&0&-1&-1&-2&-1\\
0&0&0&0&-1&-1&-1&-1&-1&-1&-1&-2
\end{array}
\right),
\qquad
\boldsymbol a=\boldsymbol 0.
\]
After the chain move, these twelve components are replaced by two components. The corresponding local matrix and rotation vector are
\[
A'=
\begin{pmatrix}
-3 & -2\\
-2 & -3
\end{pmatrix},
\qquad
\boldsymbol a'=
\begin{pmatrix}
-1\\
1
\end{pmatrix}.
\]
It remains to describe the off-diagonal blocks, which record the linking with the components of \(\boldsymbol L_{\mathrm{old}}\). Let
$\boldsymbol\ell_1,\boldsymbol\ell_2,\boldsymbol\ell_3$ be the column vectors recording these linkings for the three congruence classes of components, as indicated in Figure~\ref{fig:d3_chain_move}. Thus the linking vector of \(L_i\) depends only on \(i\) modulo \(3\), and
\[
B=
\left(
\begin{array}{cccc|cccc|cccc}
\boldsymbol\ell_1 & \boldsymbol\ell_1 & \boldsymbol\ell_1 & \boldsymbol\ell_1 &
\boldsymbol\ell_2 & \boldsymbol\ell_2 & \boldsymbol\ell_2 & \boldsymbol\ell_2 &
\boldsymbol\ell_3 & \boldsymbol\ell_3 & \boldsymbol\ell_3 & \boldsymbol\ell_3
\end{array}
\right).
\]
After the move, the two new components have linking vectors
\[
B'=
\left(
\begin{array}{cc}
\boldsymbol\ell_1+\boldsymbol\ell_3 &
\boldsymbol\ell_1+\boldsymbol\ell_3
\end{array}
\right).
\]

\begin{figure}[htbp]
    \centering
    \begin{overpic}[scale=1]{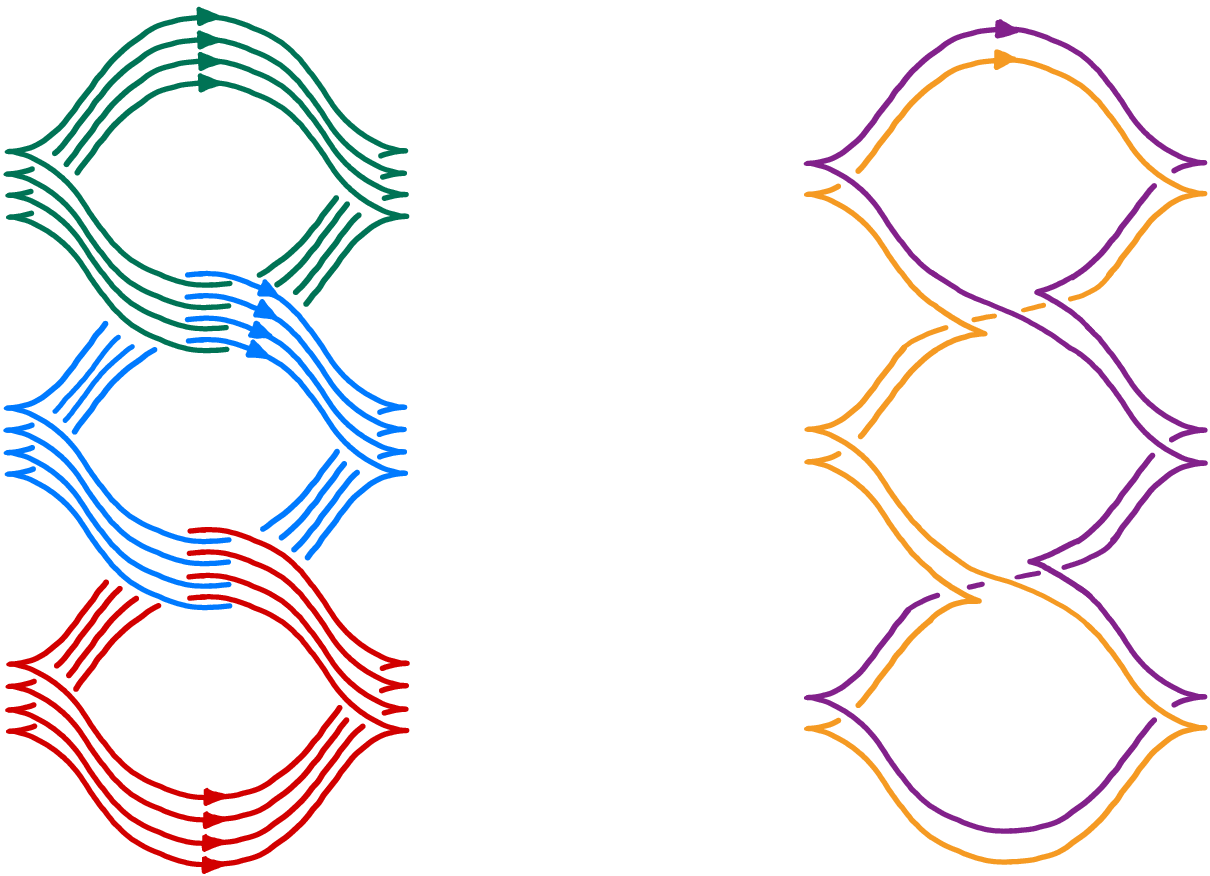}
    \put(34.5,53.7){$\rightarrow L_1$}
    \put(34.5,32.3){$\rightarrow L_2$}
    \put(34.5,11){$\rightarrow L_3$}
    \put(75,32){$L'_1$}
    \put(94,42){$L'_2$}
    
    \put(-3,59){$-$}
    \put(-3,57.1){$-$}
    \put(-3,55.5){$-$}
    \put(-3,53.5){$-$}
    \put(-3,37.8){$-$}
    \put(-3,35.8){$-$}
    \put(-3,34.2){$-$}
    \put(-3,32.2){$-$}
    \put(-3,16.5){$-$}
    \put(-3,14.6){$-$}
    \put(-3,13){$-$}
    \put(-3,11){$-$}
    \put(48,35){$\leftrightarrow$}
    \put(63,58){$-$}
    \put(63,55.5){$-$}
    \end{overpic}
    \caption{The oriented standard chain move. Only the components \(L_1\), \(L_2\), and \(L_3\) are labelled explicitly. The remaining components are labelled by moving upwards through the diagram: \(L_4\) is the second green unknot, \(L_5\) is the second blue unknot, and so on, ending with \(L_{12}\), the highest red unknot.}\label{fig:d3_chain_move}
\end{figure}
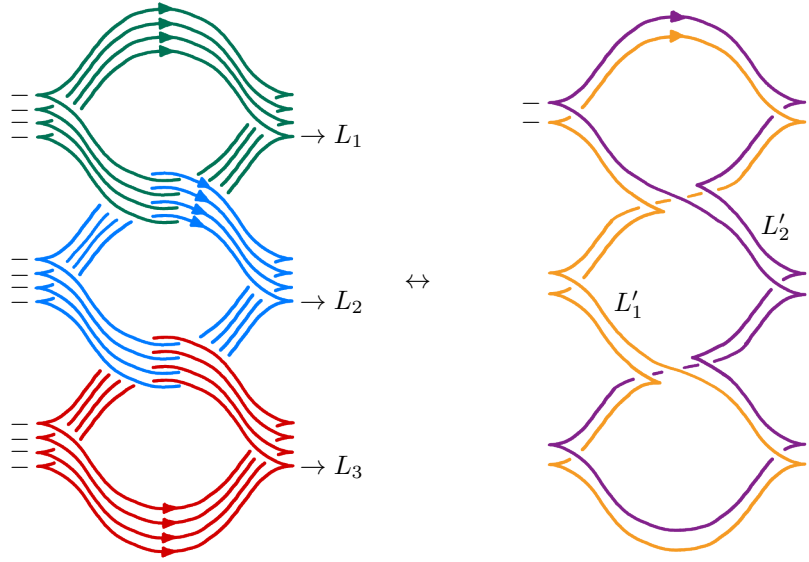

In the following lemma, we use Schur complements to isolate the effect of these moves on the \(d_3\)-formula.

\begin{lemma}\label{lem:Schur_local}
Let
\[
Q=\begin{pmatrix} A & B^T \\ B & C \end{pmatrix},
\qquad
\boldsymbol r=\binom{\boldsymbol a}{\boldsymbol b},
\qquad
Q'=\begin{pmatrix} A' & (B')^T \\ B' & C \end{pmatrix},
\qquad
\boldsymbol r'=\binom{\boldsymbol a'}{\boldsymbol b},
\]
where \(A\) and \(A'\) are symmetric and invertible\footnote{Here, we use the convention that the empty matrix is invertible and has signature zero. With this convention, Lemma~\ref{lem:Schur_local} also applies when one of the two parts involved in the move is empty.} such that
\[
B A^{-1} B^T = B'(A')^{-1}(B')^T
\qquad\text{and}\qquad
B A^{-1}\boldsymbol a = B'(A')^{-1}\boldsymbol a'.
\]
Then \(Q\boldsymbol x=\boldsymbol r\) admits a rational solution if and only if \(Q'\boldsymbol x'=\boldsymbol r'\) does. In that case, we have
\begin{align*}
    (c')^2-c^2 &= (\boldsymbol a')^T(A')^{-1}\boldsymbol a' - \boldsymbol a^TA^{-1}\boldsymbol a, \textrm{ and }\\
\sigma(Q')-\sigma(Q) &=\sigma(A')-\sigma(A).
\end{align*}
In particular, it follows that if
\[
(\boldsymbol a')^T(A')^{-1}\boldsymbol a'
-
\boldsymbol a^TA^{-1}\boldsymbol a
=
3\bigl(\sigma(A')-\sigma(A)\bigr)
+
2\bigl(\operatorname{dim} A'-\operatorname{dim} A\bigr)
-
4(q'-q),
\]
then the corresponding values of the diagrammatic \(d_3\)-formula agree, i.e.\
\[
d_3^{\mathrm{surg}}(Q',\boldsymbol r',q')
=
d_3^{\mathrm{surg}}(Q,\boldsymbol r,q).
\]
Here \(d_3^{\mathrm{surg}}(Q,\boldsymbol r,q)\) denotes the right-hand side of the surgery formula written in terms of the matrix \(Q\), the rotation vector \(\boldsymbol r\), and the number \(q\) of contact \((+1)\)-surgeries.
\end{lemma}

\begin{proof}
We write \(\boldsymbol x=\binom{\boldsymbol z}{\boldsymbol y}\). Solving the first block equation and substituting into the second shows that \(Q\boldsymbol x=\boldsymbol r\) has a solution $\boldsymbol{x}$ if and only if the equation
\[
(C-BA^{-1}B^T)\boldsymbol y
=
\boldsymbol b-BA^{-1}\boldsymbol a
\]
has a solution $\boldsymbol{y}$; and in that case $\boldsymbol{z}$ is given by $\boldsymbol z=A^{-1}(\boldsymbol a-B^T\boldsymbol y)$.
The same computation shows that \(Q'\boldsymbol x'=\boldsymbol r'\) has a solution if and only if 
\[
(C-B'(A')^{-1}(B')^T)\boldsymbol y
=
\boldsymbol b-B'(A')^{-1}\boldsymbol a',
\]
has a solution. By assumption, these two equations agree, so the solvability of the two systems agrees.
If this is the case, we compute
\[
\boldsymbol r^T\boldsymbol x
=
\boldsymbol a^TA^{-1}\boldsymbol a
+
(\boldsymbol b-BA^{-1}\boldsymbol a)^T\boldsymbol y,
\]
and the same formula holds for \((Q',\boldsymbol r')\). Hence
\[
(\boldsymbol r')^T\boldsymbol x' - \boldsymbol r^T\boldsymbol x
=
(\boldsymbol a')^T(A')^{-1}\boldsymbol a'
-
\boldsymbol a^TA^{-1}\boldsymbol a,
\]
which proves the claimed formula for \((c')^2-c^2\).
Finally, since \(A\) is invertible, \(Q\) is congruent over \(\mathbb R\) to
\[
A\oplus(C-BA^{-1}B^T),
\]
and similarly \(Q'\) is congruent over \(\mathbb R\) to
\[
A'\oplus(C-B'(A')^{-1}(B')^T).
\]
The Schur complements agree by assumption, and therefore
\[
\sigma(Q')-\sigma(Q)=\sigma(A')-\sigma(A).
\]
The last claim follows by substituting these identities into the definition of \(d_3^{\mathrm{surg}}\).
\end{proof}

We now use the lemma to prove that \(d_3^{\mathrm{surg}}\) is independent of the contact surgery link representing the contact manifold.

\begin{proof}[Proof of Proposition~\ref{prop:d3-invariance}]
By Theorem~\ref{thm:main}, it suffices to check invariance under the standard contact Kirby moves.
Legendrian isotopy does not change any of the data entering the formula, and hence does not change \(d_3^{\mathrm{surg}}\).
For a standard contact handle slide, the data changes by an integral change of basis, i.e.\
\[
Q'=P^TQP,
\qquad
\boldsymbol r'=P^T\boldsymbol r.
\]
By Lemma~\ref{lem:d3-basis}, the term \(c^2\) is unchanged. The signature, the number of components, and the number of contact \((+1)\)-surgeries are also unchanged. Hence \(d_3^{\mathrm{surg}}\) is unchanged.

It remains to consider insertions or removals of standard cancelling pairs, the standard lantern move, and the standard chain move. For these moves, we use the local data recorded above. In each case, a direct computation gives
\begin{align*}
    &BA^{-1}B^T=B'(A')^{-1}(B')^T,
\qquad
BA^{-1}\boldsymbol a=B'(A')^{-1}\boldsymbol a',\\
&(\boldsymbol a')^T(A')^{-1}\boldsymbol a'
-
\boldsymbol a^TA^{-1}\boldsymbol a
=
3\bigl(\sigma(A')-\sigma(A)\bigr)
+
2\bigl(\operatorname{dim} A'-\operatorname{dim} A\bigr)
-
4(q'-q),
\end{align*}
where \(q'-q\) denotes the change in the number of contact \((+1)\)-surgery components. Therefore Lemma~\ref{lem:Schur_local} applies and gives $d_3^{\mathrm{surg}}(Q',\boldsymbol r',q')=d_3^{\mathrm{surg}}(Q,\boldsymbol r,q)$.
\end{proof}

\begin{remark}
    On rational homology spheres, the homotopy class of an oriented tangential \(2\)-plane field is determined by its \(d_3\)-invariant together with Gompf's \(\Gamma\)-invariant, whose mod \(2\) reduction agrees with the first Chern class. Moreover, there exist explicit formulas computing both the \(\Gamma\)-invariant and the first Chern class directly from contact surgery diagrams~\cite{Gompf,Ding_Geiges_Stipsicz,Etnyre_Kegel_Onaran}. It is also possible to verify the invariance of these quantities under contact Kirby moves directly.
\end{remark}

\subsection{A modulo \(8\) invariant of contact structures} 
\label{sec:mod8inv}

The computations used to prove the invariance of \(d_3^{\mathrm{surg}}\) contain a little more information. Indeed, although the individual quantities \(\sigma(Q_{\boldsymbol L})\) and \(c_{\boldsymbol L}^2\) depend on the chosen contact surgery presentation, their difference is well defined modulo \(8\).

Still using the notation from Section~\ref{sec:diagrammatic_d3}, let \(\boldsymbol L\) be a contact surgery link representing \((M,\xi)\) with torsion first Chern class \(c_1(\xi)\). Define
\[
\delta(\boldsymbol L)=c_{\boldsymbol L}^2-\sigma(\boldsymbol L),
\]
where \(\sigma(\boldsymbol L)=\sigma(Q_{\boldsymbol L})\) and \(c_{\boldsymbol L}^2=\boldsymbol x^TQ_{\boldsymbol L}\boldsymbol x\) for any rational solution \(\boldsymbol x\) of \(Q_{\boldsymbol L}\boldsymbol x=\boldsymbol r_{\boldsymbol L}\).

\begin{proposition}\label{prop:delta_mod8}
Let \(\boldsymbol L\) and \(\boldsymbol L'\) be two contact surgery links representing the same contact structure \((M,\xi)\) with \(c_1(\xi)\) torsion. Then
\[
\delta(\boldsymbol L)-\delta(\boldsymbol L')\equiv 0 \pmod 8.
\]
Thus \(\delta(\xi):=\delta(\boldsymbol L)\in \mathbb Q/8\mathbb Z\) is an invariant of the contact structure \(\xi\).
\end{proposition}

\begin{proof}
By Theorem~\ref{thm:main}, any two contact surgery presentations of the same contact structure are related by the standard contact Kirby moves. Planar isotopies and Legendrian Reidemeister moves do not change any of the data entering the definition of \(\delta\). Standard contact handle slides preserve both \(\sigma\) and \(c^2\), since they act by integral changes of basis.

It remains to consider insertions or removals of standard cancelling pairs, the standard lantern move, and the standard chain move. Using the local matrices and rotation vectors from above for insertion of a standard cancelling pair, the standard lantern move, and the standard chain move, with the chosen
directions of the moves, we can compute the changes of \(\delta=c^2-\sigma\) as 
\begin{itemize}
    \item $0-0=0$, for the cancelling pair,
    \item $(-1)-(-1)=0$, for the standard lantern move, and
    \item $(-2)-6=-8$, for the standard chain move.   
\end{itemize}
Reversing a move changes the sign of the corresponding change, so the class of \(\delta\) in \(\mathbb Q/8\mathbb Z\) is unchanged in all cases. This proves the claim.
\end{proof}

We note that the invariance of \(\delta\in \mathbb Q/8\mathbb Z\) also has a four-dimensional interpretation. This is not needed for the proof of Proposition~\ref{prop:delta_mod8}, but it explains why a modulo \(8\) congruence is natural.

An oriented contact structure \(\xi\) on \(M\) determines a \(\operatorname{Spin}^c\) structure \(\mathfrak s_\xi\), with \(c_1(\mathfrak s_\xi)=c_1(\xi)\). On the other hand, the contact surgery diagram \(\boldsymbol L\) determines the class \(c_{\boldsymbol L}\in H^2(X_{\boldsymbol L};\mathbb Z)\) represented by the rotation vector. This class is characteristic. Indeed, $\rot(L_i)\equiv \tb(L_i)\pm1 \pmod 2$, and \(\tb(L_i)\pm1\) is precisely the smooth framing of the \(2\)-handle attached along \(L_i\), hence the self-intersection of \([\widehat\Sigma_i]\).

It follows that \(c_{\boldsymbol L}\) is the first Chern class of some \(\operatorname{Spin}^c\) structure on \(X_{\boldsymbol L}\). Since \(X_{\boldsymbol L}\) is obtained from \(D^4\) by attaching only \(2\)-handles, the group \(H^2(X_{\boldsymbol L};\mathbb Z)\) has no \(2\)-torsion. Hence this \(\operatorname{Spin}^c\) structure is uniquely determined by its first Chern class. We denote it by \(\mathfrak t_{\boldsymbol L}\), so that \(c_1(\mathfrak t_{\boldsymbol L})=c_{\boldsymbol L}\). Its restriction to \(M\) is the \(\operatorname{Spin}^c\) structure determined by the contact structure \(\xi\), namely \(\mathfrak s_\xi\).

Thus Proposition~\ref{prop:delta_mod8} can also be viewed as a special case of the following standard fact.

\begin{proposition}\label{lem:spinc_mod8}
Let \((M,\mathfrak s)\) be a closed oriented \(\operatorname{Spin}^c\) \(3\)-manifold with \(c_1(\mathfrak s)\) torsion. Then the value
\[
\sigma(X)-c_1(\mathfrak t)^2 \pmod 8
\]
is independent of the choice of compact, oriented \(\operatorname{Spin}^c\) \(4\)-manifold \((X,\mathfrak t)\) such that \(\partial X=M\) and \(\mathfrak t|_M=\mathfrak s\). Here \(c_1(\mathfrak t)^2\) denotes the rational square defined by choosing a rational relative lift of \(c_1(\mathfrak t)\), which exists because \(c_1(\mathfrak s)\) is torsion.
\end{proposition}

\begin{proof}
Let \((X,\mathfrak t)\) and \((X',\mathfrak t')\) be two such \(\operatorname{Spin}^c\) fillings of \((M,\mathfrak s)\). Form the closed, oriented \(4\)-manifold \(W=X\cup_M(-X')\). Since \(\mathfrak t|_M=\mathfrak s=\mathfrak t'|_M\), the two \(\operatorname{Spin}^c\) structures glue to a \(\operatorname{Spin}^c\) structure \(\mathfrak t_W\) on \(W\).
By additivity of the signature and of the square of the first Chern class under gluing, we have
\[
\sigma(W)-c_1(\mathfrak t_W)^2
=
\bigl(\sigma(X)-c_1(\mathfrak t)^2\bigr)
-
\bigl(\sigma(X')-c_1(\mathfrak t')^2\bigr),
\]
where the minus sign comes from the reversed orientation on \(-X'\).

For any \(\operatorname{Spin}^c\) structure on a closed oriented \(4\)-manifold, the first Chern class is characteristic~\cite[Proposition~2.4.16]{Gompf_Stipsicz}. Hence \(c_1(\mathfrak t_W)\) is a characteristic element for the unimodular intersection form of \(W\). By~\cite[Lemma~1.2.20]{Gompf_Stipsicz}, every characteristic vector \(k\) of a unimodular intersection form satisfies
\[
k^2\equiv \sigma \pmod 8.
\]
Applying this to \(k=c_1(\mathfrak t_W)\), we obtain
\[
c_1(\mathfrak t_W)^2\equiv \sigma(W)\pmod 8.
\]
Therefore \(\sigma(W)-c_1(\mathfrak t_W)^2\equiv 0\pmod 8\). Using the additivity formula above, we conclude that
\[
\sigma(X)-c_1(\mathfrak t)^2
\equiv
\sigma(X')-c_1(\mathfrak t')^2
\pmod 8.
\]
This proves the desired independence.
\end{proof}

As a direct corollary, we obtain another homotopical invariant of contact structures. 

\begin{proposition} Let $(M,\xi)$ be a contact manifold with torsion first Chern class. For an almost complex filling \((X,J)\) of \((M,\xi)\), set $\varepsilon_X(M,\xi):=-b_2^+(X)$.
\begin{enumerate}
    \item The parity of $\varepsilon(M,\xi)$ is independent of the choice of $(X,J)$ and thus an invariant of $(M,\xi)$.
    \item For any almost complex filling of $(M,\xi)$, we have
    \begin{equation*}
        d_3(M,\xi)=\frac{1}{4}\delta(M,\xi)+\varepsilon(M,\xi).
    \end{equation*}
    \item If $M$ is an integral homology sphere, then $\delta(M,\xi)$ is an integer divisible by four for any almost complex filling of $(M,\xi)$.
\end{enumerate}
\end{proposition}

\begin{proof}
    For any almost complex filling $(X,J)$ of $(M,\xi)$, we have
    \[
    \varepsilon(M,\xi)=-b_2^+(X)=-\frac{1}{2}\big(\sigma(X)+\chi(X)-1\big).
    \]
    Thus adding $\frac{1}{4}\delta(M,\xi)$ and $\varepsilon(M,\xi)$ yields the $d_3$-invariant, as claimed in (2). Statement (1) follows because $\delta(M,\xi)\in\mathbb Q/8\Z$ is an invariant of $(M,\xi)$, and (3) follows from the observation that $d_3$ takes integral values on integral homology spheres.
\end{proof}


\section{Relations between contact Kirby moves}\label{sec:proof_rel_kirbymoves}

In this section, we study the relation between the various contact Kirby moves from Section~\ref{sec:moves} and prove the independence of the standard contact Kirby moves.

\begin{proof}[Proof of Theorem~\ref{thm:indep}] For each of the standard contact Kirby moves, consider the change vector
\[
\Delta(\boldsymbol L,\boldsymbol L')
=
\bigl(
n'-n,\,
\sigma(\boldsymbol L')-\sigma(\boldsymbol L),\,
q'(\boldsymbol L')-q(\boldsymbol L),\,
c_{\boldsymbol L'}^2-c_{\boldsymbol L}^2
\bigr)
\in\mathbb Q^4,
\]
where \(n\) denotes the number of components, \(q(\boldsymbol L)\) denotes the number of contact \((+1)\)-surgery components, $\sigma(\boldsymbol L)=\sigma(Q_{\boldsymbol L})$ is the signature of the linking matrix, and the term \(c_{\boldsymbol L}^2\) is computed from the rotation vector as in Section~\ref{sec:diagrammatic_d3}. The change vector is zero for planar isotopies, Legendrian Reidemeister moves, and the standard contact handle slides.

Using the local matrices and rotation vectors from Section~\ref{sec:diagrammatic_d3}, the change vectors for insertion of a standard cancelling pair, the standard lantern move, and the standard chain move are, for the chosen directions of the moves,
\[
P=(2,0,1,0),\qquad
L=(1,-1,0,-1),\qquad
C=(-10,6,0,-2).
\]
Reversing one of the moves changes the sign of the corresponding vector. Since the vectors \(P,L,C\) are linearly independent in \(\mathbb Q^4\), none of them lies in the span of the other two. Therefore, none of the three moves can be expressed as a sequence of the other two together with planar isotopies, Legendrian Reidemeister moves, and standard contact handle slides.

It remains to show that contact handle slides cannot be expressed by the other contact Kirby moves. For this, define $N(\boldsymbol{L})$ to be the number of components of a contact surgery link $\boldsymbol{L}$ that are smooth unknots in $S^3$. Legendrian isotopies, standard lantern moves, and standard chain moves preserve $N$, while insertion or removal of a standard cancelling pair preserves the parity of $N$. On the other hand, any contact $3$-manifold admits contact surgery links $\boldsymbol{L}$ and $\boldsymbol{L}'$ differing only by a single handle slide such that $N(\boldsymbol{L}')=N(\boldsymbol{L})\pm1$. Indeed, start with any surgery link $\boldsymbol{L}_0$. To create $\boldsymbol{L}$, add two cancelling pairs $L_1^+\cup L_1^-$ and $L_2^+\cup L_2^-$ to $\boldsymbol{L}_0$ such that $L_1^{\pm}$ are unknots and $L_2^{\pm}$ are nontrivial smooth knots. Performing a handle slide of $L_1^-$ over $L_2^-$ creates a surgery link $\boldsymbol{L}'$ of the same contact manifold with one fewer unknotted component.
\end{proof}

\begin{remark}
Using the same methods, it is straightforward to show that the more general chain moves, lantern moves, and lantern destabilizations cannot be expressed as sequences of handle slides, cancellations, and Legendrian isotopies. It is also quite conceivable that the full Theorem~\ref{thm:indep} remains true after removing the adjective ``standard'' throughout. 
\end{remark}

We next ask how the more general contact Kirby moves discussed in Section~\ref{sec:moves} can be expressed in terms of the standard contact Kirby moves. Here we prove the following.

\begin{proposition}\label{prop:relations_known_moves}\hfill
\begin{enumerate}
    \item The more general contact handle slides shown in Figures~\ref{fig:contact_handle_slides1} and~\ref{fig:contact_handle_slides2} (with the blue curves equipped with contact surgery coefficients $(\pm1)$) can be expressed in terms of the two standard contact handle slides from Figure~\ref{fig:braid_intro}, together with insertions or removals of standard cancelling pairs.
    \item The lantern destabilizations can be expressed as sequences of the contact handle slides shown in Figures~\ref{fig:contact_handle_slides1} and~\ref{fig:contact_handle_slides2}, insertions or removals of standard cancelling pairs, and lantern moves.
\end{enumerate}
\end{proposition}

\begin{proof}
We first prove (1). The contact handle slides shown in Figures~\ref{fig:contact_handle_slides1} and~\ref{fig:contact_handle_slides2} are diagrammatic incarnations of braid-type relations in the mapping class group. Let \(\aa\) and \(\bb\) be simple closed curves intersecting transversely in one point. The two standard contact handle slides used in Theorem~\ref{thm:main} correspond to the basic braid relation
\[
\tau_\aa^+\tau_\bb^+
=
\tau_{\tau_\aa^+(\bb)}^+\tau_\aa^+.
\]
At the level of contact surgery diagrams, inserting or deleting a standard cancelling pair corresponds to inserting or deleting a pair of inverse Dehn twists.

The other contact handle slides arise from the same braid relation, but with one of the twists inverted or with the order of the two twists interchanged. More precisely, they correspond to relations of the form
\[
\tau_\aa^\varepsilon\tau_\bb^\delta
=
\tau_{\tau_\aa^\varepsilon(\bb)}^\delta\tau_\aa^\varepsilon,
\qquad
\tau_\aa^\delta\tau_\bb^\varepsilon
=
\tau_\bb^\varepsilon\tau_{\tau_\bb^{-\varepsilon}(\aa)}^\delta,
\]
where \(\varepsilon,\delta\in\{+,-\}\). These relations follow from the basic braid relation together with the insertion and deletion of inverse pairs.
We illustrate the argument in one representative case, shown in Figure~\ref{fig:contact_handle_as_standard}.

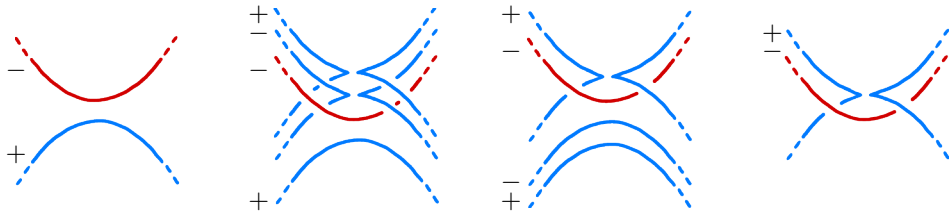
\begin{figure}[htbp]
    \centering
    \begin{overpic}[scale=1]{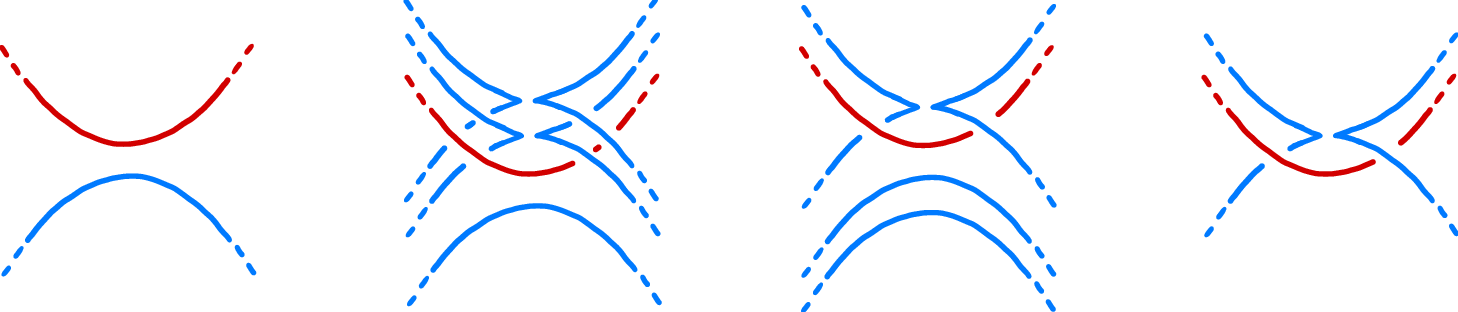}
    \put(-1,14){$-$}
    \put(-1,5){$+$}
    
    \put(25,20){$+$}
    \put(25,18){$-$}
    \put(25,14){$-$}
    \put(25,0){$+$}

    \put(52,20){$+$}
    \put(52,16){$-$}
    \put(52,2){$-$}
    \put(52,0){$+$}

    \put(80,18){$+$}
    \put(80,16){$-$}
    
    \end{overpic}
    \caption{A contact handle slide expressed as the insertion of a standard cancelling pair, a standard contact handle slide, and the removal of a standard cancelling pair.}
    \label{fig:contact_handle_as_standard}
\end{figure}

Set \(\aa\bb=\tau_\aa^+(\bb)\). Then
\[
\begin{aligned}
\tau_\aa^+\tau_\bb^-
&=
\tau_{\aa\bb}^-\tau_{\aa\bb}^+\tau_\aa^+\tau_\bb^- \\[2pt]
&=
\tau_{\aa\bb}^-\bigl(\tau_{\aa\bb}^+\tau_\aa^+\bigr)\tau_\bb^- \\[2pt]
&=
\tau_{\aa\bb}^-\bigl(\tau_\aa^+\tau_\bb^+\bigr)\tau_\bb^-
\qquad\text{by the braid relation, applied backwards} \\[2pt]
&=
\tau_{\aa\bb}^-\tau_\aa^+.
\end{aligned}
\]
In diagrammatic terms, the first equality inserts a cancelling pair, the use of the braid relation is exactly one of the standard contact handle slides, and the final cancellation deletes a cancelling pair. This proves that the corresponding contact handle slide is a composition of standard cancelling-pair moves and a standard contact handle slide.

The remaining cases are obtained in the same way, either by applying the same computation to the inverse braid relation, by interchanging the roles of \(\aa\) and \(\bb\), or by reversing the direction of the move. Hence, every contact handle slide shown in Figures~\ref{fig:contact_handle_slides1} and~\ref{fig:contact_handle_slides2} can be expressed as a sequence of standard contact handle slides together with insertions and removals of standard cancelling pairs.

We now prove (2) for the lantern destabilization with \(k=0\) strands; the general case is obtained similarly.
Let the curves \(\aa,\bb,\aa\bb,\dd_1,\dd_2,\dd_3,\dd_4\) be as in the lantern relation of Example~\ref{ex:lantern_relations}. With our conventions for translating contact surgery coefficients into Dehn twists, the contact surgery link on the left of Figure~\ref{fig:lantern_dest_proof1} represents the factorization
\[
\tau_{\dd_1}^+\tau_{\dd_2}^+\tau_{\dd_4}^+\tau_{\aa\bb}^-.
\]
The lantern relation, together with the two positive destabilizations used in the construction of the Lisca--Stipsicz lantern destabilization~\cite{Lisca_Stipsicz_lantern}, transforms this factorization into $\tau_{\dd_3}^+$, which is represented by the contact surgery link on the right of Figure~\ref{fig:lantern_dest_proof4}.

\begin{figure}[htbp]
    \centering
    \begin{overpic}[scale=1]{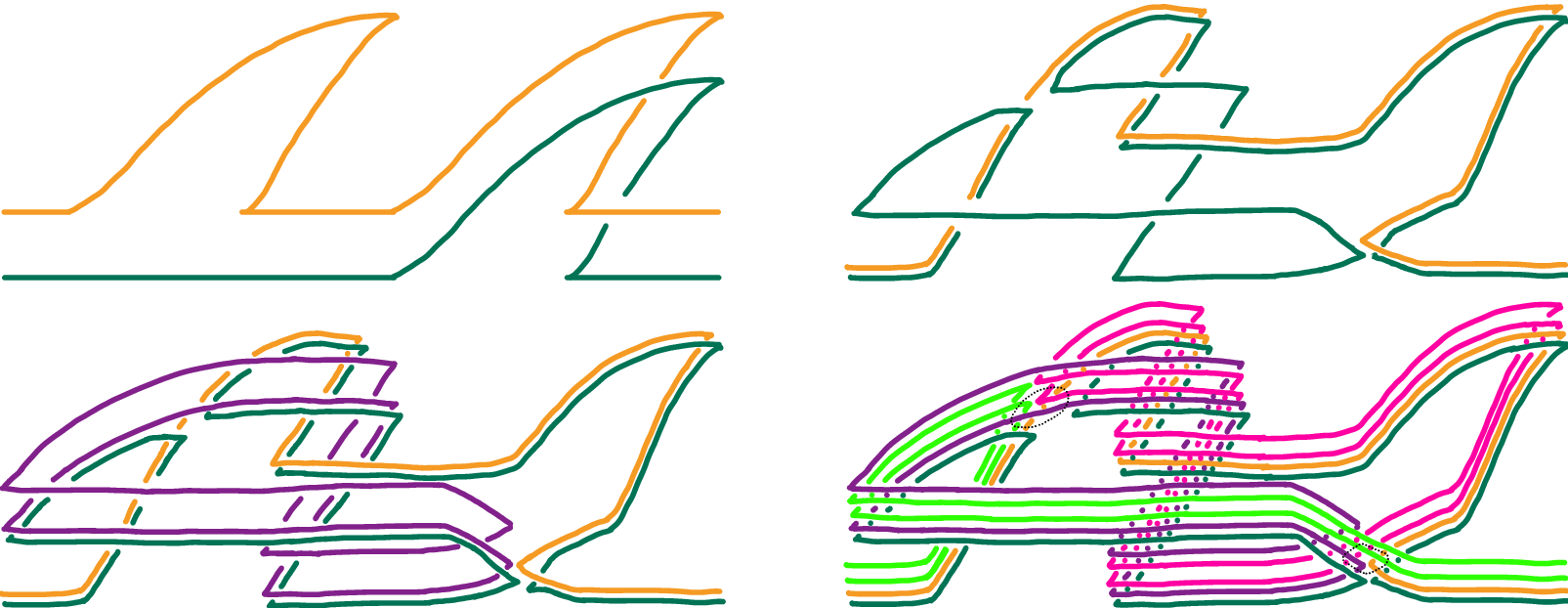}
        \put(-2.5,20.5){+}
        \put(-2.5,24.5){$-$}

        \put(51.5,21.1){$-$}
        \put(51.5,24.5){+}

        \put(-2.5,0.2){$-$}
        \put(-2.5,3.6){$+$}
        \put(-2.5,7){$+$}
        \put(25.5,12.5){$-$}

        \put(51.5,0.2){$-$}
        \put(51.5,3.6){$+$}
        \put(51.5,5.2){$-$}
        \put(51.5,7){$+$}
        \put(79.5,12.4){$-$}
        \put(79.5,14.2){$+$}
        \put(100.2,2.4){$+$}
        \put(100.2,17.5){$-$}
    \end{overpic}
    \caption{Steps (1)--(4) in the proof of the lantern destabilization: the initial diagram, a Legendrian isotopy, the insertion of one standard cancelling pair, and the insertion of two further standard cancelling pairs.}\label{fig:lantern_dest_proof1}
\end{figure}
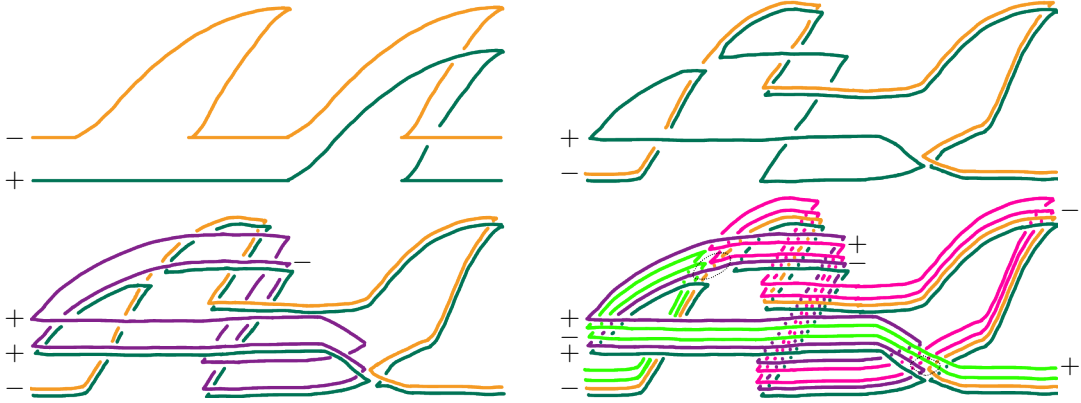

Figures~\ref{fig:lantern_dest_proof1}--\ref{fig:lantern_dest_proof4} give a diagrammatic realization of this transformation using only the allowed moves. More precisely, Figure~\ref{fig:lantern_dest_proof1} starts with the diagram for
\(\tau_{\dd_1}^+\tau_{\dd_2}^+\tau_{\dd_4}^+\tau_{\aa\bb}^-\), then performs a Legendrian isotopy and inserts three standard cancelling pairs. Figure~\ref{fig:lantern_dest_proof2} applies a lantern move, removes one standard cancelling pair, and then performs Legendrian isotopies of the purple and red components. Figure~\ref{fig:lantern_dest_proof3} continues with a Legendrian isotopy of the pink component and then performs a contact handle slide, followed by the removal of a standard cancelling pair. Finally, Figure~\ref{fig:lantern_dest_proof4} performs one more Legendrian isotopy and a second contact handle slide, again followed by the removal of a standard cancelling pair.

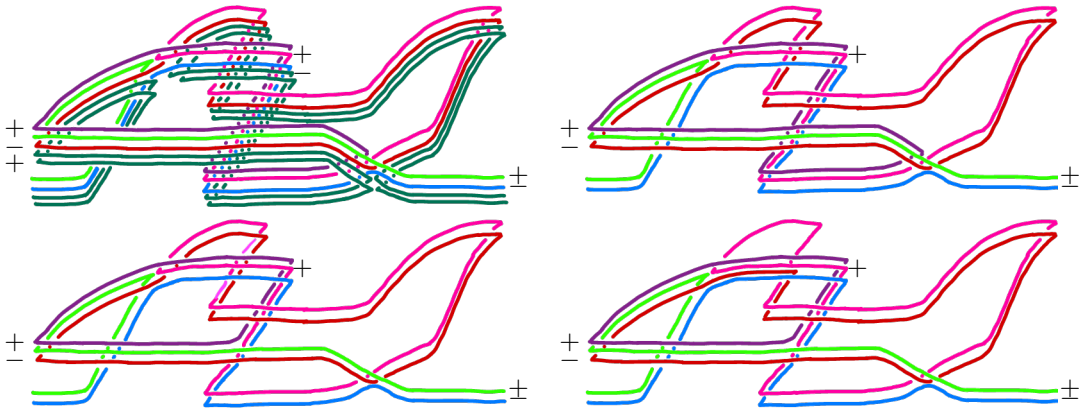
\begin{figure}[htbp]
    \centering
    \begin{overpic}[scale=1]{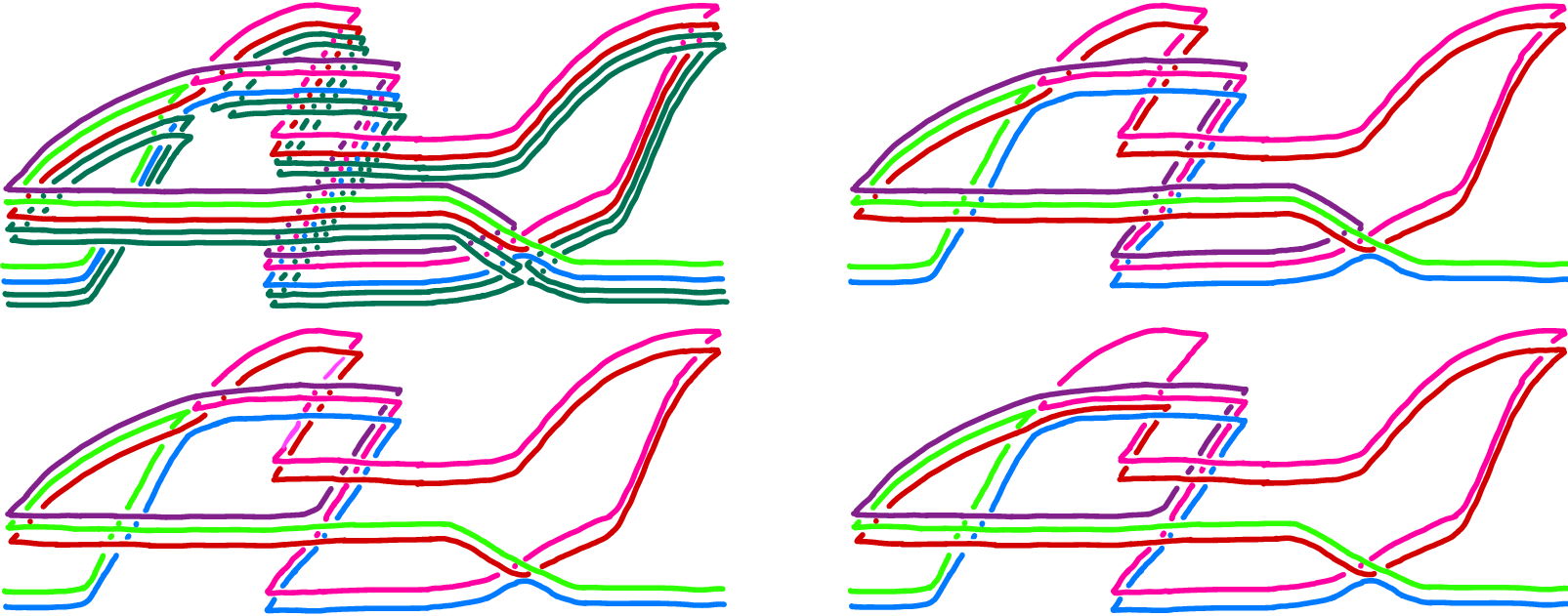}
        \put(-2.5,23.2){$+$}
        \put(-2.5,24.8){$-$}
        \put(-2.5,26.6){$+$}
        \put(25.5,32){$-$}
        \put(25.5,33.8){$+$}
        \put(46.5,22){$+$}
        \put(46.5,21){$-$}

        \put(51.5,24.8){$-$}
        \put(51.5,26.6){$+$}
        \put(79.5,33.8){$+$}
        \put(100.2,22){$+$}
        \put(100.2,21){$-$}

        \put(-2.5,4){$-$}
        \put(-2.5,5.8){$+$}
        \put(25.5,13){$+$}
        \put(46.5,1.3){$+$}
        \put(46.5,0.3){$-$}

        \put(51.5,4){$-$}
        \put(51.5,5.8){$+$}
        \put(79.5,13){$+$}
        \put(100.2,1.3){$+$}
        \put(100.2,0.3){$-$}
    \end{overpic}
    \caption{Steps (5)--(8): a lantern move, the removal of a standard cancelling pair, and Legendrian isotopies of the purple and red components.}\label{fig:lantern_dest_proof2}
\end{figure}

The final diagram is precisely the diagram representing \(\tau_{\dd_3}^+\). Hence the \(k=0\) lantern destabilization is expressed as a sequence of Legendrian isotopies, insertions and removals of standard cancelling pairs, one lantern move, and the contact handle slides shown in Figures~\ref{fig:contact_handle_slides1} and~\ref{fig:contact_handle_slides2}. As explained above, the case with \(k\) additional strands is obtained by applying the same sequence while carrying those strands along unchanged. This proves the claim.
\end{proof}

\begin{figure}[htbp]
    \centering
    \begin{overpic}[scale=1]{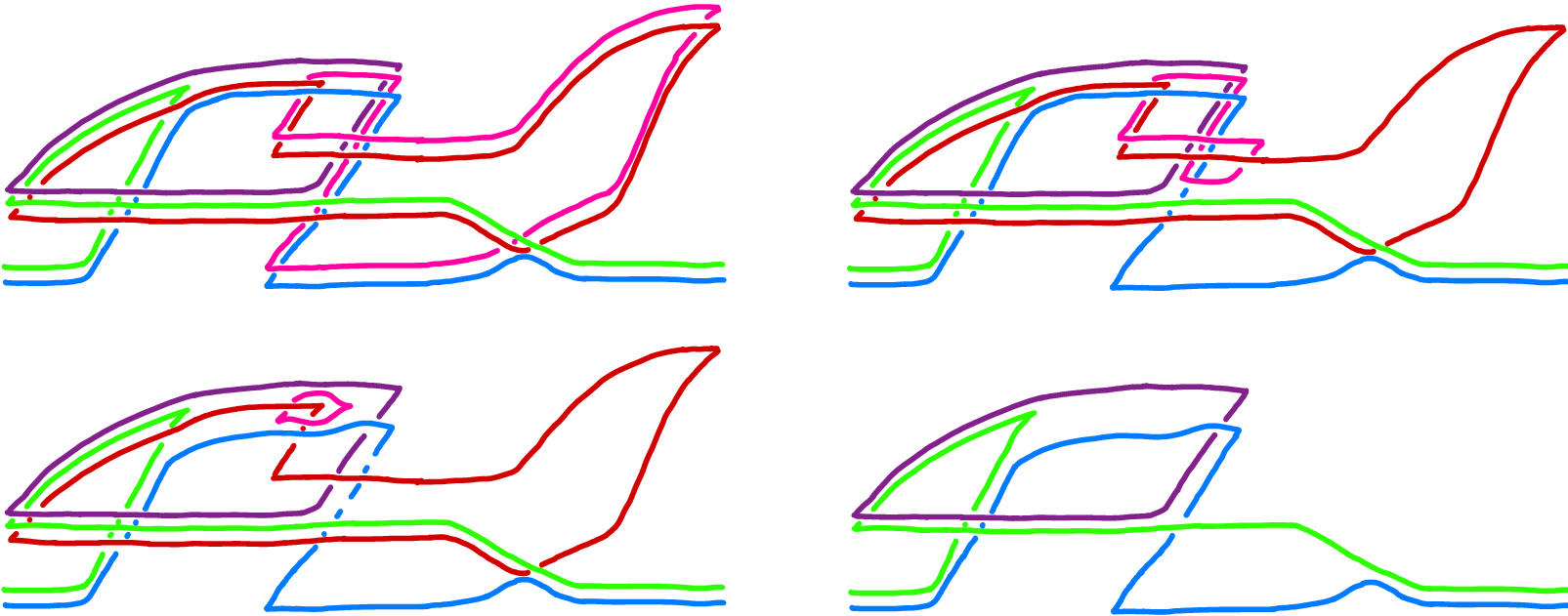}
        \put(-2.5,24.8){$-$}
        \put(-2.5,26.6){$+$}
        \put(25.5,33.8){$+$}
        \put(46.5,22){$+$}
        \put(46.5,21){$-$}

        \put(51.5,24.8){$-$}
        \put(51.5,26.6){$+$}
        \put(79.5,33.8){$+$}
        \put(100.2,22){$+$}
        \put(100.2,21){$-$}

        \put(-2.5,4){$-$}
        \put(-2.5,5.8){$+$}
        \put(22.5,12.8){$+$}
        \put(46.5,1.3){$+$}
        \put(46.5,0.3){$-$}

        \put(51.5,5.8){$+$}
        \put(100.2,1.3){$+$}
        \put(100.2,0.3){$-$}
    \end{overpic}
    \caption{Steps (9)--(12): Legendrian isotopies of the pink component, followed by a contact handle slide and the removal of a standard cancelling pair involving the red and pink components.}\label{fig:lantern_dest_proof3}
\end{figure}

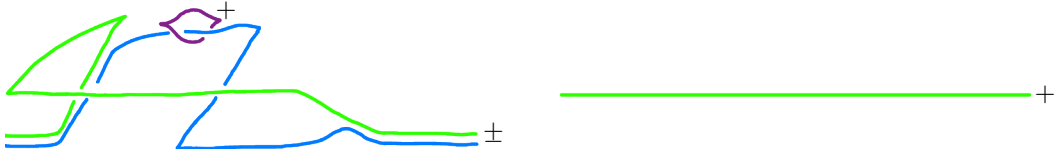
\begin{figure}[htbp]
    \centering
    \begin{overpic}[scale=1]{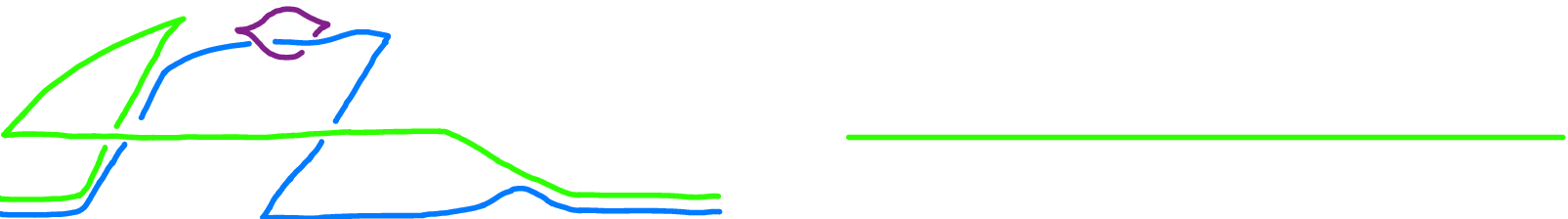}
        \put(20.5,12.8){$+$}
        \put(46.5,1){$+$}
        \put(46.5,0){$-$}

        \put(100.2,4.6){$+$}
    \end{overpic}
    \caption{Steps (13)--(14): a Legendrian isotopy of the purple component, followed by a contact handle slide and the removal of a standard cancelling pair involving the blue and purple components.}\label{fig:lantern_dest_proof4}
\end{figure}


\section{Legendrian knots in contact surgery diagrams}\label{sec:proofLeg_links_Kirby_thm}

Here we prove the contact version of Kirby's theorem for Legendrian links presented in contact surgery diagrams.

\begin{proof}[Proof of Theorem~\ref{prop:Leg_links_Kirby_thm}]
Let $J_i$ be a Legendrian link in a contact manifold $(M_i,\xi_i)$ presented by a Legendrian link in the complement of a surgery presentation $L_i$ of $(M_i,\xi_i)$, for $i=1,2$. The moves listed in Theorem~\ref{prop:Leg_links_Kirby_thm} preserve the equivalence type of the Legendrian links $J_i$. It remains to prove the converse statement, namely that any two equivalent Legendrian links $J_1$ and $J_2$ are related by the standard contact Kirby moves.

If the surgery presentations $L_1$ and $L_2$ coincide, this follows from a result of Ding and Geiges (see the last page of~\cite{Ding_Geiges_slides}): two Legendrian knots $J_1$ and $J_2$ in the complement of a surgery presentation $L$ are equivalent if and only if their front projections are related by a sequence of planar isotopies, Legendrian Reidemeister moves, and contact handle slides of the $J_i$ over components of $L$. (A handle slide corresponds to an isotopy through the belt sphere of a symplectic $4$-handle attached to the symplectization of the contact manifold, as discussed in~\cite{Ding_Geiges_slides}. From a 3-dimensional point of view, this is equivalent to taking a band connected sum with a $tb=-1$ unknot given by the meridional disk of the newly glued-in solid torus, as explained in~\cite{Casals_Etnyre_Kegel}. Thus one can show that after Legendrian Reidemeister moves these contact handle slides can always be expressed as standard contact handle slides.)

It remains to show that the standard contact Kirby moves are sufficient to transform $L_1$ into $L_2$. Ignoring the additional links $J_1$ and $J_2$, this is precisely the statement of Theorem~\ref{thm:main}. However, the links $J_1$ and $J_2$ may intersect the shaded surfaces appearing in Figures~\ref{fig:cancelling_intro}--\ref{fig:chain_intro}, so that the contact Kirby moves cannot be applied directly. We will show that, after performing suitable standard contact handle slides of the links $J_i$ over components of $L_i$, one may always assume that the links $J_i$ are disjoint from the shaded surfaces.

For cancelling pairs, this can always be achieved by an annulus twist, which can in turn be expressed as a sequence of standard contact handle slides, as discussed in~\cite{Casals_Etnyre_Kegel}. The cases of the standard lantern and standard chain moves are illustrated in Figures~\ref{fig:lantern_slide} and~\ref{fig:chain_slide_1}.

First, observe that it suffices to show that the links $J_i$ can be made disjoint from the shaded surface on the right-hand side of Figures~\ref{fig:lantern_intro} and~\ref{fig:chain_intro}, since curves outside the shaded surface remain unchanged and are therefore automatically disjoint from the corresponding shaded surface on the left-hand side.

We begin with the standard lantern move. On the left-hand side of Figure~\ref{fig:lantern_slide}, we depict a Legendrian link $J$ in black intersecting the shaded region transversely in two points. Up to isotopy and symmetry, the relevant intersections reduce to this configuration. On the right-hand side of the same figure, we have performed two standard contact handle slides of the black curve over the green and the red curve that remove these intersection points and make the link disjoint from the shaded region.

\begin{figure}[htbp]
    \centering
    \begin{overpic}[scale=1]{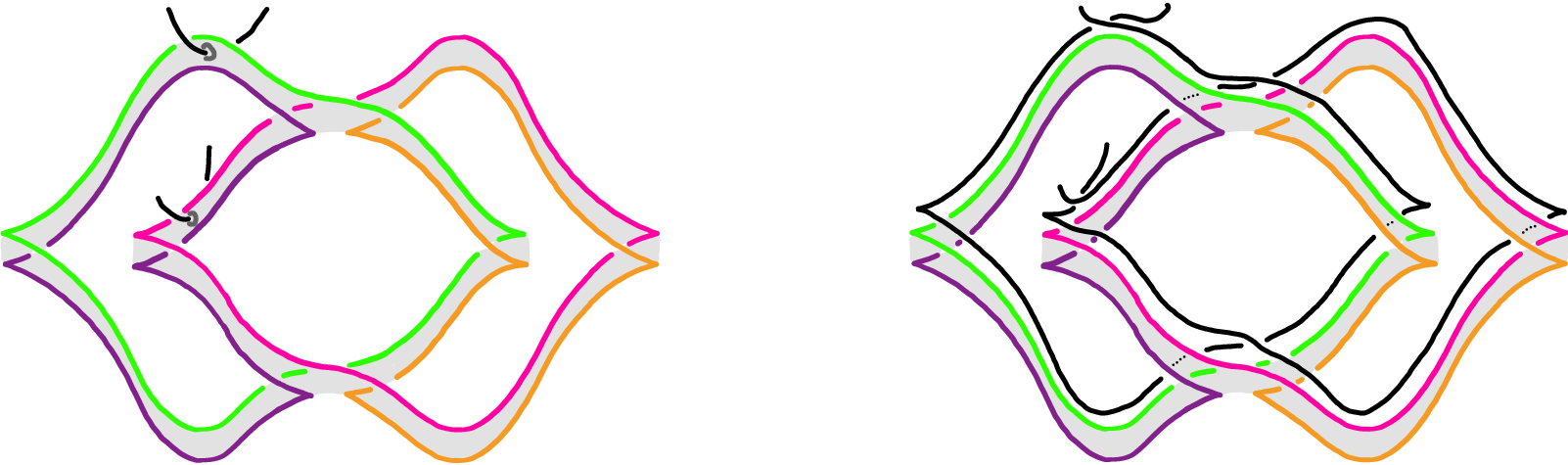}
        \put(-2.5,14){$-$}
        \put(-2.5,12){$-$}
        \put(43.5,14){$-$}
        \put(43.5,12){$-$}

        \put(55.5,14){$-$}
        \put(55.5,12){$-$}
        \put(101.5,14){$-$}
        \put(101.5,12){$-$}
    \end{overpic}
    \caption{Removing intersections of the black curve with the shaded region by handle slides.}\label{fig:lantern_slide}
\end{figure}
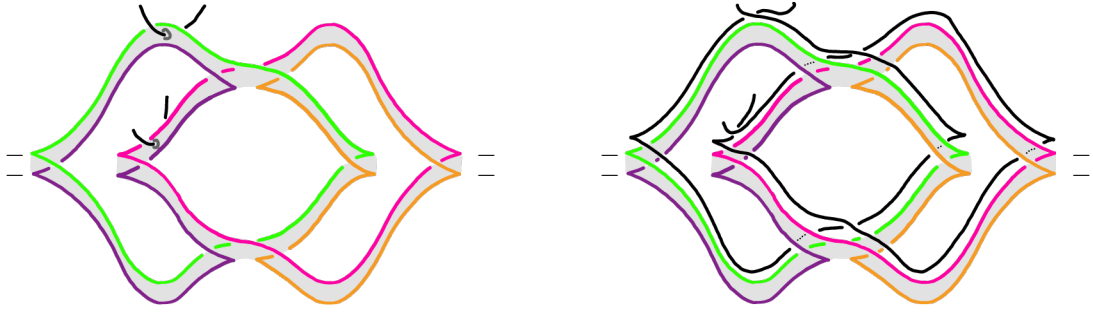

The case of the chain move is slightly more involved. In the middle of Figure~\ref{fig:chain_slide_1}, we show a Legendrian link $J$, drawn in black and red, intersecting the shaded surface transversely in two points. Again, by isotopy and symmetry, these are the configurations that need to be considered. We first treat the black intersection point. On the right-hand side of Figure~\ref{fig:chain_slide_1}, we have performed a standard contact handle slide of the black component, thereby removing the corresponding intersection point. On the left-hand side of the same figure, we show the image of the red component after a standard contact handle slide. The resulting red knot now intersects the shaded surface transversely in two points, as indicated there.

\begin{figure}[htbp]
    \centering
    \begin{overpic}[scale=1]{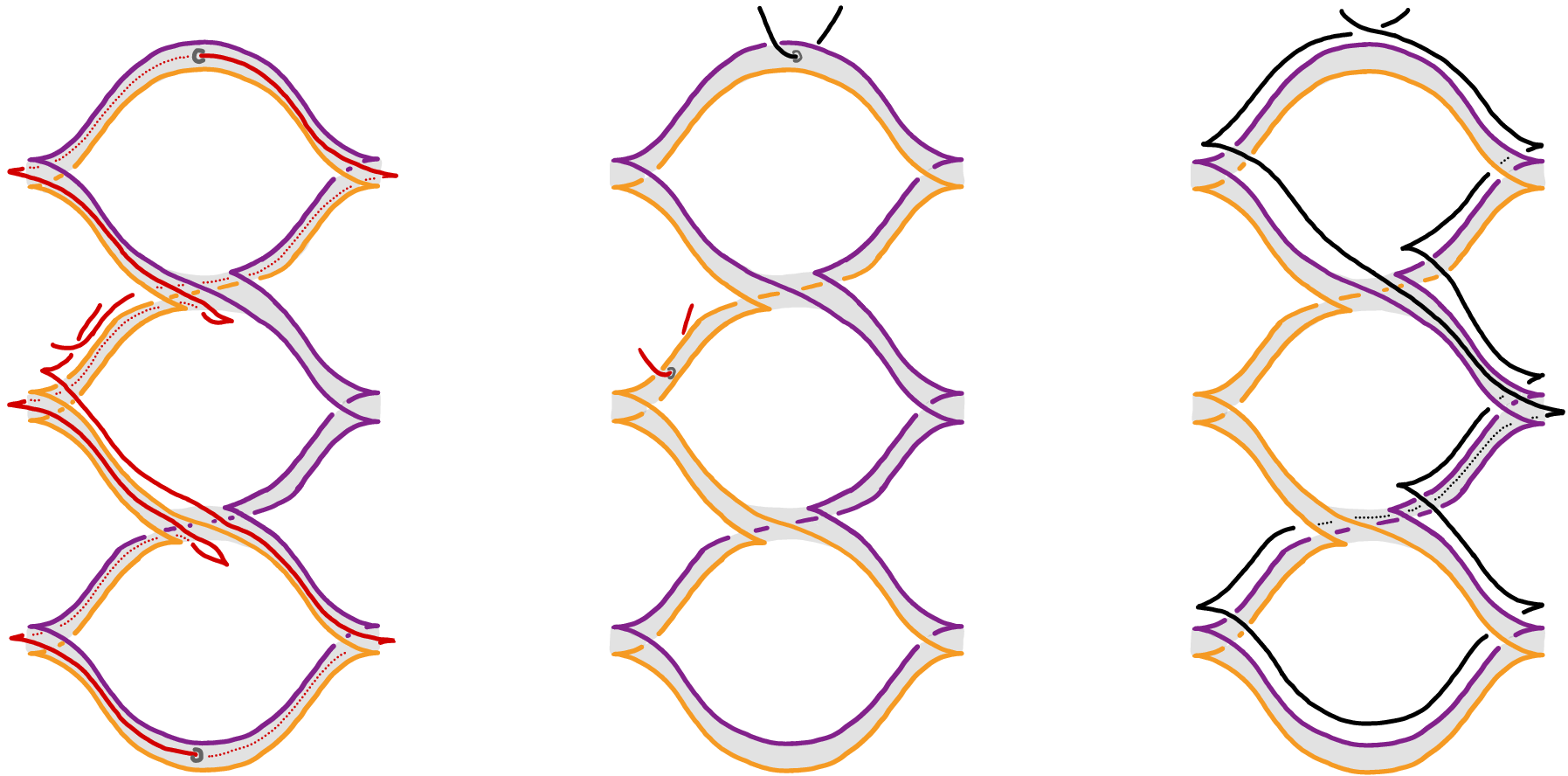}
        \put(25.5,39){$-$}
        \put(25.5,37){$-$}

        \put(62.5,39){$-$}
        \put(62.5,37){$-$}
        
        \put(100.5,39){$-$}
        \put(100.5,37){$-$}
    \end{overpic}
    \caption{Removing intersections of the black curve with the shaded region by handle slides.}\label{fig:chain_slide_1}
\end{figure}
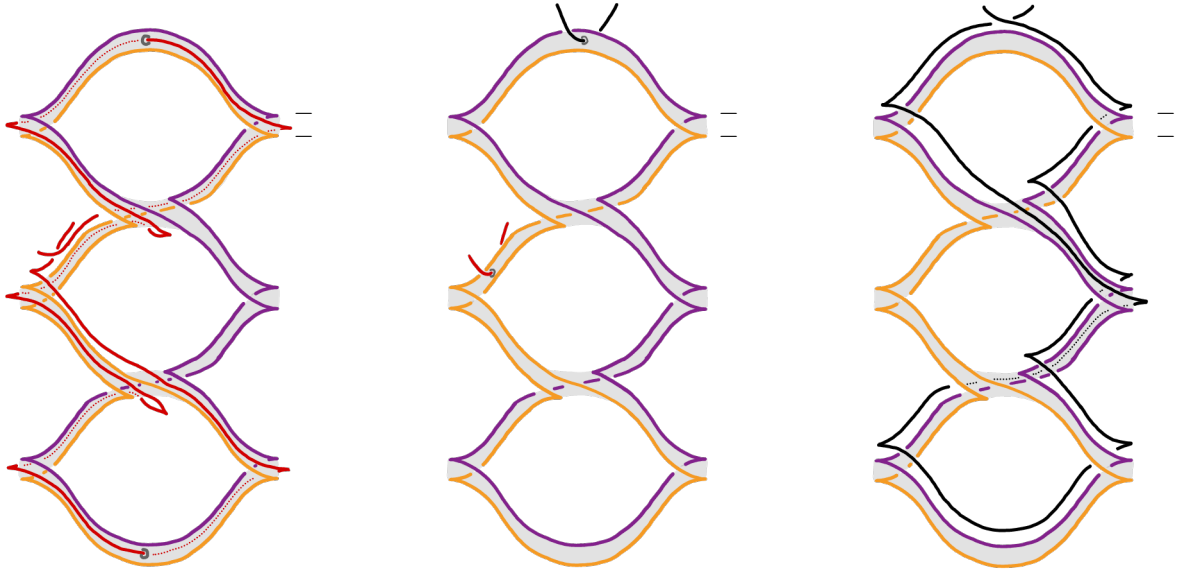

In Figure~\ref{fig:chain_slide_2}, we perform a sequence of Legendrian Reidemeister moves in the complement of the surgery link to show that these new intersections are of the same type as the original black intersection considered above. Consequently, two additional standard contact handle slides suffice to remove both remaining intersections of the red curve with the shaded region. This completes the proof.
\end{proof}

\begin{figure}[htbp]
    \centering
    \begin{overpic}[scale=1]{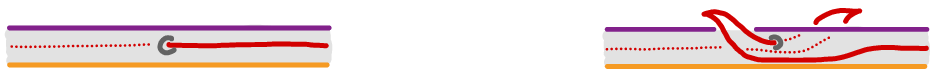}
    \end{overpic}
    \caption{An isotopy of the red curve in the complement of the orange and purple curves.}\label{fig:chain_slide_2}
\end{figure}


\section{Blow-ups and blow-downs in contact geometry}\label{sec:blow_up_discussion}

We now make precise the heuristic mentioned in Remark~\ref{rmk_weak_blowup}.
Let \(M=H_1\cup_h H_2\) be a Heegaard splitting with Heegaard surface
\(\Sigma=\partial H_1=\partial H_2\). If a mapping class \(g\in\MCG(\Sigma)\) extends over the handlebody \(H_1\), then replacing the gluing map \(h\) by \(h\circ g^{-1}\) gives the same \(3\)-manifold. Indeed, the extension of \(g\) gives a self-diffeomorphism of \(H_1\), so the change in the gluing can be absorbed into one of the two handlebodies.

In the surgery picture, this says that inserting into a surgery presentation a framed link on the Heegaard surface representing such a mapping class does not change the presented \(3\)-manifold. In this sense, a mapping class of the Heegaard surface which extends over a handlebody may be regarded as giving a weak version of a blow-up or blow-down. The usual blow-up/down corresponds to the special case of a Dehn twist along a meridional curve.

Now suppose that \((F,\phi)\) is an abstract open book supporting a contact manifold \((M,\xi)\), and set
\[
G=F\times[-\varepsilon,\varepsilon].
\]
The open book determines a contact Heegaard splitting as follows. Restricting the contact structure to \(G\) yields a contact handlebody, and its complement is a contact handlebody as well. Thus \(\partial G\) is a convex Heegaard surface for \((M,\xi)\).

In the contact setting, one has to be more careful than in the purely topological case. A diffeomorphism of \(G\) extending a boundary mapping class need not be a contactomorphism of the contact handlebody; it may change the contact structure on \(G\). Thus, a contact version of the above Heegaard-theoretic mechanism should at least require the boundary mapping class to preserve the dividing set on \(\partial G\), and extend over \(G\) as a contactomorphism of the contact handlebody. After forgetting the contact structure, any such extension is still an ordinary diffeomorphism of \(G\). Therefore, any boundary mapping class arising from this mechanism must lie in the image of the ordinary restriction map
\[
\rho_G\colon \MCG(G)\longrightarrow \MCG(\partial G).
\]

The contact surgery links relevant to us give a special class of boundary mapping classes. They come from the image of the map
\[
j\colon \MCG(F,\partial F)\longrightarrow \MCG(\partial G),
\]
defined as follows. If \(\varphi\in\mathrm{Diff}^+(F,\partial F)\), then \(j(\varphi)\) is represented by the diffeomorphism of \(\partial G\) which is equal to \(\varphi\) on \(F\times\{\varepsilon\}\), and equal to the identity on \(F\times\{-\varepsilon\}\) and on \(\partial F\times[-\varepsilon,\varepsilon]\). Since \(\varphi\) fixes \(\partial F\) pointwise, the mapping class \(j([\varphi])\) preserves the dividing set on the contact Heegaard surface \(\partial G\). Geometrically, \(j\) records the fact that page-supported surgery curves modify the monodromy on one page of the open book and do nothing on the rest of the Heegaard surface.

The following lemma shows that no nontrivial mapping class of this form can be absorbed into the handlebody, even after forgetting the contact structure.

\begin{lemma}\label{lem:trivial_intersection}
Let \(F\) be a compact oriented surface with nonempty boundary, let
\[
G=F\times[-\varepsilon,\varepsilon],
\]
and consider the maps
\[
j\colon \MCG(F,\partial F)\to \MCG(\partial G),
\qquad
\rho_G\colon \MCG(G)\to \MCG(\partial G).
\]
Then
\[
\operatorname{im}(j)\cap \operatorname{im}(\rho_G)=\{1\}.
\]
\end{lemma}

\begin{proof}
Suppose that \(j([\varphi])\in \operatorname{im}(\rho_G)\). After choosing representatives, we may assume that there is a diffeomorphism \(\Psi\colon G\to G\) whose restriction to \(\partial G\) is \(j(\varphi)\). Thus \(\Psi\) restricts to \(\varphi\) on \(F\times\{\varepsilon\}\), to the identity on \(F\times\{-\varepsilon\}\), and to the identity on \(\partial F\times[-\varepsilon,\varepsilon]\). Thus $\Psi$ defines a pseudo-isotopy relative to $\partial F$ from $\varphi$ to the identity $\id_F$. Since, for compact surfaces, any pseudo-isotopy relative to the boundary can be changed to an isotopy relative to the boundary, it follows that \([\varphi]=1\in\MCG(F,\partial F)\). Therefore \(j([\varphi])=1\), proving the claim.
\end{proof}

Consequently, no nontrivial page-supported contact surgery move can be absorbed into the contact handlebody by this mechanism. Thus, even in this weaker Heegaard-theoretic sense, the open-book setting leaves no room for a contact analogue of blow-up or blow-down.

\appendix


\section{A contact-geometric proof of Kirby's theorem}\label{sec:top_Kirby}

We conclude with a proof sketch of Kirby's theorem for smooth manifolds using contact geometry, via the Giroux correspondence and our main result.

\begin{theorem}
    Two framed links $L_1$ and $L_2$ in $S^3$ represent diffeomorphic $3$-manifolds if and only if $L_1$ can be transformed into $L_2$ by a sequence of handle slides, blow-ups, and blow-downs.
\end{theorem}

\begin{proof}[Proof sketch]
    It is straightforward to check that handle slides, blow-ups, and blow-downs do not change the diffeomorphism types of the surgered manifolds. We prove the converse. Let $L_1$ and $L_2$ be smooth surgery diagrams of diffeomorphic $3$-manifolds $M_1$ and $M_2$. First, choose Legendrian realizations $\boldsymbol{L}_i$ of the surgery links $L_i$ in $(S^3,\xist)$. Next, modify these Legendrian knots so that contact $(\pm1)$-surgery produces a contact structure $\xi_i$ on $M_i$. This can be done by stabilizing the Legendrian knots (to decrease the framing) or smoothly blowing up the surgery diagrams (to increase the framing). We refer to Figure~11 of~\cite{Ding_Geiges_Stipsicz} for the details. Next, we perform topological Kirby moves to make the Gompf invariants $\Gamma(\xi_i,\mathfrak{s})$ of the contact structures with respect to a fixed spin structure $\mathfrak{s}$ agree; see~\cite{Etnyre_Kegel_Onaran} for details. By performing a contact connected sum with a suitable overtwisted contact structure on $S^3$ (which can be done via topological Kirby moves), we can assume that $(M_1,\xi_1)$ and $(M_2,\xi_2)$ are homotopic as tangential $2$-plane fields~\cite{Gompf,Ding_Geiges_Stipsicz,Etnyre_Kegel_Onaran}. Since both contact structures are overtwisted, it follows from Eliashberg's classification of overtwisted contact structures~\cite{Eliashberg_OT} that $(M_1,\xi_1)$ and $(M_2,\xi_2)$ are contactomorphic. Thus, Theorem~\ref{thm:main} implies that we can transform the surgery descriptions of $(M_1,\xi_1)$ and $(M_2,\xi_2)$ into each other by a sequence of standard contact Kirby moves. To deduce the topological Kirby theorem, it remains to express the lantern and chain moves as a sequence of topological Kirby moves, which is an exercise in Kirby calculus.
\end{proof}


\printbibliography

@book{Gompf_Stipsicz,
  author    = {Gompf, Robert E. and Stipsicz, Andr\'as I.},
  title     = {{$4$}-manifolds and {K}irby calculus},
  series    = {Graduate Studies in Mathematics},
  volume    = {20},
  publisher = {American Mathematical Society},
  address   = {Providence, RI},
  year      = {1999},
  doi       = {10.1090/gsm/020}
}

@article {Lickorish_MCG,
    AUTHOR = {Lickorish, W. B. R.},
     TITLE = {A finite set of generators for the homeotopy group of a
              {$2$}-manifold},
   JOURNAL = {Proc. Cambridge Philos. Soc.},
  FJOURNAL = {Proceedings of the Cambridge Philosophical Society},
    VOLUME = {60},
      YEAR = {1964},
     PAGES = {769--778},
      ISSN = {0008-1981},
   MRCLASS = {54.75},
  MRNUMBER = {171269},
MRREVIEWER = {R.\ H.\ Fox},
       DOI = {10.1017/s030500410003824x},
       URL = {https://doi.org/10.1017/s030500410003824x},
}

@article{Chekanov,
  author  = {Chekanov, Yuri},
  title   = {Differential algebra of {L}egendrian links},
  journal = {Invent. Math.},
  volume  = {150},
  year    = {2002},
  number  = {3},
  pages   = {441--483},
  doi     = {10.1007/s002220200212}
}

@article{Lickorish,
  author  = {Lickorish, W. B. R.},
  title   = {A representation of orientable combinatorial {$3$}-manifolds},
  journal = {Ann. of Math. (2)},
  volume  = {76},
  year    = {1962},
  pages   = {531--540},
  doi     = {10.2307/1970373}
}

@article{Wallace,
  author  = {Wallace, Andrew H.},
  title   = {Modifications and cobounding manifolds},
  journal = {Canadian J. Math.},
  volume  = {12},
  year    = {1960},
  pages   = {503--528},
  doi     = {10.4153/CJM-1960-045-7}
}

@article{Kanda1998,
  author  = {Kanda, Yutaka},
  title   = {On the {T}hurston--{B}ennequin invariant of {L}egendrian knots and non-exactness of {B}ennequin's inequality},
  journal = {Invent. Math.},
  volume  = {133},
  year    = {1998},
  number  = {2},
  pages   = {227--242},
  doi     = {10.1007/s002220050245}
}

@article{Eliashberg_OT,
  author  = {Eliashberg, Y.},
  title   = {Classification of overtwisted contact structures on {$3$}-manifolds},
  journal = {Invent. Math.},
  volume  = {98},
  year    = {1989},
  number  = {3},
  pages   = {623--637},
  doi     = {10.1007/BF01393840}
}

@article{Gervais,
  author  = {Gervais, Sylvain},
  title   = {Presentation and central extensions of mapping class groups},
  journal = {Trans. Amer. Math. Soc.},
  volume  = {348},
  year    = {1996},
  number  = {8},
  pages   = {3097--3132},
  doi     = {10.1090/S0002-9947-96-01509-7}
}

@book{Geiges_book,
  author    = {Geiges, Hansj\"org},
  title     = {An introduction to contact topology},
  series    = {Cambridge Studies in Advanced Mathematics},
  volume    = {109},
  publisher = {Cambridge University Press},
  address   = {Cambridge},
  year      = {2008},
  doi       = {10.1017/CBO9780511611438}
}

@article{Giroux91,
  author  = {Giroux, Emmanuel},
  title   = {Convexit\'e en topologie de contact},
  journal = {Comment. Math. Helv.},
  volume  = {66},
  year    = {1991},
  number  = {4},
  pages   = {637--677},
  doi     = {10.1007/BF02566670}
}

@inproceedings{Giroux02,
  author    = {Giroux, Emmanuel},
  title     = {G\'eom\'etrie de contact: de la dimension trois vers les dimensions sup\'erieures},
  booktitle = {Proceedings of the {I}nternational {C}ongress of {M}athematicians, Vol. {II} ({B}eijing, 2002)},
  pages     = {405--414},
  publisher = {Higher Ed. Press},
  address   = {Beijing},
  year      = {2002}
}

@article{Etnyre_VHM11,
  author  = {Etnyre, John B. and Van Horn-Morris, Jeremy},
  title   = {Fibered transverse knots and the {B}ennequin bound},
  journal = {Int. Math. Res. Not. IMRN},
  year    = {2011},
  number  = {7},
  pages   = {1483--1509},
  doi     = {10.1093/imrn/rnq125}
}

@misc{Etnyre04,
  author = {Etnyre, John B.},
  title  = {Convex surfaces in contact geometry: class notes},
  year   = {2004},
  note    = {Available online at \\\url{https://etnyre.math.gatech.edu/preprints/papers/surfaces.pdf}}
}

@incollection{Etnyre06,
  author    = {Etnyre, John B.},
  title     = {Lectures on open book decompositions and contact structures},
  booktitle = {Floer homology, gauge theory, and low-dimensional topology},
  series    = {Clay Math. Proc.},
  volume    = {5},
  pages     = {103--141},
  publisher = {Amer. Math. Soc.},
  address   = {Providence, RI},
  year      = {2006}
}

@misc{BHH24,
  author        = {Breen, Joseph and Honda, Ko and Huang, Yang},
  title         = {The {G}iroux correspondence in arbitrary dimensions},
  year          = {2024},
  eprint        = {2307.02317},
  archivePrefix = {arXiv},
  primaryClass  = {math.SG},
  url           = {https://arxiv.org/abs/2307.02317}
}

@misc{LicVer1_24,
  author        = {Licata, Joan and V\'ertesi, Vera},
  title         = {Heegaard splittings and the tight {G}iroux correspondence},
  year          = {2024},
  eprint        = {2309.11828},
  archivePrefix = {arXiv},
  primaryClass  = {math.GT},
  url           = {https://arxiv.org/abs/2309.11828}
}

@book{FM12,
  author    = {Farb, Benson and Margalit, Dan},
  title     = {A primer on mapping class groups},
  series    = {Princeton Mathematical Series},
  volume    = {49},
  publisher = {Princeton University Press},
  address   = {Princeton, NJ},
  year      = {2012}
}

@book{Ozbagci_Stipsicz,
  author    = {Ozbagci, Burak and Stipsicz, Andr\'as I.},
  title     = {Surgery on contact 3-manifolds and {S}tein surfaces},
  series    = {Bolyai Society Mathematical Studies},
  volume    = {13},
  publisher = {Springer-Verlag; J\'anos Bolyai Mathematical Society},
  address   = {Berlin; Budapest},
  year      = {2004},
  doi       = {10.1007/978-3-662-10167-4}
}

@article{BM09,
  author  = {Baader, Sebastian and Ishikawa, Masaharu},
  title   = {Legendrian graphs and quasipositive diagrams},
  journal = {Ann. Fac. Sci. Toulouse Math. (6)},
  volume  = {18},
  year    = {2009},
  number  = {2},
  pages   = {285--305},
  url     = {http://afst.cedram.org/item?id=AFST_2009_6_18_2_285_0}
}

@article{Weinstein,
  author  = {Weinstein, Alan},
  title   = {Contact surgery and symplectic handlebodies},
  journal = {Hokkaido Math. J.},
  volume  = {20},
  year    = {1991},
  number  = {2},
  pages   = {241--251},
  doi     = {10.14492/hokmj/1381413841}
}

@article{Gay,
  author  = {Gay, David T.},
  title   = {Explicit concave fillings of contact three-manifolds},
  journal = {Math. Proc. Cambridge Philos. Soc.},
  volume  = {133},
  year    = {2002},
  number  = {3},
  pages   = {431--441},
  doi     = {10.1017/S0305004102006291}
}

@article{Eliashberg_handle,
  author  = {Eliashberg, Yakov},
  title   = {Topological characterization of {S}tein manifolds of dimension {$>2$}},
  journal = {Internat. J. Math.},
  volume  = {1},
  year    = {1990},
  number  = {1},
  pages   = {29--46},
  doi     = {10.1142/S0129167X90000034}
}

@article{Gompf,
  author  = {Gompf, Robert E.},
  title   = {Handlebody construction of {S}tein surfaces},
  journal = {Ann. of Math. (2)},
  volume  = {148},
  year    = {1998},
  number  = {2},
  pages   = {619--693},
  doi     = {10.2307/121005}
}

@article{Lisac_StipsiczI,
  author  = {Lisca, Paolo and Stipsicz, Andr\'as I.},
  title   = {Ozsv\'ath--{S}zab\'o invariants and tight contact three-manifolds. {I}},
  journal = {Geom. Topol.},
  volume  = {8},
  year    = {2004},
  pages   = {925--945},
  doi     = {10.2140/gt.2004.8.925}
}

@article{Lisac_StipsiczII,
  author  = {Lisca, Paolo and Stipsicz, Andr\'as I.},
  title   = {Ozsv\'ath--{S}zab\'o invariants and tight contact three-manifolds. {II}},
  journal = {J. Differential Geom.},
  volume  = {75},
  year    = {2007},
  number  = {1},
  pages   = {109--141},
  url     = {http://projecteuclid.org/euclid.jdg/1175266255}
}

@article{Lisac_StipsiczIII,
  author  = {Lisca, Paolo and Stipsicz, Andr\'as I.},
  title   = {Ozsv\'ath--{S}zab\'o invariants and tight contact three-manifolds. {III}},
  journal = {J. Symplectic Geom.},
  volume  = {5},
  year    = {2007},
  number  = {4},
  pages   = {357--384},
  url     = {http://projecteuclid.org/euclid.jsg/1213883789}
}

@article{Wand,
  author  = {Wand, Andy},
  title   = {Tightness is preserved by {L}egendrian surgery},
  journal = {Ann. of Math. (2)},
  volume  = {182},
  year    = {2015},
  number  = {2},
  pages   = {723--738},
  doi     = {10.4007/annals.2015.182.2.8}
}

@article{Geiges_Onaran,
  author  = {Geiges, Hansj\"org and Onaran, Sinem},
  title   = {Legendrian rational unknots in lens spaces},
  journal = {J. Symplectic Geom.},
  volume  = {13},
  year    = {2015},
  number  = {1},
  pages   = {17--50},
  doi     = {10.4310/JSG.2015.v13.n1.a2}
}

@article{BEE,
  author  = {Bourgeois, Fr\'ed\'eric and Ekholm, Tobias and Eliashberg, Yasha},
  title   = {Effect of {L}egendrian surgery},
  note    = {With an appendix by Sheel Ganatra and Maksim Maydanskiy},
  journal = {Geom. Topol.},
  volume  = {16},
  year    = {2012},
  number  = {1},
  pages   = {301--389},
  doi     = {10.2140/gt.2012.16.301}
}

@article{Avdek_surgery,
  author  = {Avdek, Russell},
  title   = {Combinatorial {R}eeb dynamics on punctured contact 3-manifolds},
  journal = {Geom. Topol.},
  volume  = {27},
  year    = {2023},
  number  = {3},
  pages   = {953--1082},
  doi     = {10.2140/gt.2023.27.953}
}

@article{Ding_Geiges_slides,
  author  = {Ding, Fan and Geiges, Hansj\"org},
  title   = {Handle moves in contact surgery diagrams},
  journal = {J. Topol.},
  volume  = {2},
  year    = {2009},
  number  = {1},
  pages   = {105--122},
  doi     = {10.1112/jtopol/jtp002}
}

@misc{SnapPy,
  author = {Culler, M. and Dunfield, N. and Goerner, M. and Weeks, J.},
  title  = {{SnapPy}, a computer program for studying the geometry and topology of $3$-manifolds},
  note    = {Available online at \url{http://snappy.computop.org}}
}

@article{Lu_Kirbystheorem,
  author  = {Lu, Ning},
  title   = {A simple proof of the fundamental theorem of {K}irby calculus on links},
  journal = {Trans. Amer. Math. Soc.},
  volume  = {331},
  year    = {1992},
  number  = {1},
  pages   = {143--156},
  doi     = {10.2307/2154000}
}

@article{Matveev_Polyak_Kirbytheorem,
  author  = {Matveev, S. and Polyak, M.},
  title   = {A geometrical presentation of the surface mapping class group and surgery},
  journal = {Comm. Math. Phys.},
  volume  = {160},
  year    = {1994},
  number  = {3},
  pages   = {537--550},
  url     = {http://projecteuclid.org/euclid.cmp/1104269709}
}

@article{Ghiggini08,
  author  = {Ghiggini, Paolo},
  title   = {On tight contact structures with negative maximal twisting number on small {S}eifert manifolds},
  journal = {Algebr. Geom. Topol.},
  volume  = {8},
  year    = {2008},
  number  = {1},
  pages   = {381--396},
  doi     = {10.2140/agt.2008.8.381}
}

@article{Lisca_Matic04,
  author  = {Lisca, Paolo and Mati\'c, Gordana},
  title   = {Transverse contact structures on {S}eifert 3-manifolds},
  journal = {Algebr. Geom. Topol.},
  volume  = {4},
  year    = {2004},
  pages   = {1125--1144},
  doi     = {10.2140/agt.2004.4.1125}
}

@misc{regina,
  author = {Burton, B. and Budney, R. and Pettersson, W. and others},
  title  = {Regina: Software for low-dimensional topology},
  note    = {Available online at \url{http://regina-normal.github.io}}
}

@article{Ding_Geiges,
  author  = {Ding, Fan and Geiges, Hansj\"org},
  title   = {A {L}egendrian surgery presentation of contact 3-manifolds},
  journal = {Math. Proc. Cambridge Philos. Soc.},
  volume  = {136},
  year    = {2004},
  number  = {3},
  pages   = {583--598},
  doi     = {10.1017/S0305004103007412}
}

@article{Ding_Geiges_Stipsicz,
  author  = {Ding, Fan and Geiges, Hansj\"org and Stipsicz, Andr\'as I.},
  title   = {Surgery diagrams for contact 3-manifolds},
  journal = {Turkish J. Math.},
  volume  = {28},
  year    = {2004},
  number  = {1},
  pages   = {41--74}
}

@article{Kegel_Legendrian_complement,
  author  = {Kegel, Marc},
  title   = {The {L}egendrian knot complement problem},
  journal = {J. Knot Theory Ramifications},
  volume  = {27},
  year    = {2018},
  number  = {14},
  pages   = {1850067, 36},
  doi     = {10.1142/S0218216518500670}
}

@article{Kegel_Onaran,
  author  = {Kegel, Marc and Onaran, Sinem},
  title   = {Contact surgery graphs},
  journal = {Bull. Aust. Math. Soc.},
  volume  = {107},
  year    = {2023},
  number  = {1},
  pages   = {146--157},
  doi     = {10.1017/S0004972722000375}
}

@article{Legendrian_Reidemeister,
  author  = {Swiatkowski, Jacek},
  title   = {On the isotopy of {L}egendrian knots},
  journal = {Ann. Global Anal. Geom.},
  volume  = {10},
  year    = {1992},
  number  = {3},
  pages   = {195--207},
  doi     = {10.1007/BF00136863}
}

@phdthesis{Kegel_thesis,
  author = {Kegel, Marc},
  title  = {Legendrian knots in surgery diagrams and the knot complement problem},
  school = {Universit\"at zu K\"oln},
  year   = {2017}
}

@article{Lisca_Stipsicz_lantern,
  author  = {Lisca, Paolo and Stipsicz, Andr\'as I.},
  title   = {Contact surgery and transverse invariants},
  journal = {J. Topol.},
  volume  = {4},
  year    = {2011},
  number  = {4},
  pages   = {817--834},
  doi     = {10.1112/jtopol/jtr022}
}

@article{Etnyre_Kegel_Onaran,
  author  = {Etnyre, John and Kegel, Marc and Onaran, Sinem},
  title   = {Contact surgery numbers},
  journal = {J. Symplectic Geom.},
  volume  = {21},
  year    = {2023},
  number  = {6},
  pages   = {1255--1333}
}

@article{Casals_Etnyre_Kegel,
  author  = {Casals, Roger and Etnyre, John and Kegel, Marc},
  title   = {Stein traces and characterizing slopes},
  journal = {Math. Ann.},
  volume  = {389},
  year    = {2024},
  number  = {2},
  pages   = {1053--1098},
  doi     = {10.1007/s00208-023-02662-2}
}

@article{Durst_Kegel_OB,
  author  = {Durst, Sebastian and Kegel, Marc},
  title   = {Computing rotation numbers in open books},
  journal = {J. G\"okova Geom. Topol. GGT},
  volume  = {12},
  year    = {2018},
  pages   = {71--92}
}

@article{Honda_classification,
  author  = {Honda, Ko},
  title   = {On the classification of tight contact structures. {I}},
  journal = {Geom. Topol.},
  volume  = {4},
  year    = {2000},
  pages   = {309--368},
  doi     = {10.2140/gt.2000.4.309}
}

@misc{LicVer2_24,
  author        = {Licata, Joan and Scharitzer, Matthias and V\'ertesi, Vera},
  title         = {The Giroux Correspondence in dimension 3},
  year          = {2026},
  eprint        = {2408.01079},
  archivePrefix = {arXiv},
  primaryClass  = {math.GT},
  url           = {https://arxiv.org/abs/2408.01079}
}

@article{Kirby,
  author  = {Kirby, Robion},
  title   = {A calculus for framed links in {$S^3$}},
  journal = {Invent. Math.},
  volume  = {45},
  year    = {1978},
  number  = {1},
  pages   = {35--56},
  doi     = {10.1007/BF01406222}
}

@article{Ding_Geiges_fillable,
  author  = {Ding, Fan and Geiges, Hansj\"org},
  title   = {Symplectic fillability of tight contact structures on torus bundles},
  journal = {Algebr. Geom. Topol.},
  volume  = {1},
  year    = {2001},
  pages   = {153--172},
  doi     = {10.2140/agt.2001.1.153}
}

@misc{StenhedeAlgorithm,
  author        = {Eric Stenhede},
  title         = {An algorithm to Legendrian realize a curve on a ribbon surface},
  year          = {2026},
  eprint        = {2604.08010},
  archivePrefix = {arXiv},
  primaryClass  = {math.GT},
  url           = {https://arxiv.org/abs/2604.08010}
}

@article{Avdek13,
  author  = {Avdek, Russell},
  title   = {Contact surgery and supporting open books},
  journal = {Algebr. Geom. Topol.},
  volume  = {13},
  year    = {2013},
  number  = {3},
  pages   = {1613--1660},
  doi     = {10.2140/agt.2013.13.1613}
}

@article{Durst_Kegel,
  author  = {Durst, S. and Kegel, M.},
  title   = {Computing rotation and self-linking numbers in contact surgery diagrams},
  journal = {Acta Math. Hungar.},
  volume  = {150},
  year    = {2016},
  number  = {2},
  pages   = {524--540},
  doi     = {10.1007/s10474-016-0660-8}
}

@article{contact_class_algo,
  author  = {Plamenevskaya, Olga},
  title   = {A combinatorial description of the {H}eegaard {F}loer contact invariant},
  journal = {Algebr. Geom. Topol.},
  volume  = {7},
  year    = {2007},
  pages   = {1201--1209},
  doi     = {10.2140/agt.2007.7.1201}
}

@article{contact_class,
  author  = {Ozsv\'ath, Peter and Szab\'o, Zolt\'an},
  title   = {Heegaard {F}loer homology and contact structures},
  journal = {Duke Math. J.},
  volume  = {129},
  year    = {2005},
  number  = {1},
  pages   = {39--61},
  doi     = {10.1215/S0012-7094-04-12912-4}
}

@incollection{ECH,
  author    = {Hutchings, Michael},
  title     = {Lecture notes on embedded contact homology},
  booktitle = {Contact and symplectic topology},
  series    = {Bolyai Soc. Math. Stud.},
  volume    = {26},
  pages     = {389--484},
  publisher = {J\'anos Bolyai Math. Soc.},
  address   = {Budapest},
  year      = {2014},
  doi       = {10.1007/978-3-319-02036-5_9}
}

@article{symplectic_homology,
  author  = {Viterbo, C.},
  title   = {Functors and computations in {F}loer homology with applications. {I}},
  journal = {Geom. Funct. Anal.},
  volume  = {9},
  year    = {1999},
  number  = {5},
  pages   = {985--1033},
  doi     = {10.1007/s000390050106}
}

@incollection{Contact_homology,
  author    = {Eliashberg, Y. and Givental, A. and Hofer, H.},
  title     = {Introduction to Symplectic Field Theory},
  booktitle = {Visions in Mathematics: {GAFA} 2000 Special Volume, Part II},
  pages     = {560--673},
  publisher = {Birkh\"auser Basel},
  address   = {Basel},
  year      = {2010},
  doi       = {10.1007/978-3-0346-0425-3_4},
  isbn      = {978-3-0346-0425-3}
}

\end{document}